\documentclass[a4paper]{amsart}

\usepackage[utf8]{inputenc}
\usepackage[T1]{fontenc}
\usepackage{float}
\usepackage{amsthm}
\usepackage{xspace}
\usepackage{amssymb,amsmath,stmaryrd}
\usepackage[english]{babel}
\usepackage[shortlabels]{enumitem}
\usepackage[all]{xy}
\usepackage{caption}
\usepackage{subcaption}
\usepackage{wrapfig}
\usepackage{xcolor}
\usepackage{xr-hyper}
\usepackage[textwidth=16cm,margin=2.5cm]{geometry}
\usepackage{hyperref}

\theoremstyle{plain}
\newtheorem{theorem}{Theorem}[section]
\newtheorem{remark}[theorem]{Remark}
\newtheorem{lemma}[theorem]{Lemma}
\newtheorem{corollary}[theorem]{Corollary}
\newtheorem{question}[theorem]{Question}
\newtheorem{proposition}[theorem]{Proposition}

\theoremstyle{definition}
\newtheorem{assumption}[theorem]{Assumption}
\newtheorem{definition}[theorem]{Definition}
\newtheorem{notation}[theorem]{Notation}
\newtheorem{example}[theorem]{Example}
\numberwithin{equation}{section}

\newcommand{\ca}{\mbox{$C\sp*$-}al\-ge\-bra\xspace}
\newcommand{\cas}{\mbox{$C\sp*$-}al\-ge\-bras\xspace}
\newcommand{\starhomo}{\mbox{$\sp*$-}ho\-mo\-morphism\xspace}

\newcommand{\stariso}{\mbox{$\sp*$-}iso\-morphism\xspace}
\newcommand{\starisos}{\mbox{$\sp*$-}iso\-morphisms\xspace}
\newcommand{\fct}[3]{\ensuremath{#1\colon #2\rightarrow #3}\xspace}

\newcommand{\Z}{\ensuremath{\mathbb{Z}}\xspace}
\newcommand{\N}{\ensuremath{\mathbb{N}}\xspace}
\newcommand{\C}{\ensuremath{\mathbb{C}}\xspace}

\newcommand{\Q}{\ensuremath{\mathbb{Q}}\xspace}
\newcommand{\K}{\ensuremath{\mathbb{K}}\xspace}
\newcommand{\ie}{\emph{i.e.}\xspace}
\newcommand{\cf}{\emph{cf.}\xspace}
\newcommand{\eg}{\emph{e.g.}\xspace}
\newcommand{\id}{\ensuremath{\operatorname{id}}\xspace}

\newcommand{\cok}{\operatorname{cok}}

\newcommand{\noshow}[1]{}
\newcommand{\kk}{\ensuremath{\operatorname{KK}}\xspace}
\newcommand{\uline}[1]{\underline{#1}}
\newcommand{\uuline}[1]{\underline{\underline{#1}}}
\newcommand{\FKRplus}{\ensuremath{\operatorname{FK}^+_\mathcal{R}}\xspace}
\newcommand{\FKR}{\ensuremath{\operatorname{FK}_\mathcal{R}}\xspace}
\newcommand{\FKRs}{\ensuremath{\operatorname{FK}_\mathcal{R}^s}\xspace}

\newcommand{\A}{\ensuremath{\mathfrak{A}}\xspace}
\newcommand{\B}{\ensuremath{\mathfrak{B}}\xspace}
\newcommand{\Asf}{\mathsf{A}}
\newcommand{\Bsf}{\mathsf{B}}

\newcommand{\Esf}{\mathsf{E}}
\newcommand{\Prim}{\operatorname{Prim}}
\newcommand{\Prime}{\operatorname{Prime}}
\newcommand{\calP}{\ensuremath{\mathcal{P}}\xspace}
\newcommand{\GLZ}[1][n]{\ensuremath{\operatorname{GL}(#1,\Z)}\xspace}
\newcommand{\GL}{\ensuremath{\operatorname{GL}}\xspace}
\newcommand{\SLZ}[1][n]{\ensuremath{\operatorname{SL}(#1,\Z)}\xspace}
\newcommand{\SL}{\ensuremath{\operatorname{SL}}\xspace}
\newcommand{\GLPZ}[1][\mathbf{n}]{\ensuremath{\operatorname{GL}_\calP(#1,\Z)}\xspace}
\newcommand{\GLP}{\ensuremath{\operatorname{GL}_\calP}\xspace}
\newcommand{\SLPZ}[1][\mathbf{n}]{\ensuremath{\operatorname{SL}_\calP(#1,\Z)}\xspace}
\newcommand{\SLP}{\ensuremath{\operatorname{SL}_\calP}\xspace}

\newcommand{\MZ}[1][\mathbf{m}\times\mathbf{n}]{\ensuremath{\mathfrak{M}(#1,\Z)}\xspace}
\newcommand{\MPZ}[1][\mathbf{m}\times\mathbf{n}]{\ensuremath{\mathfrak{M}_\calP(#1,\Z)}\xspace}
\newcommand{\MPZc}[1][\mathbf{m}\times\mathbf{n}]{\ensuremath{\mathfrak{M}^\circ_\calP(#1,\Z)}\xspace}
\newcommand{\MPZcc}[1][\mathbf{m}\times\mathbf{n}]{\ensuremath{\mathfrak{M}^{\circ\circ}_\calP(#1,\Z)}\xspace}
\newcommand{\MPZccc}[1][\mathbf{m}\times\mathbf{n}]{\ensuremath{\mathfrak{M}^{\circ\circ\circ}_\calP(#1,\Z)}\xspace}
\newcommand{\Mplus}[1][m\times n]{\ensuremath{\mathfrak{M}^+(#1,\Z)}\xspace}
\newcommand{\MPplusZ}[1][\mathbf{m}\times\mathbf{n}]{\ensuremath{\mathfrak{M}^+_\calP(#1,\Z)}\xspace}
\newcommand{\ftn}[3]{ #1 \colon #2 \rightarrow #3 }
\newcommand{\setof}[2]{\left\{ #1 \;\middle|\; #2 \right\}}
\newcommand{\GLPE}{\GLP-equivalent\xspace}
\newcommand{\SLPE}{\SLP-equivalent\xspace}
\newcommand{\GLPEe}{\GLP-equivalence\xspace}
\newcommand{\SLPEe}{\SLP-equivalence\xspace}
\newcommand{\Pul}{Pulelehua\xspace}
\newcommand{\Meq}{\ensuremath{\sim_{M\negthinspace E}}\xspace}
\newcommand{\MCeq}{\ensuremath{\sim_{C\negthinspace E}}\xspace}
\newcommand{\MCPeq}{\ensuremath{\sim_{P\negthinspace E}}\xspace}
\newcommand{\OO}{\mbox{\texttt{\textup{(O)}}}\xspace}
\newcommand{\II}{\mbox{\texttt{\textup{(I)}}}\xspace}
\newcommand{\RR}{\mbox{\texttt{\textup{(R)}}}\xspace}
\newcommand{\RRplus}{\mbox{\texttt{\textup{(R+)}}}\xspace}
\newcommand{\SSS}{\mbox{\texttt{\textup{(S)}}}\xspace}
\newcommand{\CC}{\mbox{\texttt{\textup{(C)}}}\xspace}
\newcommand{\PP}{\mbox{\texttt{\textup{(P)}}}\xspace}

\newcommand{\CO}{\mbox{\texttt{\textup{(Col)}}}\xspace}
\newcommand{\mm}{\underline{m}}
\newcommand{\oneone}{\underline{1}}

\newenvironment{smallpmatrix}{\left(\begin{smallmatrix}}{\end{smallmatrix}\right)}

\newcommand{\first}[1]{}%{\color{blue}\textbf{First:} #1 }}
\newcommand{\third}[1]{}
\newcommand{\toke}[1]{}

\newcommand{\skippable}{\color{black}}
\newcommand{\nomoreskippable}{\color{black}}

\makeatletter
\def\citefirst{\@ifnextchar[{\@with}{\@without}}
\def\@with[#1]{\cite[#1]{arXiv:1604.05439v2}\xspace}
\def\@without{\cite{arXiv:1604.05439v2}\xspace}
\makeatother

\title[The complete classification of unital graph $C^*$-algebras]{The complete classification of unital graph $C^*$-algebras: Geometric and strong}
\date{\today}
\author{S\o{}ren Eilers}
\address{Department of Mathe\-matical Sciences, University of Copen\-hagen, Universi\-tets\-park\-en~5, DK-2100 Copen\-hagen, Den\-mark}
\email{eilers@math.ku.dk}
\author{Gunnar Restorff}
\address{Department of Science and Technology, University of the Faroe Islands, N\'{o}at\'{u}n~3, FO-100 T\'{o}rshavn, the Faroe Islands}
\email{gunnarr@setur.fo}
\author{Efren Ruiz}
\address{Department of Mathematics, University of Hawaii, Hilo, 200 W.~Kawili St., Hilo, Hawaii, 96720-4091 USA}
\email{ruize@hawaii.edu}
\author{Adam P.~W.~S\o{}rensen}
\address{Department of Mathematics, University of Oslo, PO BOX 1053 Blindern, N-0316 Oslo, Norway}
\email{apws@math.uio.no}

\begin{document}

\begin{abstract}
We provide a complete classification of the class of  unital graph \cas\ --- prominently containing the full family of Cuntz-Krieger algebras --- showing that Morita equivalence in this case is determined 
by ordered, filtered $K$-theory. 
The classification result is \emph{geometric} in the sense that it establishes that any Morita equivalence between $C^*(E)$ and $C^*(F)$ in this class can be realized by a sequence of moves leading from $E$ to $F$, in a way resembling the role of Reidemeister moves on knots.  As a key ingredient, we introduce a new class of such moves, establish that they leave the graph algebras invariant, and prove that after this augmentation, the list of moves becomes complete in the sense described above.

Along the way, we prove that every  (reduced, filtered) $K$-theory order isomorphism can be lifted to an isomorphism between the stabilized \cas{} --- and, as a consequence, that every such order isomorphism preserving the class of the unit comes from a \stariso between the unital graph $C^*$-algebras themselves.

It follows that the question of Morita equivalence   and \stariso 
amongst unital graph $C^*$-algebras is a decidable one.
As immediate examples of applications of our results we revisit the classification problem for quantum lens spaces and verify, in the unital case, the Abrams-Tomforde conjectures.
\end{abstract}

\maketitle

\section{Introduction}
\subsection{Background}

The ambition to classify graph $C^*$-algebras --- and their precursors, the Cuntz-Krieger algebras --- goes back to the inception of the field. In fact, as recently highlighted in \cite{arXiv:1511.01193v3}, Enomoto, Fujii and Watatani already in \cite{MR634920} showed by \emph{ad hoc} methods how to classify all the simple Cuntz-Krieger algebras that may be described by a $3\times 3$-matrix, predating by more than a decade the classification result of R\o rdam (\cite{MR1340839}) covering all simple Cuntz-Krieger algebras.

With the classification problem resolved for all simple graph $C^*$-algebras  around the turn of the millennium by the superposition of \cite{MR0397420} and \cite{MR1780426} due to the realization that any such $C^*$-algebra is either AF or purely infinite, the quest for classification of graph $C^*$-algebras moved into the realm of non-simple classification, where it has been a key problem  ever since.
The endeavour has evolved in parallel with the gradual realization of what invariants may prove to be complete in the case when the number of (gauge invariant) ideals is finite and the \cas in question are not stably finite. In this sense, the fundamental results obtained on the classification of certain classes of graph \cas are playing a role parallel to the one played by \cite{MR1340839}  as a catalyst for the Kirchberg-Phillips classification in \cite{MR1780426}.
 
Based on ideas pioneered by Huang (\cite{MR1301504}), the first two general results on the classification problem for nonsimple graph \cas were obtained by R\o rdam in \cite{MR1446202} and by Restorff in \cite{MR2270572} by very different methods. R\o rdam showed the importance of involving the full data contained in six-term exact sequences of the \cas given and proved a very complete classification theorem while restricting the ideal lattice to be as small as possible: only one nontrivial ideal. In Restorff's work, the ideal lattice was arbitrary among the finite ideal lattices, but as his method was to reduce the problem to classification of shifts of finite type and appeal to deep results by Boyle and Huang  from symbolic dynamics (\cite{MR1907894}, \cite{MR1990568}), only graph \cas in the Cuntz-Krieger class were covered.

Subsequent progress has mainly followed the approach in \cite{MR1446202} (see \cite{arXiv:1101.5702v3}, \cite{MR2563693}, \cite{MR3056712}), and hence applies only to restricted kinds of ideal lattices but with few further restrictions on the nature of the underlying graphs. The case of purely infinite graph \cas with finitely many ideals has been resolved (very interestingly, by a different invariant than what we use here), in recent work by Bentmann and Meyer (\cite{MR3628788}), but as summarized in \cite{MR3142033} there is not at present sufficient technology to take this approach much farther in the mixed cases than to \cas with three or four primitive ideals. 

In the paper at hand we close the  classification problem for all {unital} graph \cas, which includes \emph{all} Cuntz-Krieger algebras and thereby answer all the open questions in the addenda of \cite{MR2270572}; the proof strategy follows the  strategy from \cite{MR2270572} as generalized by the authors in various constellations over a period of 5 years (\cite{MR3082546}, \cite{MR2666426}, \cite{MR3047630}, \cite{MR3759003}).

\subsection{Approach}

Our method of proof, a substantial elaboration of key ideas from the authors' earlier work along with key ideas from symbolic dynamics (\cite{MR1990568}, \cite{MR1907894}), leads to a  \emph{geometric} classification, allowing us to conclude from a Morita equivalence between a pair of graph \cas $C^*(E)$ and $C^*(F)$ that a sequence  of  basic moves on the graphs will lead from $E$ to $F$ in a way resembling the role of Reidemeister moves on knots. The standard list of such moves, defining the notion of \emph{Cuntz move equivalence}, contains  moves that are closely related to those defining flow equivalence for shift spaces, along with the so-called \emph{Cuntz splice} (Move~\CC) which was introduced in \cite{MR866492}.

It is a key corollary of the main result of the present paper that any stable isomorphism between two unital graph $C^*$-algebras \emph{of real rank zero} can be realized by a Cuntz move equivalence, but as demonstrated in  \cite{MR3759003}, this is not generally the case. Indeed,
consider the graphs $E$ and $F$ from Figure~\ref{fig:counterex-graph1} and \ref{fig:counterex-graph2}, respectively. 
\begin{figure}[h]
\quad    \begin{subfigure}[b]{0.4\textwidth}
         $$\xymatrix{\bullet\ar@(ul,ur)[]\ar[d]\ar@/_2.5em/[dd]\ar@/_3.5em/[dd] \\ \bullet \ar@(ul,dl)[]\ar@(u,l)[] \ar[d] \\ \bullet \ar@(dr,dl)[]}$$\\
        \caption{A graph $E$}
        \label{fig:counterex-graph1}
    \end{subfigure}\qquad
    ~ %add desired spacing between images, e. g. ~, \quad, \qquad, \hfill etc. 
    \begin{subfigure}[b]{0.4\textwidth}
        $$\xymatrix{\bullet\ar@(ul,ur)[]\ar[d] \\ \bullet \ar@(dl,ul)[]\ar@(u,l)[] \ar[d] \\ \bullet \ar@(dr,dl)[]}$$\\
        \caption{A graph $F$}
        \label{fig:counterex-graph2}
    \end{subfigure}
        \caption{Two graphs that are not Cuntz move equivalent but $C^*(E)\otimes\K\cong C^*(F)\otimes\K$}\label{fig:counterex}
\end{figure}
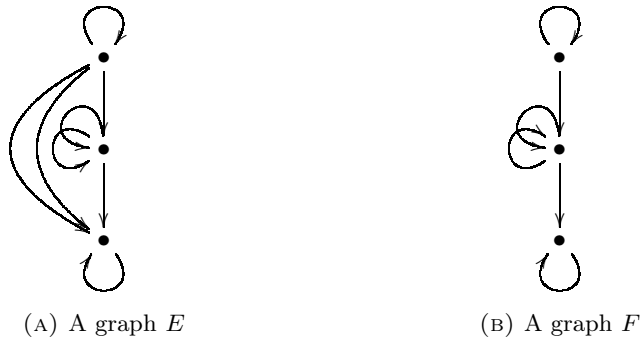
In \cite{MR3759003} it is shown that $C^*(E)\otimes\K$ is isomorphic to $C^*(F)\otimes\K$, but that $E$ is not Cuntz move equivalent to $F$. To be able to get a geometric classification, we therefore need to come up with a new invariant move for graphs that is not defined when the corresponding graph $C^*$-algebras are of real rank zero. In Definition~\ref{def:hashmove}, we define Move~\PP\ accomplishing this. The move applied to the graph $F$ can be seen in Figure~\ref{fig:counterex-changed-graph2}, while we have applied Move~\CC to $E$ in Figure~\ref{fig:counterex-changed-graph1}. It follows from the classical results of Boyle and Huang that the two graphs in Figure~\ref{fig:counterex-changed} are move equivalent, and using the results of \cite{MR3759003}, one can show that all the corresponding stabilized graph \cas are stably isomorphic. 
\begin{figure}[h]
    \quad
    \begin{subfigure}[b]{0.4\textwidth}
        $$\xymatrix{\bullet\ar@(ul,ur)[]\ar[d]\ar@/_2.5em/[dd]\ar@/_3.5em/[dd] \\ \bullet \ar@(ul,dl)[]\ar@(u,l)[] \ar[d]\ar@/^/[r] & \bullet\ar@(dr,dl)[]\ar@/^/[r]\ar@/^/[l] & \bullet\ar@(dr,dl)[]\ar@/^/[l] \\ \bullet \ar@(dr,dl)[]}$$\\
        \caption{Move \CC\ --- the Cuntz splice --- performed to the graph $E$}
        \label{fig:counterex-changed-graph1}
    \end{subfigure}\qquad
    ~ %add desired spacing between images, e. g. ~, \quad, \qquad, \hfill etc. 
    \begin{subfigure}[b]{0.4\textwidth}
        $$\xymatrix{\bullet\ar@(ul,ur)[]\ar[d]\ar@/^1em/[drr]\ar@/^2em/[drr] \\ \bullet \ar@(ul,dl)[]\ar@(u,l)[] \ar[d]\ar@/^/[r] & \bullet\ar@(dr,dl)[]\ar@/^/[r]\ar@/^/[l] & \bullet\ar@(dr,dl)[]\ar@/^/[l] \\ \bullet \ar@(dr,dl)[]}$$\\
        \caption{Move~\PP performed to the graph $F$}
        \label{fig:counterex-changed-graph2}
    \end{subfigure}
    \caption{Two graphs that are Cuntz move equivalent}\label{fig:counterex-changed}
\end{figure}
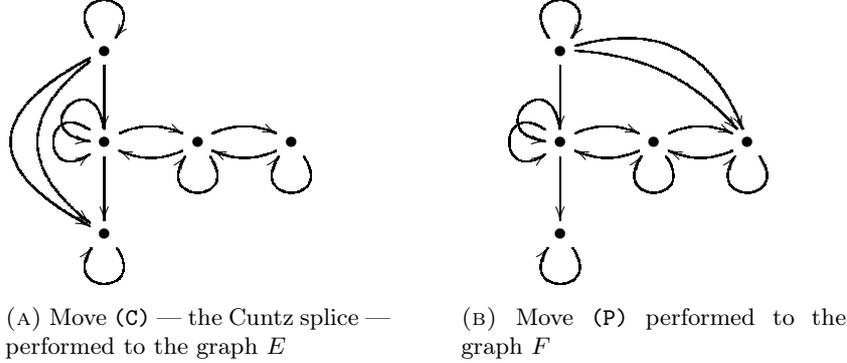

The main line of attack on establishing classification for general unital graph $C^*$-algebras follows the approach of establishing that the new move \PP\ does not change the Morita equivalence class of the $C^*$-algebra, something we prove after substantial work drawing on several recent results in the general classification theory of non-simple $C^*$-algebras. Using a localization trick essentially originating in \cite{MR866492} (see also \cite{MR1340839}, \cite{MR3713535}) we are able to establish invariance of this move in a setting which is way outside the scope of known general classification results.

By invoking  invariance of moves, the classification result follows as soon as one may infer from an isomorphism at the level of the invariants that there is a finite number of moves leading from one graph to another. We establish this by a substantial elaboration of the approach which led to the classification of flow equivalence of reducible shift spaces. Apart from having to develop tools to deal with the moves \CC\ and \PP, which have no counterpart in dynamics, we must also develop the theory for rectangular matrices to deal with graphs with singular vertices (sinks and infinite emitters), and in fact this part of the proof takes up more than half of the volume of the paper. A detailed roadmap of the rather technical argument is provided in Section \ref{section:outline} below.

\subsection{Improved classification}

In most areas of classification of unital $C^*$-algebras the proofs that establish classification by $K$-theory up to stable isomorphism may be rather easily amended to conclude that isomorphism of the  $K$-theory augmented with the class of the unit  leads to  exact isomorphism of the $C^*$-algebras. The classification of non-simple Cuntz-Krieger algebras is a notable exception to this rule, and indeed all of the foundational papers 
\cite{MR1301504}, \cite{MR1446202}, \cite{MR2270572} had to leave the question of exact isomorphism open in spite of providing stable classification results. The problem, as crystallized in \cite[Question 1--3]{MR2270572}, was that the given classification results were not \emph{strong}; it was not possible to infer that any given isomorphism at the level of the invariant would be realized by the \stariso guaranteed by the classification results.

An unexpected obstruction for obtaining strong classification was identified in \cite{MR3535322}, as indeed the full filtered $K$-theory fails to have this property even for Cuntz-Krieger algebras. This focused the efforts on establishing strong classification for the standard reduced invariant, and 
in \cite{MR3624419}, a complete description was given of the action on 
$K$-theory by the moves constituting Cuntz move equivalence. After obtaining such a description for Move \PP as well, we are able to conclude that our classification result is indeed {strong}:  Every given isomorphism at the level of the invariant is in fact induced by the moves in an appropriate sense. By methods that by now are standard, we obtain exact classification whenever the given $K$-theoretical isomorphism preserves the class of the unit, answering \cite[Question 1--3]{MR2270572} not just for the Cuntz-Krieger algebras of real rank zero considered by the second named author, but for all unital graph $C^*$-algebras.

The final Question 4 in \cite{MR2270572}  asks if there exist larger nice classes than the real rank zero Cuntz-Krieger algebras for which filtered $K$-theory is a complete invariant, and we now know that this indeed is the case --- even if one were unwilling to augment the invariant considered in \cite{MR2270572} by an order, the original version is complete for all unital graph $C^*$-algebras that are purely infinite. Further,  we now know that Restorff's classification result may be applied to any pair of purely infinite $C^*$-algebras having the same invariant as Cuntz-Krieger algebras, since by recent work of Bentmann, there are no \emph{phantom} Cuntz-Krieger algebras (\cite{MR3142030},\cite{arXiv:1511.09463v1}). Our work raises the question of whether phantom unital graph $C^*$-algebras exist, and we conjecture this to not be the case, at least when the real rank is zero.

\subsection{Applications and future work}

Unital graph $C^*$-algebras are ubiquitous in modern operator algebra, and hence our classification result can be expected to have many immediate applications, not least because the approach presented in this paper shows that the isomorphism problem in this case is decidable --- indeed, for any pair of graphs there are concrete algorithms to reduce the problem to one concerning the existence of a pair of invertible integer matrices on a certain block form which intertwine the adjacency matrices of the graphs. We give two concrete applications, addressing the uniqueness question for quantum lens spaces, and confirming the Abrams-Tomforde conjecture in the unital case.

In conclusion, let us briefly discuss the status of the classification problem for general (not necessarily unital) graph $C^*$-algebras with finitely many (gauge invariant) ideals. Since we see no obvious way to reduce to the unital case solved here, nor to mimic the geometric approach which is essential for our proof, it stands to reason that entirely different methods are going to be necessary. It is worth noting that complete classification results exist in the case where all simple subquotients are of the same type --- either AF (solved in \cite{MR0397420}) or purely infinite (solved in \cite{MR3628788}) --- and the early general results obtained by three of the authors (\cite{MR2666426},\cite{MR3056712}) similarly require restrictions on the amount of ``mixing'', so we predict that the key will be to  resolve how such issues influence the general $K\! K$-theoretic machinery which we in the unital case may replace with a geometric approach.
 
\skippable
\section{Preliminaries for statement of main theorem}
\label{genprel}

In this section, we introduce notation, concepts and results concerning graphs and their $C^*$-algebras that are needed for stating the main results. Except for the definition of  Move \PP and Pulelehua equivalence ($\MCPeq$), which are new to this paper, \cf\ Definitions~\ref{def:hashmove} and \ref{def:equivalences}, the definitions and notation follow closely \cite{MR3759003}. 
Also some important results from \cite{MR3759003} about how to view graph \cas as \cas over the space of those ideals that are prime within the gauge invariant ideals are restated. %This is done to make the paper at hand more self contained with respect to notation and thus easier to read.

\subsection{\texorpdfstring{$C^*$}{C*}-algebras over topological spaces}
Throughout this paper, we will be working with \cas over topological spaces satisfying the $T_0$ separation condition. We will use the notation and definitions from \cite{MR3759003}, which are mainly adopted from \cite[\S\S{}2.2, 2.3]{MR2545613} --- concepts such as locally closed subsets, \cas over $X$, $X$-equivariant homomorphisms and isomorphisms, and notation such as $\mathbb{O}(X)$, $\mathbb{LC}(X)$, and $\mathbb{I}(\A)$ (see \cite{MR3759003} and the references therein). As usual, $\Prim(\A)$ denotes the \emph{primitive ideal space} of \A, equipped with usual hull-kernel topology. For each $X$-equivariant homomorphism $\Phi$ and each locally closed subset $Y\subseteq X$, $\fct{\Phi}{\A}{\B}$ induces a canonical \starhomo $\fct{\Phi}{\A(Y)}{\B(Y)}$. 

\subsection{Graphs and their matrices}

By a \emph{graph} we mean a directed graph. We say that a graph is \emph{finite}, if it has finitely many vertices and edges --- all graphs considered will be \emph{countable}, \ie, there are countably many vertices and edges. We follow the definitions and notation used in \cite{MR3759003}, which are mainly adopted from \cite[\S{}2]{MR3713535} --- concepts such as vertices, edges, range maps, source maps, loops, paths, length of a path, empty paths, cycles, vertex-simple cycles, exit for a cycle, return paths, regular vertices, singular vertices, sources, sinks, isolated vertices, adjacency matrices, Condition~(K), and notation like $r$, $s$, $E^0$, $E_\mathrm{reg}^0$,$E_\mathrm{sing}^0$, $E^1$, $\N$, $\N_0$, $\Asf_E$, $\Esf_A$, as well as we write $u\geq v$ when there is a path from vertex $u$ to vertex $v$.

As in \cite{MR3759003}, it will be essential for our approach to shift between a graph and its adjacency matrix. 
For convenience and to better fit previous work, we will mainly be considering adjacency matrices over the set $\{1,2,...,n\}$, where $n$ is the number of vertices in the graph --- up to graph isomorphism, this is always possible. 
In the same way as in \cite{MR3759003}, we will for a graph $E$ let $\Asf_{E}^\bullet$ denote the matrix obtained from adjacency matrix $\Asf_{E}$ by removing all rows corresponding to singular vertices of $E$, we let $\Bsf_E$ denote the matrix $\Asf_{E} - I$, and we let $\Bsf_{E}^\bullet$ denote the matrix $\Bsf_E$ with the rows corresponding to singular vertices of $E$ removed. 

\subsection{Graph \texorpdfstring{$C^*$}{C*}-algebras}
We follow the notation and definition for graph \cas in \cite{MR1670363}; this is not the convention used in Raeburn's monograph \cite{MR2135030}. 
For a graph $E$, we use the standard notion of a \emph{Cuntz-Krieger $E$-family} and of the associated \emph{graph \ca}, $C^*(E)$ (\cf\ \cite[\S{}2.1]{MR3713535} and the references therein). 

For a graph $E$, we define the canonical gauge action $\gamma$ and the gauge invariant ideals in the usual way, and let $\mathbb{I}_\gamma(C^*(E))$ denote the sublattice of $\mathbb{I}(C^*(E))$ consisting of gauge invariant ideals (\cf\ \cite[\S{}2.3]{MR3759003}). 

Every graph \ca (of a countable graph) is separable, nuclear in the UCT class (\cite{MR1738948},\cite{MR2117597}). 
A graph \ca is unital if and only if the corresponding graph has finitely many vertices. 
The class of graph \cas contains all Cuntz-Krieger algebras, and a graph \ca is isomorphic to a Cuntz-Krieger algebra if and only if the corresponding graph is finite with no sinks (\cf\ \cite[Theorem~3.12]{MR3391894}).

\subsection{Structure of graph \texorpdfstring{$C^*$}{C*}-algebras}
For this paper, it is essential for us to view the graph \cas as $X$-algebras over a topological space $X$ that --- in general --- is different from the primitive ideal space. This is due to the fact that when there exist ideals that are not gauge invariant, then there are infinitely many ideals. 
The space we choose to work with corresponds to the distinguished ideals being exactly the gauge invariant ideals. In \cite{MR3759003}, we show a \ca{ic} characterization of the gauge invariant ideals and we describe the space $X=\Prime_\gamma(C^*(E))$, such that the distinguished ideals under a canonical $X$-algebra action is exactly the gauge invariant ideals. %For the benefit of the reader, we will recall the main definitions and results, but 

We refer the reader to \cite[\S\S{}3.1, 3.2]{MR3759003} and the references therein for definitions, notation, results, and a more detailed exposition of theory of ideal lattices and prime ideals of graph \cas, and the theory of gauge invariant prime ideals in specific --- standard concepts like hereditary subsets, saturated subsets, breaking vertices $B_H$ corresponding to a saturated and hereditary subset $H$, admissible pairs, the hereditary subset $H(S)$ generated by a subset $S$, the saturation $\overline{S}$ of a subset $S$, the ideal $\mathfrak{I}_{(H,S)}$ corresponding to an admissible pair $(H,S)$ ($\mathfrak{I}_H$ if $B_H=\emptyset$), the gap projections $p_{v_0}^H$, the primitive ideal space $\Prim(\A)$ of a \ca, and the hull $\operatorname{hull}(\mathfrak{I})$ of an ideal $\mathfrak{I}$. It is a key result in the theory that $(H,S)\mapsto\mathfrak{I}_{(H,S)}$ is a lattice isomorphism between the lattice of admissable pairs and the lattice of gauge invariant ideals. 

For the benefit of the reader, we will recall the main definitions and results from the theory of gauge invariant prime ideals, which was devoloped in \cite{MR3759003}, and therefore are less standard; a more detailed exposition is found in \cite[\S\S{}3.1, 3.2]{MR3759003}.

\begin{definition}
Let $E = (E^0 , E^1 , r , s )$ be a graph. 
Let $\Prime_\gamma(C^*(E))$ denote the set of all proper ideals that are prime within the set of proper gauge invariant ideals. 

For every ideal $\mathfrak{I}$, we let $\operatorname{hull}_\gamma(\mathfrak{I})$ denote the set $\setof{\mathfrak{p} \in\Prime_\gamma(C^*(E))}{\mathfrak{p}\supseteq\mathfrak{I}}$ (similarly to the definition of $\operatorname{hull}(\mathfrak{I})$). 
We equip $\Prime_\gamma(C^*(E))$ with the topology where the closure of a subset $S\subseteq\Prime_\gamma(C^*(E))$ is 
$$\operatorname{hull}_\gamma(\cap S)=\setof{\mathfrak{p}\in\Prime_\gamma(C^*(E))}{\mathfrak{p}\supseteq\cap S}.$$
\end{definition}

The following proposition tells us how we may consider a graph \ca as an algebra over $\Prime_\gamma(C^*(E))$ such that the distinguished ideals are exactly the gauge invariant ideals. Proofs and a more thorough exposition are given in \cite{MR3759003}.

\begin{proposition}[\cf\ {\cite[Lemmata~3.3, 3.5, 3.6, and 3.7]{MR3759003}}] \label{prop:structure-first-prime-gamma}
Let $E$ be a graph.
\begin{enumerate}[(1)]
\item Isomorphisms between graph \cas preserve gauge invariant ideals --- the same holds for isomorphisms of the stabilization (with the standard identification of the ideal lattice of a \ca with that of its stabilization). 
\item
Every primitive gauge invariant ideal of $C^*(E)$ is in $\Prime_\gamma(C^*(E))$, and every primitive ideal of $C^*(E)$ that is not gauge invariant has a largest gauge invariant ideal as a subset, moreover this gauge invariant ideal is in $\Prime_\gamma(C^*(E))$. 
\item 
The map 
$$\mathfrak{I}\mapsto\operatorname{hull}_\gamma(\mathfrak{I})
=\setof{\mathfrak{p}\in\Prime_\gamma(C^*(E))}{\mathfrak{p}\supseteq\mathfrak{I}}$$ 
is an order reversing bijection from the gauge invariant ideals of $C^*(E)$ to the closed subsets of $\Prime_\gamma(C^*(E))$. 
Its inverse map is $S\mapsto \cap S$.
\item
Consider the map $\zeta$ from $\Prim(C^{*}(E))$ to $\Prime_\gamma(C^{*}(E))$ sending each primitive ideal to the largest element of $\Prime_\gamma(C^{*}(E))$ that it contains. 
This map is continuous and surjective. Moreover,
\begin{equation}
\zeta^{-1}(\operatorname{hull}_\gamma(\mathfrak{I}))
=\operatorname{hull}(\mathfrak{I}),
\end{equation}
for every gauge invariant ideal $\mathfrak{I}$ of $C^*(E)$, so the distinguished ideals under this $\Prime_\gamma(C^{*}(E))$ action are exactly the gauge invariant ideals. 
\end{enumerate}
\end{proposition}

\begin{remark}
Assume that $E$ is a graph with finitely many vertices. Then $E$ satisfies Condition~(K) if and only if $C^*(E)$ has real rank zero if and only if $C^*(E)$ has finitely many ideals. 
When $E$ satisfies Condition~(K), then $\Prim(C^*(E))=\Prime_\gamma(C^*(E))$. 
\end{remark}

\subsection{Moves on graphs}
We need certain moves on graphs in order to formulate the geometric classification result, and these moves are also essential for the proofs of the classification results of this paper. 
The moves Move~\SSS (remove a regular source), Move~\RR (reduction at a regular vertex), Move~\OO (outsplit at a non-sink), Move~\II (insplit at a regular non-source), and Move~\CC (Cuntz splicing a vertex with at least two return paths) on a graph $E $ are all introduced in Propositions~3.1 and~3.2, Theorems~3.3 and~3.5 and the definition following Theorem~3.5 in \cite{MR3082546}, and we denote the resulting graphs by $E_S$, $E_R$, $E_O$, $E_I$, and $E_C$, respectively. 

%Note that we will also allow Cuntz splicing singular vertices (that support at least two return paths); see 
See \cite{MR3713535} and especially Definition~2.11 therein for the definition and notation used for the Cuntz spliced graph. 
We also use the notation $E_{v,-}$ for this graph --- even in the case where $v$ is not regular or not supporting two return paths. We can also Cuntz splice the vertex $u_1$ in $E_{v,-}$, and the resulting graph we denote $E_{v,--}$. 
In Figures~\ref{fig:counterex-graph1} and~\ref{fig:counterex-changed-graph1} we saw an example of a Cuntz splicing. 
See \cite[Notation~4.1 and Example~4.2]{MR3713535} for more general illustrations of the Cuntz splice.
Let $S$ be a subset of $E^0$.  We denote the graph where we have Cuntz spliced each vertex $w$ in $S$ by $E_{S,-}$.  The additional vertices will be denoted by $v_1^w$ and $v_2^w$ for $w \in S$. 

Now we will introduce a new important move, that is needed to classify unital graph \cas in general. It has been shown in \cite{MR3082546} that the moves \SSS, \RR, \OO, and \II preserve the Morita equivalence class of the associated graph \cas, and it has been shown in \cite{MR3713535} that the Cuntz splice also preserves the Morita equivalence class of the associated graph \ca. 
Using results from \cite{MR3624419}, we will in the current paper keep track of what such a Morita equivalence induces on the reduced filtered $K$-theory as well (see Corollary~\ref{cor:cuntzspliceinvariant}). In the current paper we will, moreover, show that the new move \PP\ also preserves the Morita equivalence class of the associated graph \ca with a certain control on the behaviour on the reduced filtered $K$-theory (Proposition~\ref{prop:hash} and Theorem~\ref{thm:InvarianceOfPulelehua}). 

\begin{definition}[Move \PP: Eclosing a cyclic component] \label{def:hashmove}
Let $E=(E^0,E^1,r,s)$ be a graph and let $u$ be a regular vertex that supports a loop and no other return path, the loop based at $u$ has an exit, and if $w \in E^0 \setminus \{u\}$ and $s_E^{-1} (u) \cap r_E^{-1}(w) \neq \emptyset$, then $w$ is a regular vertex that supports at least two distinct return paths.  We construct $E_{u, P}$ as follows.  

Set $S = \setof{w \in E^0 \setminus \{u\}}{s_E^{-1}(u) \cap r_E^{-1}(w) \neq \emptyset}$ (since the loop based at $u$ has an exit, $S$ is nonempty, and clearly $u \notin S$).  Set $E_{u, P}^0 = E_{S,-}^0$ and
\[
E_{u, P}^1 = E_{S,-}^1 \sqcup \setof{\overline{e}_w, \tilde{e}_w }{w \in S, e \in s_E^{-1} (u) \cap r_E^{-1}(w) }
\]
with $s_{E_{u, P}} \vert_{E_{S,-}^1} = s_{E_{S,-}}$, $r_{E_{u, P}} \vert_{E_{S,-}^1} = r_{E_{S,-}}$, $s_{E_{u, P}} ( \overline{e}_w ) = s_{E_{u, P}} ( \tilde{e}_w ) = u$, and $r_{E_{u, P}} ( \overline{e}_w ) = r_{E_{u, P}} ( \tilde{e}_w ) = v_2^w$.
We call $E_{u,P}$ the \emph{graph obtained by eclosing $E$ at $u$}, and we say that $E_{u,P}$ is formed by performing Move \PP to $E$. 

Note that 
\[
\Bsf_{E_{u, P} } ( u, v_2^w ) = 2 \Bsf_E ( u, w )
\]
for all $w \in S$.  Therefore, $E_{u,P}$ is precisely $E_{S,-}$ where we add additional edges from $u$ to $v_2^w$ for each $w \in S$.  The number of edges from $u$ to $v_2^w$ is twice the number of edges from $u$ to $w$.  In Figures~\ref{fig:counterex-graph2} and~\ref{fig:counterex-changed-graph2} we saw an example of Move~\PP. 
See Notation~\ref{notation:OnceAndTwice-hash-2} and
Example~\ref{example:hash} for illustrations of Move~\PP, in the special case that $E$ and $u$ satisfy the conditions in Assumption~\ref{assumption:1-hash} (here $E_{u,\#}$ is exactly $E_{u,P}$ for these special case). 
\end{definition}

\begin{definition}\label{def:equivalences}
The equivalence relation generated by the moves \OO, \II, \RR, \SSS together with graph isomorphism is called \emph{move equivalence}, and is denoted \Meq. 
The equivalence relation generated by the moves \OO, \II, \RR, \SSS, \CC together with graph isomorphism is called \emph{Cuntz move equivalence}, and is denoted \MCeq. 
The equivalence relation generated by the moves \OO, \II, \RR, \SSS, \CC, \PP together with graph isomorphism is called \emph{\Pul equivalence}, and is denoted \MCPeq. 
\end{definition}

%\skippable
It was essentially proved in \cite{MR2054048} that move equivalence of graphs implies the corresponding graph \cas to be stably isomorphic (see also  \cite[Propositions~3.1, 3.2 and 3.3 and Theorem~3.5]{MR3082546}). 
Combining this with \cite[Theorem~4.8]{MR3713535}, it follows that Cuntz move equivalence of graphs also implies the corresponding graph \cas to be stably isomorphic. 

As in \cite[Definition~2.10]{MR3759003}, we also extend the notation of equivalences to adjacency matrices. Note that the Cuntz move equivalence, $\MCeq$, is called \emph{move prime equivalence} in \cite{MR3082546}. Since the similarity of the two terms could create confusion, we have chosen to use the term \emph{Cuntz move equivalence} instead. 

\subsection{Reduced filtered \texorpdfstring{$K$}{K}-theory}(\cf\ {\cite[Section~3.3]{MR3759003}}).

Let $X$ be a topological space satisfying the $T_0$ separation
property and let \A be a \ca over $X$.  For open subsets $U_{1} ,
U_{2} , U_{3}$ of $X$ with $U_{1} \subseteq U_{2} \subseteq U_{3}$, let $Y_{1} = U_{2} \setminus U_{1}, Y_{2} = U_{3} \setminus U_{1},
Y_{3} = U_{3} \setminus U_{2}\in \mathbb{LC}(X)$.  Then the
diagram
\begin{equation*}
\xymatrix{
K_{0} ( \A ( Y_{1}  ) ) \ar[r]^{ \iota_{*} } & K_{0} ( \A ( Y_{2} ) ) \ar[r]^{ \pi_{*} } & K_{0} ( \A ( Y_{3} ) ) \ar[d]^{\partial_{*}} \\
K_{1} ( \A ( Y_{3}  ) ) \ar[u]^{ \partial_{*}} & K_{1} ( \A ( Y_{2} ) ) \ar[l]^{ \pi_{*} } & K_{1} ( \A ( Y_{1} ) ) \ar[l]^{\iota_{*}}
}
\end{equation*}
is an exact sequence. The collection of all such exact sequences is an invariant of the \cas over $X$ often referred to as the \emph{filtered $K$-theory}. We use here a refined notion called \emph{reduced filtered $K$-theory}, which only includes some of these groups and some of these maps; the reduced filtered $K$-theory of a \ca \A is denoted $\FKR(X;\A)$. There is also a variant called \emph{ordered reduced filtered $K$-theory}, which takes into account the order of the involved groups; the ordered reduced filtered $K$-theory of a \ca \A is denoted $\FKRplus(X;\A)$. 
We refer the reader to \cite[\S{}3.3]{MR3759003} for a detailed exposition of the reduced filtered $K$-theory.

\begin{remark}[\cf\ {\cite[Remark~3.22]{MR3759003}}]\label{howtocompute}
Let $E$ be a graph. 
Then $C^*(E)$ has a canonical structure as a $\Prime_\gamma(C^*(E))$-algebra.
So if $E$ has finitely many vertices --- or, more generally, if $\Prime_\gamma(C^*(E))$ is finite --- then we can consider the reduced filtered $K$-theory, $\FKR(\Prime_\gamma(C^*(E));C^*(E))$. 
We use the results of \cite{MR2922394} to identify the $K$-groups and the homomorphisms in the cyclic six term sequences using the adjacency matrix of the graph. 
\end{remark}

\begin{remark}[\cf\ {\cite[Remark~3.23]{MR3759003}}]
Let \A be an $X$-algebra. 
Since $\mathfrak{I}\mapsto\mathfrak{I}\otimes\K$ is a lattice isomorphism between $\mathbb{I}(\A)$ and $\mathbb{I}(\A\otimes\K)$, the \ca $\A\otimes\K$ is an $X$-algebra in a canonical way, and the embedding $\kappa_\A$ given by $a\mapsto a\otimes e_{11}$ is an $X$-equivariant homomorphism from $\A$ to $\A\otimes\K$. 
Also, it is clear that  $\FKR(X;\kappa_\A)$ is an (order) isomorphism. 
Note also that the invariant $\FKR(X;-)$ has been considered in \cite{MR3177344,MR3349327}.
\end{remark}
\nomoreskippable

\begin{definition}\label{def:PreservingTheUnit}
Let $E_1$ and $E_2$ be graphs with finitely many vertices.
Assume that we have a homeomorphism between $\Prime_\gamma(C^*(E_1))$ and $\Prime_\gamma(C^*(E_2))$.  
We view $C^{*} ( E_{2} )$ as a $\Prime_\gamma(C^*(E_1))$-algebra via this homeomorphism.

Set $X=\Prime_\gamma(C^*(E_1))$.  From \cite[Lemma~8.3]{MR3349327} it follows that every homomorphism $\phi$ from $\FKR(X;C^*(E_1))$ to $\FKR(X;C^*(E_2))$ induces a unique homomorphism $\phi_0$ from $K_0(C^*(E_1))$ to $K_0(C^*(E_2))$ --- this gives, in fact, a unique functor $H$ from the category of exact real-rank-zero-like $\mathcal{R}$-modules to the category of abelian groups such that $H\circ \FKR=K_0\circ F$, where $F$ is the forgetful functor sending an $X$-algebra to the corresponding \ca (\cf\ \cite[Definitions~3.6 and~3.9]{MR3349327} and \cite[Theorems~3.2 and~3.3]{arXiv:1301.7695v1}). 
We say that the reduced filtered $K$-theory isomorphism $\phi$ \emph{preserves the class of the unit} if $\phi_0([1_{C^*(E_1)}]_0)=[1_{C^*(E_2)}]_0$. 
Whenever such an isomorphism exists, 
we write $(\FKR(X;C^*(E_1)),[1_{C^*(E_1)}]_0)\cong(\FKR(X;C^*(E_2)),[1_{C^*(E_2)}]_0)$ and say that $(\FKR(X;C^*(E_1)),[1_{C^*(E_1)}]_0)$ and $(\FKR(X;C^*(E_2)),[1_{C^*(E_2)}]_0)$ are \emph{isomorphic}. % whenever such an isomorphism exists.
\end{definition}

\section{Main results}

\label{sec:main}

Our first main theorem is the following geometric classification of all unital graph \cas up to stable isomorphism. 
\begin{theorem}\label{thm:main-1}
Let $E_1$ and $E_2$ be graphs with finitely many vertices. 
Then the following are equivalent:
\begin{enumerate}[(1)]
\item\label{thm:main-1-item-1} 
$E_{1} \MCPeq E_{2}$, 
\item\label{thm:main-1-item-2} 
$C^{*} (E_{1} )\otimes\K \cong C^{*} ( E_{2} )\otimes\K$, and,
\item\label{thm:main-1-item-3} 
there exists a homeomorphism between $\Prime_\gamma(C^*(E_1))$ and $\Prime_\gamma(C^*(E_2))$ such that we have an isomorphism from 
$\FKRplus(\Prime_\gamma(C^*(E_1)); C^{*} (E_{1} ) )$ to $\FKRplus(\Prime_\gamma(C^*(E_1)); C^{*} ( E_{2} ) )$, where we view $C^{*} ( E_{2} )$ as a $\Prime_\gamma(C^*(E_1))$-algebra via the homeomorphism between $\Prime_\gamma(C^*(E_1))$ and $\Prime_\gamma(C^*(E_2))$.
\end{enumerate}
\end{theorem}

Note that when $E$ has finitely many vertices and satisfies Condition~(K), then there are only finitely many ideals and they are all gauge invariant, and consequently $\Prim(C^*(E))=\Prime_\gamma(C^*(E))$. Since in this case we cannot perform Move~\PP, we get the following.

\begin{corollary}\label{cor:main-2}
Let $E_1$ and $E_2$ be graphs with finitely many vertices satisfying Condition~(K).
Then the following are equivalent:
\begin{enumerate}[(1)]
\item\label{cor:main-2-item-1} 
$E_{1} \MCeq E_{2}$, 
\item\label{cor:main-2-item-2} 
$C^{*} (E_{1} )\otimes\K \cong C^{*} ( E_{2} )\otimes\K$, and,
\item\label{cor:main-2-item-3} 
there exists a homeomorphism between $\Prim(C^*(E_1))$ and $\Prim(C^*(E_2))$ such that  
we have an isomorphism from $\FKRplus(\Prim(C^*(E_1)); C^{*} (E_{1} ) )$ to $\FKRplus(\Prim(C^*(E_1)); C^{*} ( E_{2} ) )$, where we view $C^{*} ( E_{2} )$ as a $\Prim(C^*(E_1))$-algebra via the above homeomorphism between $\Prim(C^*(E_1))$ and $\Prim(C^*(E_2))$.
\end{enumerate}
\end{corollary}

In the case that no cyclic component has only noncyclic strongly connected components as immediate successors, we cannot perform Move~\PP\ --- for finite graphs this condition is considered in \cite{MR3759003} and is called Condition~(H). So we get the following. 

\begin{corollary}\label{cor:main-2.5}
Let $E_1$ and $E_2$ be graphs with finitely many vertices satisfying that no cyclic component has only noncyclic strongly connected components as immediate successors.
Then the following are equivalent:
\begin{enumerate}[(1)]
\item\label{cor:main-2.5-item-1} 
$E_{1} \MCeq E_{2}$, 
\item\label{cor:main-2.5-item-2} 
$C^{*} (E_{1} )\otimes\K \cong C^{*} ( E_{2} )\otimes\K$, and,
\item\label{cor:main-2.5-item-3} 
there exists a homeomorphism between $\Prime_\gamma(C^*(E_1))$ and $\Prime_\gamma(C^*(E_2))$ such that we have an isomorphism from
$\FKRplus(\Prime_\gamma(C^*(E_1)); C^{*} (E_{1} ) )$ to $\FKRplus(\Prime_\gamma(C^*(E_1)); C^{*} ( E_{2} ) )$, where we view $C^{*} ( E_{2} )$ as a $\Prime_\gamma(C^*(E_1))$-algebra via the homeomorphism between $\Prime_\gamma(C^*(E_1))$ and $\Prime_\gamma(C^*(E_2))$.
\end{enumerate}
\end{corollary}

Our second main theorem is the following theorem that says that the functor $\FKR$ is a strong classification functor up to stable isomorphism of all unital graph \cas.  This closes an open problem from \cite[Question~3]{MR2270572}, that has very interesting connections to the result \cite[Theorem~3.5/Proposition~4.3]{MR3535322} --- for more on this see Remark~\ref{rem:openquestionCK} below.

\begin{theorem}\label{thm:main-3}
Let $E_1$ and $E_2$ be graphs with finitely many vertices.  Assume that we have a homeomorphism between $\Prime_\gamma(C^*(E_1))$ and $\Prime_\gamma(C^*(E_2))$.
We view $C^{*} ( E_{2} )$ as a $\Prime_\gamma(C^*(E_1))$-algebra via this homeomorphism.

Set $X=\Prime_\gamma(C^*(E_1))$.  Then for every isomorphism $\phi$ from $\FKRplus(X; C^{*} (E_{1} ) )$ to $\FKRplus(X; C^{*} ( E_{2} ) )$, there exists an $X$-equivariant isomorphism $\Phi$ from $C^{*} (E_{1} )\otimes\K$ to $C^{*} ( E_{2} )\otimes\K$ satisfying $\phi=\FKR(X;\Phi)$ (where we use the canonical identification of $\FKRplus(X; C^{*} (E_{i} ) )$ with $\FKRplus(X; C^{*} (E_{i}) \otimes\K )$, for $i=1,2$).
\end{theorem}

Using this result, we can show the following theorem --- which is our third main theorem. It says that the functor $\FKR$ together with the class of the unit is a strong classification functor up to isomorphism of all unital graph \cas.  This closes an open problem from \cite[Question~1]{MR2270572}.  

\begin{theorem}\label{thm:main-4}
Let $E_1$ and $E_2$ be graphs with finitely many vertices.
Assume that we have a homeomorphism between $\Prime_\gamma(C^*(E_1))$ and $\Prime_\gamma(C^*(E_2))$.
We view $C^{*} ( E_{2} )$ as a $\Prime_\gamma(C^*(E_1))$-algebra via this homeomorphism.

Set $X=\Prime_\gamma(C^*(E_1))$.  Then for every isomorphism $\phi$ from $\FKRplus(X; C^{*} (E_{1} ) )$ to $\FKRplus(X; C^{*} ( E_{2} ) )$ that preserves the class of the unit (in the sense of Definition~\ref{def:PreservingTheUnit}), there exists an $X$-equivariant isomorphism $\Phi$ from $C^{*} (E_{1} )$ to $C^{*} ( E_{2} )$ satisfying $\phi=\FKR(X;\Phi)$.
\end{theorem}

As an immediate consequence, we get the following classification result.
\begin{corollary}\label{cor:main-5}
Let $E_1$ and $E_2$ be graphs with finitely many vertices.  
Then the following are equivalent
\begin{enumerate}[(1)]
\item\label{cor:main-5-item-2} 
$C^{*} (E_{1} )\cong C^{*} ( E_{2} )$, and,
\item\label{cor:main-5-item-3} 
there exists a homeomorphism between $\Prime_\gamma(C^*(E_1))$ and $\Prime_\gamma(C^*(E_2))$ such that 
\[
(\FKRplus(\Prime_\gamma(C^*(E_1));C^*(E_1)),[1_{C^*(E_1)}]_0)
\] 
and  
\[
(\FKRplus(\Prime_\gamma(C^*(E_1));C^*(E_2)),[1_{C^*(E_2)}]_0)
\] 
are isomorphic, where we view $C^{*} ( E_{2} )$ as a $\Prime_\gamma(C^*(E_1))$-algebra via the given homeomorphism between $\Prime_\gamma(C^*(E_1))$ and $\Prime_\gamma(C^*(E_2))$.
\end{enumerate}
\end{corollary}

\begin{remark}\label{rem:openquestionCK}
The results above close the classification problem for the case of general Cuntz-Krieger algebras as well as for the case of general unital graph \cas (over countable graphs) and they answer the remaining open problems from the addendum in \cite{MR2270572}. 
An important motivation for showing these strong classification results has been the results in \cite{arXiv:1101.5702v3} and \cite{MR3535322}. Building on results from \cite{MR2953205}, it is shown in \cite{arXiv:1101.5702v3}  that for finite connected $T_0$ spaces that are not so-called accordion, the full filtered K-theory does not classify the stable, (strongly) purely infinite, nuclear, separable \cas in the bootstrap class --- on the contrary the results of \cite{MR2270572} show that all purely infinite Cuntz-Krieger algebras are classified up to stable isomorphism by the coarser invariant \emph{reduced filtered $K$-theory} as well as the full filtered $K$-theory. The results in \cite[Theorem 3.5]{MR3535322} even show that the full filtered $K$-theory is \emph{not} a strong classification functor up to stable isomorphism for the class of purely infinite Cuntz-Krieger algebras with at most four primitive ideals, while the smaller invariant \FKR is.  
\end{remark}

\subsection{Strategy of proof and structure of the paper}\label{section:outline}
The proof of Theorem~\ref{thm:main-1} goes as follows. 
That \ref{thm:main-1-item-1} implies \ref{thm:main-1-item-2} is shown in three stages. 
First we note that Move equivalence implies stable isomorphism (which has already been shown elsewhere, \cf\ \cite[Theorem~2.7]{MR3759003} and the references there). 
Then we note that in \cite[Theorem~4.8]{MR3713535} it was proved that Move \CC (Cuntz splicing) implies stable isomorphism. 
Finally, Section~\ref{sec:MoveP} is to a large extent devoted to prove that Move \PP implies stable isomorphism (\cf\ Theorem~\ref{thm:InvarianceOfPulelehua}).
That \ref{thm:main-1-item-2} implies \ref{thm:main-1-item-3} is clear. 
The proof that \ref{thm:main-1-item-3} implies \ref{thm:main-1-item-1} is done together with the proof of Theorem~\ref{thm:main-3}. 
Corollaries~\ref{cor:main-2} and~\ref{cor:main-2.5} will follow directly from Theorem~\ref{thm:main-1} --- as explained above. 

Theorem~\ref{thm:main-4} will also follow from Theorem~\ref{thm:main-3} using \cite[Theorem~3.3]{arXiv:1301.7695v1} (\cf~also \cite[Theorem~2.1]{MR2379290}). 
Corollary~\ref{cor:main-5} will clearly follow from Theorem~\ref{thm:main-4}. 

The outline of the proof of Theorem~\ref{thm:main-3} and of the proof that \ref{thm:main-1-item-3} implies \ref{thm:main-1-item-1} in Theorem~\ref{thm:main-1} is as follows. 
In parts, we emulate the previous proofs that go from reduced filtered $K$-theory data to stable isomorphism or flow equivalence, as in \cite{MR1990568, MR1907894, MR2270572, MR3759003,MR3624419}.
A key component of those proofs is manipulation of the matrix $\Bsf_E^{\bullet}$, in particular that we can perform basic row and column operations without changing the stable isomorphism class or the move equivalence class, depending on context. 
It was proved in \cite{MR3713535} that these matrix manipulation are allowed (\cf\ Proposition~\ref{prop:matrix-moves}).
When we do edge expansions and the above mentioned matrix manipulations, we get induced stable isomorphisms; in \cite{MR3624419} the reduced filtered $K$-theory isomorphisms induced by such stable isomorphisms are described (\cf\ Section~\ref{sec:edgeandtoke}). 
Once we understand these matrix manipulations, our proof of Theorem~\ref{thm:main-1} goes through 6 steps.

\begin{enumerate}[(1)]
 \item[Step 1] \label{mainthm-step-standardform}  First we reduce the problem by finding graphs $F_1$ and $F_2$ in a certain standard form such that $F_i \Meq E_i$ and $C^*(E_i)\otimes\K\cong C^*(F_i)\otimes\K$. This standard form will ensure that the adjacency matrices $\Bsf^{\bullet}_{F_i}$ have the same size and block structure, and that they satisfy certain additional technical conditions. This will be done in Section~\ref{sec:standardform}. 
 
 \item[Step 2] \label{mainthm-step-FKtoGL} In Section~\ref{sec:boylehuang} we generalize a result of Boyle and Huang (\cite{MR1990568}), to show that the isomorphism $\FKRplus(X; C^{*} (F_{1} ) ) \cong \FKRplus(X; C^{*} ( F_{2} ) )$ is induced by a \GLPEe from $\Bsf^{\bullet}_{F_1}$ to $\Bsf^{\bullet}_{F_2}$ with a certain positivity condition (where the action of $X=\Prime_\gamma(C^{*} (E_{1}))$ is as given in the theorem). 
 
 \item[Step 3] \label{mainthm-step-Pulelehua} In Section~\ref{sec:MoveP} we find graphs $F_1', F_2'$ such that $F_i' \MCPeq F_i$ and such that we have an equivariant isomorphism from $C^*(F_i)\otimes\K$ to $C^*(F_i')\otimes\K$ that induces an isomorphism on the reduced filtered $K$-theory that is induced by a \GLPEe such that the diagonal blocks of the composed \GLPEe from $\Bsf^{\bullet}_{F_1'}$ to $\Bsf^{\bullet}_{F_2'}$ are \SL-matrices for all blocks that do not correspond to simple purely infinite subquotients. 
 
 \item[Step 4] \hypertarget{mainthm-step-GLtoSL}{} In Section~\ref{sec:crelle-trick} we find graphs $G_1, G_2$ such that $G_i \MCeq F_i'$ and such that we have an equivariant isomorphism from $C^*(F_i')\otimes\K$ to $C^*(G_i)\otimes\K$ that induces an isomorphism on the reduced filtered $K$-theory that is induced by a \GLPEe such that the composed \GLPEe from $\Bsf^{\bullet}_{G_1}$ to $\Bsf^{\bullet}_{G_2}$ is in fact an \SLPEe. 
 
 \item[Step 5] \label{mainthm-step-SLtoMoves} In Section~\ref{sec:boyle}, we generalize Boyle's positive factorization result from \cite{MR1907894} to show that there exists a positive \SLPEe between $\Bsf^{\bullet}_{G_1}$ and $\Bsf^{\bullet}_{G_2}$.
 
 \item[Step 6] It now follows from the results of Section~\ref{sec:derivedmoves} that $G_1 \Meq G_2$ and hence that $E_1 \MCPeq E_2$. 
 Moreover, it follows from the results of Section~\ref{sec:edgeandtoke} that there exists an equivariant isomorphism from $C^*(G_1)\otimes\K$ to $C^*(G_2)\otimes\K$ that induces exactly the same morphism on the reduced filtered $K$-theory as the above \SLPEe from $\Bsf^{\bullet}_{G_1}$ to $\Bsf^{\bullet}_{G_2}$. Thus showing that the original isomorphism between the ordered reduced filtered $K$-theory of $C^*(E_1)$ and $C^*(E_2)$ can be lifted to a \stariso between $C^*(E_1)\otimes\K$ and $C^*(E_2)\otimes\K$. 
\end{enumerate}

In Section~\ref{sec:notation-for-proof} we introduce some notation and concepts about block matrices needed in the proof.
In Section~\ref{sec:proof} we combine the results of the previous sections to prove the main theorem. 

In Section~\ref{sec:addenda}, we address some applications of the above main results of this paper. 
In Section~\ref{ordereasy}, we note that we can relax the conditions about positivity of reduced, filtered $K$-theory isomorphisms because of a certain amount of automatic positivity for graph \cas, and we raise the question of existence of phantom unital graph \cas. 
In Section~\ref{sec:decidability}, we address the decidability, and prove that Morita equivalence of unital graph \cas is decidable using results of Boyle and Steinberg. 
In Section~\ref{sec:postliminal}, we prove for unital graph \cas that are postliminal/of type I that they are stably isomorphic if and only if they are isomorphic, and apply this to the quantum lens spaces to complete the investigation started in \cite{MR3759003}.
In Section~\ref{sec:abrams-tomforde}, we prove that the Abrams-Tomforde conjectures hold true for graphs with finitely many vertices. 

\skippable
\section{Derived moves and matrix moves}
\label{sec:derivedmoves}

In \cite[Section~2.5]{MR3759003}, we discussed --- following and  generalizing  \cite[Section 5]{MR3082546} --- ways of changing the graphs without   changing their move equivalence class and we introduced  a \emph{collapse} move, and presented criteria allowing us to conclude that two graphs are move equivalent when one arises from the other by a row or column addition of the $\Bsf$-matrices. 

We refer the reader to \cite[Section~3]{MR3759003} and the references therein for the collapse move (alternatively, \cite[Definition~2.7]{MR3713535}); we denote the resulting graph by $E_{COL}$ and denote the move \CO. Here, it is also noted that  collapsing a regular vertex that does not support a loop yields move equivalent graphs.

In the result \cite[Proposition~2.12]{MR3759003} below from \cite{MR3713535}, we showed how we can perform row and column additions on $\Bsf_E$ without changing the move equivalence class of the associated graphs, when $E$ is a graph with finitely many vertices. 
In Section~\ref{sec:edgeandtoke} we will see how these act on reduced filtered $K$-theory. 
Since we will constantly make use of this result, we will for the convenience of the reader recall it here. 

\begin{proposition}[{\cf\ \cite[Proposition~2.12 and Remark~2.13]{MR3759003}}]
\label{prop:matrix-moves}
Let $E = (E^0, E^1, r,s)$ be a graph with finitely many vertices. 
Suppose $u,v \in E^0$ are distinct vertices with a path from $u$ to $v$.
Let $E_{u,v}$ be equal to the identity matrix except for on the $(u,v)$'th entry, where it is $1$. 
Then $\Bsf_E E_{u,v}$ is the matrix formed from $\Bsf_E$ by adding the $u$'th column into the $v$'th column, while $E_{u,v}\Bsf_E $ is the matrix formed from $\Bsf_E$ by adding the $v$'th row into the $u$'th row.
Then the following holds.
\begin{enumerate}[(i)]
\item Suppose $u$ supports a loop or suppose that there is an edge from $u$ to $v$ and $u$ emits at least two edges. Then  
\[
	\Asf_E \Meq \Bsf_E E_{u,v} + I.
\]\label{prop:matrix-moves:I}
\item Suppose $v$ is regular and either $v$ supports a loop or there is an edge from $u$ to $v$. Then  
\[
	\Asf_E \Meq E_{u,v}\Bsf_E  + I.
\]\label{prop:matrix-moves:II}
\end{enumerate}

Note that we can also use this proposition backwards to subtract columns or rows in $\Bsf_E$ as long as the addition that undoes the subtraction is legal.
\end{proposition}

\section{Notation needed for the proof}
\label{sec:notation-for-proof}

In this section we provide notation and results that are needed in the proof of the main result. This is mainly a condensed version of some definitions, results and remarks from \cite{MR3759003}. The reader is referred to \cite{MR3759003} for a more comprehensive exposition, and especially for the proofs of the mentioned results.

\subsection{The component poset}(\cf\ \cite[Section~3.2]{MR3759003}).
For our purposes, it will be essential to work with block matrices in a way that resembles the ideal structure and the filtered $K$-theory of the graph \cas. 
For the case we are considering --- where there are \emph{finitely many vertices} --- the framework for this has been provided in \cite{MR3759003}. 
We will recall some of the basic definitions and the main results. 

\begin{definition}\label{def:structure-a}
Let $E=(E^0,E^1,r,s)$ be a graph with finitely many vertices. 
We say that a nonempty subset $S$ of $E^{0}$ is \emph{strongly connected} if for any two vertices $v,w\in S$ there exists a nonempty path from $v$ to $w$. 
In particular every vertex in a strongly connected set has to be the base of a cycle. 
We let $\Gamma_E$ denote the set of all maximal strongly connected sets together with all singletons consisting of singular vertices that are not the base point of a cycle. The elements of $\Gamma_E$ are called the \emph{components of the graph} $E$ and the vertices in $E^{0}\setminus\cup\Gamma_E$ are called the \emph{transition states} of $E$. 
In particular, with this convention, regular sources are called transition states. 
A component is called a \emph{cyclic component} if one of the vertices in it has exactly one return path.

We define a partial order $\geq$ on $\Gamma_E$ by saying that $\gamma_1\geq\gamma_2$ if there exist vertices $v_1\in\gamma_1$ and $v_2\in\gamma_2$ such that $v_1\geq v_2$. 
We say that a subset $\sigma\subseteq\Gamma_E$ is hereditary if whenever $\gamma_1,\gamma_2\in \Gamma_E$ with $\gamma_1\in \sigma$ and $\gamma_1\geq\gamma_2$, then $\gamma_2\in \sigma$. 
We equip $\Gamma_E$ with the topology that has the hereditary subsets as open sets. 
For every subset $\sigma\subseteq\Gamma_E$, we let $\eta(\sigma)$ denote the smallest hereditary subset of $\Gamma_E$ containing $\sigma$. 
\end{definition}

The next two propositions tell us the following.  Suppose $E$ is a graph with finitely many vertices such that every infinite emitter emits infinitely many edges to any vertex it emits any edge to and every transition state has exactly one outgoing edge.  Then we can partition the vertices of $E$ such that we can use this partition to write the adjacency matrix in a block form that reflects the ideal structure. This will enable us to compute $K$-theory more easily and to talk about the $K$-web, \GLPEe{s} and \SLPEe{s}. This is essential for what follows. 

\begin{proposition}[{\cite[Lemmata~3.13 and~3.17 and Corollary~3.14]{MR3759003}}]\label{prop:first-collect-structure-1}
Let $E=(E^0,E^1,r,s)$ be a graph with finitely many vertices. 
\begin{enumerate}[(1)]
\item The map $\eta\mapsto\overline{H(\cup\eta)}$ from the set of hereditary subsets of $\Gamma_E$ to the set of saturated hereditary subsets of $E^0$ is a bijective order isomorphism (with respect to the order coming from set containment). The inverse map is given by the equivalence classes of $(\cup\Gamma_E)\cap H$ for each saturated hereditary subset $H\subseteq E^0$. 
\item
If $E$ does not have any transition state, then every hereditary subset of $E^0$ is saturated and $\eta\mapsto \cup \eta$ is a lattice isomorphism between the hereditary subsets of $\Gamma_E$ and the hereditary subsets of $E^0$. 
\item
If $\gamma\in\Gamma_E$, then $\overline{H(\gamma)\setminus\gamma}$ is the largest proper saturated hereditary subset of $\overline{H(\gamma)}$. 
\item 
If every transition state has exactly one edge going out, then 
$H_1\cup H_2=\overline{H_1\cup H_2}$ for all saturated hereditary subsets $H_1,H_2\subseteq E^0$ and the collection 
$\setof{ \overline{H(\gamma)}\setminus\overline{H(\gamma)\setminus\gamma} }{ \gamma\in\Gamma_E}$ is a partition of $E^0$. 
\end{enumerate}
\end{proposition}

\begin{proposition}[{\cite[Lemmata~3.15 and~3.16 and Proposition~3.18]{MR3759003}}] \label{prop:first-collect-structure-2}
Let $E$ be a graph with finitely many vertices such that every infinite emitter emits infinitely many edges to any vertex it emits any edge to. 

Then $B_H=\emptyset$ for every saturated hereditary subset $H \subseteq E^0$, and the map \fct{\upsilon_E}{\Gamma_E}{\Prime_\gamma(C^*(E))} given for each $\gamma_0\in\Gamma_E$ by
$\upsilon_E(\gamma_0)=\mathfrak{J}_{\overline{H(\cup\eta_{\gamma_0}})},$
where 
$\eta_{\gamma_0}=\Gamma_E\setminus\setof{\gamma\in\Gamma_E}{\gamma\geq\gamma_0},$
is a homeomorphism. 

Let $\omega$ denote the map given by Proposition~\ref{prop:first-collect-structure-1} and \cite[Fact~3.2]{MR3759003}, \ie, 
$\omega(\eta)=\mathfrak{J}_{\overline{H(\cup\eta)}}$
for every hereditary subset $\eta$ of $\Gamma_E$, and let $\varepsilon$ denote the map from $\mathbb{O}(\Prime_\gamma(C^*(E)))$ to $\mathbb{I}_\gamma(C^*(E))$ given in Proposition~\ref{prop:structure-first-prime-gamma}, \ie, 
$\varepsilon(O)=\cap(\Prime_\gamma(C^*(E))\setminus O)$ for every open subset $O\subseteq\Prime_\gamma(C^*(E))$. 
Then we have a commuting diagram 
$$\xymatrix{\mathbb{O}(\Gamma_E)\ar[d]^{\upsilon_E}_\cong\ar[r]_-\cong^-\omega & 
\mathbb{I}_\gamma(C^*(E))\ar@{=}[d] \\ 
\mathbb{O}(\Prime_\gamma(C^*(E)))\ar[r]_-\cong^-\varepsilon
& \mathbb{I}_\gamma(C^*(E))}$$
of lattice isomorphisms. 
\end{proposition}

The following proposition ensures that we can --- up to unital isomorphism --- get every graph \ca in such a form. If we only ask for stable isomorphism, we can even assume that there are no transition states and all cyclic components are singletons. It will be important for us to reduce to these forms. 

\begin{proposition}[{\cite[Lemma~3.17]{MR3759003}}] \label{prop:structure-2}
Let $E$ be a graph with finitely many vertices. 
\begin{enumerate}[(1)]
\item\label{prop:structure-2-circ}
There exists a graph $F$ with finitely many vertices such that every infinite emitter emits infinitely many edges to any vertex it emits any edge to, every transition state has exactly one edge going out, $E\Meq F$, and $C^*(E)\cong C^*(F)$, 
\item\label{prop:structure-2-circcirc}
There exists a graph $F'$ with finitely many vertices, such that every infinite emitter emits infinitely many edges to any vertex it emits any edge to, $F'$ has no transition states, $E\Meq F'$ and $C^*(E)\otimes\K\cong C^*(F')\otimes\K$,
\item\label{prop:structure-2-circcirccirc}
There exists a graph $F''$ with finitely many vertices, such that every infinite emitter emits infinitely many edges to any vertex it emits any edge to, $F''$ has no transition states, all cyclic components are singletons, $E\Meq F''$ and $C^*(E)\otimes\K\cong C^*(F'')\otimes\K$. 
\end{enumerate}
\end{proposition}

\subsection{Block matrices and equivalences}(\cf\ \cite[Section~4.1]{MR3759003}).
In this subsection, we recall some some notation and definitions about rectangular block matrices and equivalences between such. This is a generalization of the definitions in \cite{MR1907894,MR1990568} (in the finite matrix case) to the cases with rectangular diagonal blocks and possibly vacuous blocks as well. This generalization is introduced in \cite{MR3759003}, where it has been defined and discussed, and we use the notation and definitions verbatim as in \cite{MR3759003}. For the convenience of the reader, we will briefly recall which notation and definitions we will use.

We use $\MZ[m\times n]$ for $m\times n$-matrices (equivalently: group homomorphisms from $\Z^n$ to $\Z^m$), and we use \Mplus for the positive $m\times n$ matrices and write $B>0$ whenever $B\in\Mplus$. We let $B(i,j)$ denote the $(i,j)$'th entry of the corresponding matrix, and we let $\gcd B$ be the greatest common divisor of the entries $B(i,j)$, for $i=1,\ldots,m$, $j=1,\ldots,n$, if $B$ is nonzero, and zero otherwise. 

\begin{assumption} \label{ass:preorder}
Let $N\in\N$. 
For the rest of the paper, we let $\calP=\{1,2,\ldots,N\}$ denote a partially ordered set with order $\preceq$ satisfying
$$i\preceq j\Rightarrow i\leq j,$$
for all $i,j\in\calP$, where $\leq$ denotes the usual order on \N. 
We denote the corresponding irreflexive order by $\prec$.
\end{assumption}

We will use multiindices, operations on multiindices, the four sets of block matrices \MZ, \MPZ, \GLP, and \SLP, and operations on block matrices, the embedding (functor) $\iota_\mathbf{r}$, as well as $\GLPEe$ and $\SLPEe$ --- we refer the reader to \cite[Definitions~4.4, 4.6, 4.7, and Remarks~4.5, 4.8]{MR3759003}.

\subsection{\texorpdfstring{$K$}{K}-web and induced isomorphisms}(\cf\ \cite[Section~4.2]{MR3759003}).
\label{sec:K-web-and-induced-isomorphisms}
In what follows we will need the $K$-web, $K(B)$, of a matrix $B \in \MPZ$ and the description of how a \GLPEe	 $\ftn{(U,V)}{B}{B'}$ induces an isomorphism $\ftn{ \kappa_{(U,V)} }{ K(B) }{K(B')}$. 
This is a generalization of the definitions in \cite{MR1990568} (in the finite matrix case) to the cases with rectangular diagonal blocks and possibly vacuous blocks as well. This generalization is introduced in \cite{MR3759003}, where it has been defined and discussed. There concepts like $\cok B$, $\ker B$, the (reduced) $K$-web of $B$, $K(B)$, and $K$-web isomorphisms are introduced. Note that any \GLPEe $\ftn{ (U,V) }{ B }{ B' }$ induces a $K$-web isomorphism from $B$ to $B'$; we denote this induced isomorphism by $\kappa_{(U,V)}$. 
To avoid possible confusion, we point out that $\mathrm{Imm}(i)$ denotes the set of immediate predecessors of $i\in\calP$ (we say that $j$ is an \emph{immediate predecessor of $i$} if $j \prec i$ and there is no $k$ such that $j \prec k \prec i$).  
Also note that we have a canonical identification of the $K$-webs $K(B)$ and $K(-\iota_{\mathbf{r}}(-B))$ by embedding a vector by setting it to be zero on the new coordinates. 

\subsection{Block structure for graphs}(\cf\ \cite[Section~4.3]{MR3759003}).

We will need to work with graphs having their adjacency matrices in very certain block structure. For this we introduced in \cite[Section~4.3]{MR3759003}
the notation $\Bsf_E\in\MPZc$ if $\Bsf_E$ is a matrix indexed over $\{1,2,\ldots,|E^0|\}$  such that it satisfies certain conditions. 
We also define the subsets $\MPZcc$ and $\MPZccc$ of $\MPZc$ there. 
Note that if we have $\Bsf_E\in\MPZc$ and $\Bsf_{E'}\in\MPZc[\mathbf{m}'\times\mathbf{n}']$, then we say that a \starhomo $\Phi$ from $C^*(E)$ to $C^*(E')$ (or from $C^*(E)\otimes \K$ to $C^*(E')\otimes \K$) is \calP-equivariant if $\Phi$ is $\Prime_\gamma(C^*(E))$-equivariant under the canonical identification $\Prime_\gamma(C^*(E))\cong\Gamma_E\cong\calP\cong\Gamma_{E'}\cong \Prime_\gamma(C^*(E'))$ coming from the block structure. 

\subsection{Reduced filtered \texorpdfstring{$K$}{K}-theory, \texorpdfstring{$K$}{K}-web and \texorpdfstring{\GLPEe}{GLP-equivalence}}(\cf\ \cite[Section~4.4]{MR3759003}).
\label{sec:red-filtered-K-theory-K-web-GLP-and-SLP-equivalences}

In \cite{MR3759003}, we interconnect the reduced filtered $K$-theory, the $K$-web and $\GLPEe{s}$ --- we will follow the definitions and notation introduced there.  
For a partially ordered set $(\calP, \preceq )$ that satisfies Assumption~\ref{ass:preorder}, we introduce the partially ordered set $(\calP^\mathsf{T},\preceq^\mathsf{T})$, which really is the set $\calP$ equipped with the opposite order, followed by a permutation to ensure that it satisfies 
Assumption~\ref{ass:preorder}.
We introduce $\mathbf{m}^\mathsf{T}$ and $J_\mathbf{m}$ for a multiindex $\mathbf{m}$.
For $\Bsf_E\in\MPZcc[\mathbf{m}_E \times \mathbf{n}_E]$, we let 
$\mathsf{C}_E=J_{\mathbf{n}_E}\left(\Bsf_{E}^\bullet\right)^\mathsf{T}J_{\mathbf{m}_E}$. 
It was noted that a reduced filtered $K$-theory isomorphism from $\FKR(\calP;C^*(E))$ to $\FKR(\calP;C^*(F))$ corresponds exactly to a (reduced) $K$-web isomorphism from $K(\mathsf{C}_E)$ to 
$K(\mathsf{C}_F)$ together with an isomorphism from 
$\ker(\mathsf{C}_E\{i\})$ 
to $\ker (\mathsf{C}_F\{i\})$ for every $i\in(\calP^{\mathsf{T}})_{\min{}}$, where $\calP_{\min}=\setof{i\in\calP}{ j\preceq i\Rightarrow i=j}$ and $(\calP^\mathsf{T})_{\min}=\setof{i\in\calP}{ j\preceq^\mathsf{T} i\Rightarrow i=j}$.

If $\Bsf_E\in\MPZccc$, then positivity is very easy to describe on the gauge simple subquotients --- and it turns out that this will be all we need (\cf\ Section~\ref{ordereasy}). For components with a vertex supporting at least two distinct return paths, the positive cone is all of $K_0$. For the other components, the ordered $K_0$ is canonically isomorphic to $(\Z,\N_0)$. 

The reduced filtered $K$-theory isomorphism from $\FKR(\calP;C^*(E))$ to $\FKR(\calP;C^*(F))$ induced by a \GLPEe $(U,V)$ from $-\iota_{\mathbf{r}_E}(-\Bsf_E^\bullet)$ to $-\iota_{\mathbf{r}_F}(-\Bsf_F^\bullet)$ will be denoted by $\FKR(U,V)$, while we let $\FKRs(U,V)$ denote corresponding induced reduced filtered $K$-theory isomorphism from $\FKR(\calP;C^*(E)\otimes\K)$ to $\FKR(\calP;C^*(F)\otimes\K)$. 

\section{Maps induced on \texorpdfstring{$\FKR$}{FKR} by edge expansion and matrix moves}
\label{sec:edgeandtoke}

We will show in Section~\ref{sec:boyle} a positive factorization theorem as in \cite[Theorem~4.4]{MR1907894}. 
So if we have a reduced filtered $K$-theory isomorphism that is induced by an \SLPEe (subject to a certain positivity condition), we know that it can be decomposed into basic positive equivalences and using the results in Proposition~\ref{prop:matrix-moves}, we see that the original graphs are move equivalent, and, consequently, the corresponding graph \cas are stably isomorphic. But since we actually want to prove the stronger statement, that there exists an equivariant stable isomorphism inducing the original reduced filtered $K$-theory isomorphism, we need to know that a basic positive equivalence $(U,V)$ induces a stable isomorphism such that the induced map on $\FKR$ is exactly $\FKRs(U,V)$. 
Similarly, we need to know how edge expansion changes $K$-theory when going via stable isomorphisms. 

These results are proved in \cite{MR3624419}, and we state the needed results here for the convenience of the reader. First we define the edge expanded graph. 

\begin{definition}\label{def: edge expansion homomorphism}
Let $E = ( E^0 , E^1 , r_E , s_E)$ be a graph and let $e_0 \in E^1$.  Set $v_0 = s_E ( e_0 )$.  Let $F$ be the \emph{simple expansion graph} at $e_0$ defined by 
\[
F^0 = E^0 \sqcup \{ \tilde{v}_0 \}  \qquad \text{and} \qquad F^1 = \left( E^1 \setminus \{ e_0 \}  \right) \sqcup \{ f_1 , f_2 \}
\]
where $s_F \vert_{ E^1 \setminus \{ e_0 \} } = s_E \vert_{ E^1 \setminus \{ e_0 \} }$, $r_F \vert_{E^1 \setminus \{ e_0 \}} = r_E \vert_{ E^1 \setminus \{ e_0 \} }$, $s_F ( f_1 ) = v_0$, $r_F ( f_1 ) = \tilde{v}_0 = s_F ( f_2 )$, and $r_F ( f_2 )  = r_E( e_0 )$.  
\end{definition}

Note that $E \Meq F$ by applying Move~\RR to $F$.

\begin{definition}
For a subset $S\subseteq\{1,2,\ldots,N\}=\calP$, we let $\mathbf{e}_S$ denote the multiindex $\mathbf{r}=(r_1,r_2,\ldots,r_N)$ satisfying $r_i=1$ for all $i\in S$ and $r_i=0$ for all $i\not\in S$. 
Also we will write $\mathbf{e}_i$ instead of $\mathbf{e}_ {\{i\}}$ when $i\in\calP$. 
\end{definition}

The following two propositions are proved in \cite{MR3624419}. 

\begin{proposition}[{\cf\ \cite[Proposition~5.4]{MR3624419}}] \label{prop:edge-expansion}
Let $E$ be a graph with finitely many vertices, and let $v_0$ be a vertex in $E^0$ that supports a cycle $\mu = \mu_1 \cdots \mu_n$. 
Assume that $\Bsf_E\in\MPZc$, and that $v_0$ belongs to the block $j\in\calP$. 
Let $F$ be the graph in Definition~\ref{def: edge expansion homomorphism} with $e_0 = \mu_1$ (\ie, the graph is obtained with a simple expansion of $e_0$). Then $E\Meq F$.

Moreover, we can consider $\Bsf_{F}$ as an element of $\MPZc[(\mathbf{m}+\mathbf{e}_j)\times(\mathbf{n}+\mathbf{e}_j)]$ and there exist $U\in\SLPZ[\mathbf{m}+\mathbf{e}_j]$ and $V\in\SLPZ[\mathbf{n}+\mathbf{e}_j]$ that are the identity matrices everywhere except for the $j$'th diagonal block, such that $(U,V)$ is an \SLPEe from $-\iota_{\mathbf{e}_j}(-\Bsf_{E}^\bullet)$ to  $\Bsf_{F}^\bullet$, and such that there exists a \calP-equivariant isomorphism $\Phi$ from $C^{*}(E) \otimes \K$ to $C^{*}(F) \otimes \K$ satisfying $\FKR(\calP;\Phi)=\FKRs(U,V)$. 
\end{proposition}

\begin{proposition}[{\cf\ \cite[Proposition~6.1]{MR3624419}}]\label{prop:toke}
Let $E$ be a graph such that $\Bsf_E,\Bsf_F\in\MPZccc$. 
Let $u,v\in E^0$ with $u\neq v$ and $\Asf_{E}(u,v)>0$ and let $E_{u,v}\in\SLPZ[\mathbf{n}]$ denote the positive basic elementary matrix with $E_{u,v}(u,v)=1$ and equal to the identity everywhere else. 

Assume that $E_{u,v}\Bsf_E=\Bsf_{F}$ and the vertex $v$ in $E$ is regular. 
Then $(E_{u,v}^\bullet,I)$ is an \SLPEe from $\Bsf_E^\bullet$ to $\Bsf_{F}^\bullet$, where $E_{u,v}^\bullet$ is the matrix we get from $E_{u,v}$ by removing all columns and rows corresponding to singular vertices of $E$. 
Moreover, there exists a $\calP$-equivariant isomorphism $\Phi$ from $C^*(E)$ to $C^*(F)$ such that $\FKR(\calP;\Phi)=\FKR(E_{u,v}^\bullet,I)$.

Now assume instead that $\Bsf_E E_{u,v}=\Bsf_{F}$ and the vertex $u$ emits at least two edges. 
Then $(I,E_{u,v})$ is an \SLPEe from $\Bsf_E^\bullet$ to $\Bsf_{F}^\bullet$. Moreover, there exists a $\calP$-equivariant isomorphism $\Phi$ from $C^*(E)\otimes\K$ to $C^*(F)\otimes\K$ such that $\FKR(\calP;\Phi)=\FKRs(I,E_{u,v})$. 
\end{proposition}

\nomoreskippable

\section{Canonical and standard forms}
\label{sec:standardform}

In this section, we prove that every graph with finitely many vertices is move equivalent to a graph in canonical form (see Definition~\ref{def:CanonicalForm}).
This will allow us to reduce the proof of our classification result to graphs in canonical form.  
In fact, we will do even better: 
We will reduce the proof of our classification result to graphs whose adjacency matrices are in the same block form.
  
The first results of this type were the results in Proposition~\ref{prop:structure-2}, that allow us to remove breaking vertices and transition states and to make cyclic components $1\times 1$.

\begin{lemma}\label{lem:legalcolumnaddition}
Let $E$ be a graph such that $\Bsf_E\in\MPZccc$, let $u$ and $v$ be distinct vertices of $E$ such that $\Asf_E ( u, v ) > 0$ and $u$ emits at least two edges.  Let $E_{(u,v)}\in\SLPZ[\mathbf{n}]$ denote the positive basic elementary matrix with $E_{(u,v)}(u,v)=1$ and equal to the identity everywhere else and let $F$ be the graph such that $\Bsf_E E_{(u,v)} = \Bsf_{F}$.  Then $\Bsf_{F}$ is an element of $\MPZccc$ if one of the following holds:
\begin{enumerate}[(1)]
\item \label{lem:legalcolumnaddition-item1}
$u$ is the base point of a loop,
\item \label{lem:legalcolumnaddition-item2}
$\Bsf_{E} ( u, v ) \geq 2$, or
\item \label{lem:legalcolumnaddition-item3}
$u$ and $v$ are in the same strongly connected component $\gamma_0$ and $u$ emits at least two edges to vertices in $\gamma_0$.
\end{enumerate} 
\end{lemma}

\begin{proof}
To show that $\Bsf_{F} \in \MPZccc$ it is enough to show that for all $w, w' \in E^0$, there exists a path from $w$ to $w'$ in $E$ if and only if there exists a path from $w$ to $w'$ in $F$.  Note that $F$ can be described as follows:  Take an edge $f$ from $u$ to $v$.  Set $F^0 = E^0$, 
\[
F^1 = \left( E^1 \setminus \{f\} \right) \sqcup \setof{\overline{e}}{e \in r_E^{-1} (u)}
\]
with $s_{F} ( e ) = s_E (e)$, $r_{F} (e) = r_E (e)$ for all $e \in E^1 \setminus \{ f \}$, and $s_{F} ( \overline{e} ) = s_E (e)$ and $r_{F} ( \overline{e} ) = v$.  

Suppose there exists a path $f_1 \cdots f_n$ from $w$ to $w'$ in $F$.  If $f_i = \overline{e_i}$, then $s_{F} ( f_i ) = s_E ( e_i )$ and $r_F ( e_i ) = u$.  Therefore, replacing all the edges $f_i$ that are equal to $\overline{ e_i }$ with $e_i f$, we get a path from $w$ to $w'$ in $E$.   Hence, if there exists a path from $w$ to $w'$ in $F$, then there exists a path from $w$ to $w'$ in $E$.  

We now show that if there exists a path $e_1\cdots e_n$ from $w$ to $w'$ in $E$, then there exists a path from $w$ to $w'$ in $F$.  Suppose $e_i = f$.  If $u$ is the base point of a loop $e$, then we may replace $e_i$ with $\overline{e}$, and if $\Bsf_E ( u, v ) \geq 2$, we may replace $e_i$ with an edge $g \in ( s_E^{-1}(u) \cap r_E^{-1}(v) ) \setminus \{f\}$.  Doing this for all $e_i$ with $e_i = f$, we obtain a path in $F$ from $w$ to $w'$ if \ref{lem:legalcolumnaddition-item1} or \ref{lem:legalcolumnaddition-item2} hold.

Suppose $u$ and $v$ are in the same strongly connected component $\gamma_0$ and $u$ emits at least two edges to vertices in $\gamma_0$.  We may assume that $u$ is not the base point of a loop and $\Bsf_E (u,v) = 1$.  Otherwise, the previous cases will imply that there exists a path from $w$ to $w'$ in $F$.  Suppose $e_i = f$.  If $i \neq 1$, then $r_E ( e_{i-1} ) = u$.  Replace $e_{i-1} e_i$ with $\overline{e_{i-1}}$.  Suppose $e_1 = f$.  By assumption there exists $v' \in \gamma_0 \setminus \{ v, u \}$ and an edge $g \in s^{-1}(u) \cap r^{-1}(v')$.  Since $v', v \in \gamma_0$, there exists a path $\mu_1 \cdots \mu_m$ in $E$ from $v'$ to $v$.  If $\mu_j = f$ (note that $j \neq 1$), then $r_E( \mu_{j-1} ) = u$.  So, replacing $\mu_{j-1} \mu_j$ with $\overline{ \mu_{j-1} }$ for all $\mu_j = f$, we get a path $\nu$ from $v'$ to $v$ in $F$.  Then $g \nu$ is a path from $u$ to $v$ in $F$.  So replacing $e_1$ with $g \nu$ if $e_1 = f$ and replacing $e_{i-1} e_i$ with $\overline{ e_{i-1} }$ whenever $i \neq 1$ and $e_i = f$, we obtain a path from $w$ to $w'$ in $F$. 
\end{proof}

The following result allows us to make sure that all vertices that support a cycle support at least one loop. It also brings the adjacency matrix in a form such that every entry that possibly could be nonzero is nonzero. 

\begin{lemma}\label{lem:GoingFromEdgeExpansionToCanonicalForm}
Let $E$ be a graph such that $\Bsf_E\in\MPZccc$. 

Then $E \Meq F$, where $F$ is a graph such that $F^0=E^0$, $F^0_{\mathrm{reg}}=E^0_{\mathrm{reg}}$, $\Bsf_{F}\in\MPZccc$, $\Bsf_F^\bullet\{i\}>0$ for all $i\in\calP$ satisfying $m_i > 1$, all singular vertices $v$ satisfy the property that if there exists a path of positive length from $v$ to $w$, then $|s^{-1}(v) \cap r^{-1} (w) | = \infty$, and $\Bsf_F^\bullet\{i,j\}>0$ whenever $i,j\in\calP$ with $i\neq j$, $i\preceq j$ and $m_i\neq 0$. 

Moreover, the move equivalence we get from $E$ to $F$ can be obtained by using \SLPEe{s} of the form in the second part of Proposition~\ref{prop:toke}. This means that we get an \SLPEe $(I,V)$ from $\Bsf_E$ to $\Bsf_F$ and a $\calP$-equivariant isomorphism $\Psi$ from $C^*(E)\otimes\K$ to $C^*(F)\otimes\K$ such that $\FKR(\calP;\Psi)=\FKRs(I,V)$. 
\end{lemma}
\begin{proof}
We will prove the lemma by doing column additions as in the second part of Proposition~\ref{prop:toke}, and then use this proposition to get the results. 
We first want to show that we can get every vertex in a component corresponding to a diagonal block with $m_i>1$ to have at least two loops. 
These vertices correspond exactly to the vertices in the strongly connected components with at least two elements. 
So let there be given $\gamma\in\Gamma_E$ with $|\gamma|>1$. 

The first step is to show that we can get at least one vertex in $\gamma$ with at least one loop. 
Since $\gamma$ is strongly connected, but not a cyclic component, there exists a vertex $u\in\gamma$ that emits at least two edges to vertices in $\gamma$. 
If $u$ is the base of a loop, we are done. 
If not, then there exists a vertex $v\in\gamma\setminus\{u\}$ such that there is an edge from $u$ to $v$. 
Choose a path from $v$ to $u$ going through distinct vertices $v=v_0$, $v_1,\ldots,v_n=u$. 
Then we can add the column corresponding to $v_n$ to the column corresponding to $v$. Now $v_{n-1}$ will emit at least two edges (one to $v_n$ and one to $v$). 
Then we can add the column corresponding to $v_{n-1}$ to the column corresponding to $v$. Now $v_{n-2}$ will emit at least two edges (one to $v_{n-1}$ and one to $v$). 
By induction we can do column additions in this way within the component to get at least one loop based at $v$. Note that by Lemma~\ref{lem:legalcolumnaddition} the resulting graph gives a matrix in $\MPZccc$ in each step.

So now we have a vertex in $v_0'=v\in\gamma$ with at least one loop. 
Since $\gamma$ is strongly connected, we have that $v$ must emit at least two edges and there is at least one vertex $v_1'\in\gamma\setminus\{v_0'\}$ that emits an edge to $v_0'$. 
Now, using column addition, we can make sure that $v_1'$ emits at least two edges. 
As above, we can use column additions to get at least one loop at $v_1'$ --- and $v_0'$ will still have at least one loop. 
If $\{v_0',v_1'\}\neq\gamma$, then there exists at least one vertex $v_2'\in\gamma\setminus\{v_0',v_1'\}$ that emits an edge to $\{v_0',v_1'\}$. 
Again, using column addition, we can make sure that $v_2'$ emits at least two edges. 
As above, we can use column additions to get at least one loop at $v_2'$ --- and $v_0'$ and $v_1'$ will still have at least one loop. 
By induction, we get that by using column additions within the component, every vertex has at least one loop and it emits at least one edge to another vertex of the component. 
Again, by Lemma~\ref{lem:legalcolumnaddition} the resulting graph gives a matrix in $\MPZccc$ in each step.  

Suppose $u$ and $v$ are vertices in $\gamma$ and there is no edge from $u$ to $v$.  Choose a path from $u$ to $v$ through distinct vertices $u = u_0$, $u_1 \ldots, u_n = v$.  Then we can add the column corresponding to $u_{n-1}$ to the column corresponding to $u_n = v$. 
Since $u_{n-1}$ supports a loop, for any edge from a vertex $w$ to a vertex $w'$ in the graph before the column addition there will still be an edge from $w$ to $w'$ after the column addition. Moreover, now there is an edge from $u_{n-2}$ to $u_n=v$.
By induction we can do column additions in this way within the component to get an edge from $u$ to $v$ (and in the new graph, there will be an edge from $w$ to $w'$ if there was an edge in the graph before the operations).  
Continuing this, we get that for all $u,v \in \gamma$, there is an edge from $u$ to $v$.  
Given $v\in\gamma$, we can choose $v'\in\gamma\setminus\{v\}$. 
By adding the column corresponding to $v'$ to the column corresponding to $v$, we get that $v$ supports at least two loops with all the previous properties preserved. 
Continuing this for all vertices in $\gamma$, we get that every vertex in $\gamma$ is a base point of at least two loops and for all $u,v \in \gamma$, there exists an edge from $u$ to $v$.  We do this for all strongly connected components with at least two elements.  Call the resulting graph $E'$.  By Lemma~\ref{lem:legalcolumnaddition}, $\Bsf_{E'} \in \MPZccc$ in each step.  

Now assume that $u,v\in (E')^0$ such that $u\geq v$, $u\neq v$. 
If $\Bsf_{E'}(u,v)>0$, then we do nothing, so assume that $\Bsf_{E'}(u,v)=0$. 
Recall that every regular vertex of $E'$ has at least one loop, so the only vertices not emitting at least two edges are the sinks and the vertices in cyclic components that do not emit edges to other components. 
Therefore, there exists a path from $u$ to $v$ through distinct vertices $u=u_0$, $u_1,\ldots,u_n=v$ that all emit at least two edges (except for possibly the last vertex, $v$) with $n\geq 2$ and $u_i$ being the base point of a loop or $\Bsf_{E'} ( u_i , u_{i+1} ) = \infty$ (when $u_i$ is an infinite emitter that is not a base point of a cycle), for $i=0,1,\ldots,n-1$.  Using Proposition~\ref{prop:toke}, we can add the column corresponding to $u_{n-1}$ to the column corresponding to $v$. Now $u_{n-2}$ will emit at least two edges (one to $u_{n-1}$ and one to $v$), and all other obtained properties are preserved.  Then we can add the column corresponding to $u_{n-2}$ to the column corresponding to $v$. Now $u_{n-3}$ will emit at least two edges (one to $u_{n-2}$ and one to $v$), and all other obtained properties are preserved.  By induction, we get that by using column additions through this path, we get a graph $E''$ such that $\Bsf_{E''} ( u, v ) > 0$, and such that all other obtained properties are preserved.  And by Lemma~\ref{lem:legalcolumnaddition}, $\Bsf_{E''} \in \MPZccc$ in each step.

Let $F$ be the resulting graph where we have done column additions to get that $\Bsf_F ( u, v ) > 0$ for all $u \geq v$ with $u\neq v$.  
The original graph satisfied, that every infinite emitter emits infinitely many edges to every vertex it emits any edge to. This property is preserved when we do the above column additions. Therefore the resulting graph satisfies this, and it even 
satisfies that all singular vertices $v$ satisfy the property that if there exists a path of positive length from $v$ to $w$, then $|s^{-1}(v) \cap r^{-1} (w) | = \infty$. 
By Lemma~\ref{lem:legalcolumnaddition}, $\Bsf_{F} \in \MPZccc$ in each step.
\end{proof}

We now expand on the conditions we can put on graphs. 
To turn $K$-theory isomorphisms into \GLPEe{s} or \SLPEe{s}, the matrices $\Bsf^{\bullet}_E$ and $\Bsf^{\bullet}_{E'}$ must have sufficiently big diagonal blocks, this requirement is captured in \ref{def:canonical-item-size3} and \ref{def:canonical-item-rowrank} below. 
The positivity condition, \ref{def:canonical-item-positive}, is also convenient when dealing with matrix manipulations. 
Condition \ref{def:canonical-item-paths} and \ref{def:canonical-item-positive} ensures that we can apply Proposition~\ref{prop:matrix-moves} to do matrix manipulations.
Condition~\ref{def:canonical-item-singularfirst} ensures us that for graphs $E$ and $F$ with $\Bsf_E,\Bsf_F\in\MPZccc$ in canonical form, we have that $\Bsf_E^\bullet=\Bsf_F^\bullet$ implies that $\Bsf_E=\Bsf_F$, so the graphs are equal (up to relabelling of the vertices and edges). 

\begin{definition}\label{def:CanonicalForm}
Let $E$ be a given graph.
We say that $E$ is in \emph{canonical form} if $\Bsf_E\in\MPZccc$ and it satisfies the following
\begin{enumerate}[(1)]
\item \label{def:canonical-item-loops-inf} 
every vertex of $E$ is either a regular vertex that is the base point of a loop or a singular vertex $v$ satisfying the property that if there exists a path of positive length from $v$ to $w$, then $| s^{-1}(v) \cap r^{-1} (w) | = \infty$;
\item \label{def:canonical-item-paths} for all regular vertices $v, w$ of $E$ with $v \geq w$, there exists a path in $E$ from $v$ to $w$ through regular vertices in $E$;
\item \label{def:canonical-item-size3} $m_{i} \geq 3$ whenever $i$ corresponds to a strongly connected component that is not cyclic; 
\item \label{def:canonical-item-positive} if $i \preceq j$, $\Bsf^{\bullet}_{E} \{ i, j \}$ is not the empty matrix and either $i\neq j$ or $m_i>1$, then $\Bsf^{\bullet}_{E} \{ i , j \} > 0$; 
\item \label{def:canonical-item-rowrank}  if $m_i>1$ for an $i\in\calP$, then the Smith normal form of $\Bsf^{\bullet}_{E} \{ i \}$ has at least two $1$'s (for rectangular matrices, see Theorem~\ref{thm:smith-normal-form} for the definition of Smith normal form) --- in particular, the rank of $\Bsf^{\bullet}_{E} \{ i \}$ is at least $2$; and
\item \label{def:canonical-item-singularfirst}
the vertices are ordered such that the singular vertices come before the regular vertices in each component (when writing down the adjacency matrix, $\Asf_E$).
\end{enumerate}
\end{definition}

\begin{lemma}\label{lem:we-can-put-it-in-canonical-form}
Let $E$ be a graph with finitely many vertices. 
Then there exists a graph $F$ such that $F$ is in canonical form and $E\Meq F$.

If $E$ is a graph with $\Bsf_E\in\MPZccc$ satisfying Conditions~\ref{def:canonical-item-size3}, \ref{def:canonical-item-rowrank} and~\ref{def:canonical-item-singularfirst} in Definition~\ref{def:CanonicalForm}, 
then there exist a graph $F$ in canonical form with $E\Meq F$, $\Bsf_{F}\in \MPZccc$ and a \calP-equivariant isomorphism $\Phi$ from $C^*(E)\otimes\K$ to $C^*(F)\otimes\K$ such that $\FKR( \Phi )$ is induced by an \SLPEe from $\Bsf_E^\bullet$ to $\Bsf_{F}^\bullet$ (via the canonical isomorphisms).

If $E$ is a graph with $\Bsf_E\in\MPZccc$, 
then there exist a graph $F$ in canonical form with $E\Meq F$, $\Bsf_{F}\in \MPZccc[(\mathbf{m}+\mathbf{r})\times(\mathbf{n}+\mathbf{r})]$ for some multiindex $\mathbf{r}$, and a \calP-equivariant isomorphism $\Phi$ from $C^*(E)\otimes\K$ to $C^*(F)\otimes\K$ such that $\FKR(\Phi)$ is induced by an \SLPEe from $-\iota_\mathbf{r}(-\Bsf_E^\bullet)$ to $\Bsf_{F}^\bullet$ (via the canonical isomorphisms).

Now assume that $E$ is a graph in canonical form with $\Bsf_E\in\MPZccc$, and let $\mathbf{r}\in\N_0^N$ be a multiindex such that $r_i=0$ for all $i\in\calP$ with $n_i=1$. 
Then there exist a graph $F$ in canonical form with $E\Meq F$, $\Bsf_{F}\in \MPZccc[(\mathbf{m}+\mathbf{r})\times(\mathbf{n}+\mathbf{r})]$ and a \calP-equivariant isomorphism $\Phi$ from $C^*(E)\otimes\K$ to $C^*(F)\otimes\K$ such that $\FKR(\Phi)$ is induced by an \SLPEe from $-\iota_\mathbf{r}(-\Bsf_E^\bullet)$ to $\Bsf_{F}^\bullet$ (via the canonical isomorphisms).
\end{lemma}
\begin{proof}
We will find a graph $F$ in canonical form such that $E\Meq F$ as follows.
We will find a series of graphs that are move equivalent to $E$ and satisfy more and more conditions of the canonical form. 
We can use Proposition~\ref{prop:structure-2} to get a graph $E'$ with finitely many vertices such that $E\Meq E'$ and $\Bsf_{E'}\in\MPZccc$ for a suitable $\calP$. 
Now we can do inverse reduction moves as in Proposition~\ref{prop:edge-expansion} up to two times for each noncyclic strongly connected component to make sure that the Conditions~\ref{def:canonical-item-size3} and~\ref{def:canonical-item-rowrank} are satisfied. 
If necessary, we permute the vertices within each component to obtain Condition~\ref{def:canonical-item-singularfirst} as well. 
Now we can use Lemma~\ref{lem:GoingFromEdgeExpansionToCanonicalForm} to get Conditions~\ref{def:canonical-item-loops-inf} and~\ref{def:canonical-item-positive} satisfied. Note that the other conditions are still satisfied. 
Call the resulting graph $F$. 
Condition~\ref{def:canonical-item-paths} follows from Condition~\ref{def:canonical-item-positive}.
Thus $F\Meq E$ and $F$ is in canonical form. 

If $E$ is a graph with $\Bsf_E\in\MPZccc$, 
then it follows from Proposition~\ref{prop:edge-expansion} and Lemma~\ref{lem:GoingFromEdgeExpansionToCanonicalForm} there exist a graph $F$ in canonical form with $E\Meq F$, $\Bsf_{F}\in \MPZccc[(\mathbf{m}+\mathbf{r})\times (\mathbf{m}+\mathbf{r})]$, for some multiindex $\mathbf{r}$, and a \calP-equivariant isomorphism $\Phi$ from $C^*(E)\otimes\K$ to $C^*(F)\otimes\K$ such that $\FKR(\Phi)$ is induced by an \SLPEe from $-\iota_\mathbf{r}(-\Bsf_E^\bullet)$ to $\Bsf_{F}^\bullet$ (via the canonical isomorphisms).
If $E$ already satisfies Conditions~\ref{def:canonical-item-size3}, \ref{def:canonical-item-rowrank} and~\ref{def:canonical-item-singularfirst}, then we can choose all entries of $\mathbf{r}$ to be zero.

Now assume that $E$ is a graph in canonical form with $\Bsf_E\in\MPZccc$, and let $\mathbf{r}\in\N_0^N$ be a multiindex such that $r_i=0$ for all $i\in\calP$ with $n_i=1$. 
Then it follows from Proposition~\ref{prop:edge-expansion} that we have a graph $E'$ that is obtained by a series of simple edge expansion moves that satisfies that $E\Meq E'$, $\Bsf_{E'}\in\MPZccc[(\mathbf{m}+\mathbf{r})\times(\mathbf{n}+\mathbf{r})]$, and that there exists a \calP-equivariant isomorphism $\Phi$ from $C^*(E)\otimes\K$ to $C^*(E')\otimes\K$ such that $\FKR( \Phi )$ is induced by an \SLPEe from $-\iota_\mathbf{r}(-\Bsf_E^\bullet)$ to $\Bsf_{E'}^\bullet$ (via the canonical isomorphisms). 
It now follows from Lemma~\ref{lem:GoingFromEdgeExpansionToCanonicalForm} that we can get a graph $F$ in canonical form with $E'\Meq F$, $\Bsf_{F}\in \MPZccc[(\mathbf{m}+\mathbf{r})\times(\mathbf{n}+\mathbf{r})]$, such that we have a \calP-equivariant isomorphism $\Psi$ from $C^*(E')\otimes\K$ to $C^*(F)\otimes\K$ such that $\FKR(\Psi)$ is induced by an \SLPEe from $\Bsf_{E'}^\bullet$ to $\Bsf_{F}^\bullet$ (via the canonical isomorphisms).
\end{proof}

\begin{definition}[{\cite[Definition~4.22]{MR3759003}}]\label{def:StandardForm}
Let $E$ and $F$ be graphs with finitely many vertices. We say that $(\Bsf_E,\Bsf_F)$ is in \emph{standard form} if $E$ and $F$ are in canonical form and $\Bsf_E,\Bsf_F\in\MPZccc$ for some multiindices $\mathbf{m}$ and $\mathbf{n}$.  This means that the adjacency matrices have exactly the same sizes and block structures.
\end{definition}

In \cite[Definition~4.22]{MR3759003}, the \emph{temperatures} of $E$ and $F$ is used to define when a pair $(\Bsf_E,\Bsf_F)$ is in standard form.  Since we have included the assumption that $E$ and $F$ are in canonical form, the temperature condition in \cite[Definition~4.22]{MR3759003} follows from the fact that $\Bsf_E,\Bsf_F\in\MPZccc$.

The notion of a standard form is of course only useful if we can assume that our graphs have the standard form; the next proposition shows that we can indeed assume this, if the corresponding \cas have isomorphic ordered reduced filtered $K$-theory. 

\begin{proposition}\label{prop:standard-form}
Let $E_{1}$ and $E_{2}$ be graphs with finitely many vertices. 
Set $X=\Prime_\gamma(C^*(E_1))$, assume that we have a homeomorphism from $\Prime_\gamma(C^*(E_1))$ to $\Prime_\gamma(C^*(E_2))$ and view $C^*(E_1)$ and $C^*(E_2)$ as $X$-algebras under this homeomorphism. 
Assume, moreover, that we have an isomorphism $\varphi\colon\FKRplus(X;C^*(E_1))\rightarrow\FKRplus(X;C^*(E_2))$. 

Then there exists a pair of graphs $(F_1,F_2)$ with finitely many vertices such that 
\begin{itemize}
\item
$E_i\Meq F_i$, for $i=1,2$, 
\item 
$\Bsf_{F_1},\Bsf_{F_2}\in\MPZccc$, for some multiindices $\mathbf{m},\mathbf{n}$ and a poset \calP satisfying Assumption~\ref{ass:preorder},

\item 
$(\Bsf_{F_1},\Bsf_{F_2})$ is in standard form, 
 
\item
there is an equivariant isomorphism $\Phi_i$ from $C^{*}(E_i)\otimes\K$ to $C^{*}(F_i)\otimes\K$, for $i=1,2$,
\item 
the order reversing isomorphisms between \calP and $\Gamma_{F_1}$ and between \calP and $\Gamma_{F_2}$ corresponding to the chosen block structures $\Bsf_{F_1},\Bsf_{F_2}\in\MPZccc$ is compatible with the chosen homeomorphism between $\Prime_\gamma(C^{*}(E_1))$ and $\Prime_\gamma(C^{*}(E_2))$ via the isomorphisms $\Phi_1$ and $\Phi_2$.
\end{itemize}
\end{proposition}

\begin{proof}
It follows from Lemma~\ref{lem:we-can-put-it-in-canonical-form} that we can find graphs $G_1, G_2$ such that $G_i \Meq E_i$ and $G_i$ are in canonical form, $i = 1,2$. 
Therefore we have \starisos $\Phi_i$ from $C^*(E_i)\otimes\K$ to $C^*(G_i)\otimes\K$, for $i=1,2$. 
Write $\Bsf_{G_1}\in\mathfrak{M}_{\calP_1}^{\circ\circ\circ}(\mathbf{m}_{G_1}\times\mathbf{n}_{G_1},\Z)$ and $\Bsf_{G_2}\in\mathfrak{M}_{\calP_2}^{\circ\circ\circ}(\mathbf{m}_{G_2}\times\mathbf{n}_{G_2},\Z)$ for multiindices $\mathbf{m}_{G_i}$, $\mathbf{n}_{G_i}$ and a poset $\calP_i$ satisfying Assumption~\ref{ass:preorder}, for $i=1,2$. 
The isomorphisms $\Phi_1$, $\Phi_2$ together with the homeomorphism between $\Prime_\gamma(C^*(E_1))$ to $\Prime_\gamma(C^*(E_2))$ gives a canonical isomorphism between $\calP_1$ and $\calP_2$. 
Thus we can without loss of generality assume that $\calP=\calP_1=\calP_2$ and that the above mentioned canonically induced isomorphism from $\calP_1$ to $\calP_2$ is the identity. 

Note that $n_{G_1,i}$ is the number of elements in the $i$'th component of $G_1$, while $n_{G_1,i}-m_{G_1,i}$ is the number of singular vertices in the $i$'th component of $G_1$ --- and similarly for $G_2$. 
For each $\gamma\in\Gamma_{G_i}$, the number of singular vertices in $\gamma$ is the rank of $K_0(\mathfrak{J}_{\overline{H(\gamma)}}^{C^*(G_i)} / \mathfrak{J}_{\overline{H(\gamma)\setminus \gamma}}^{C^*(G_i)})$ minus the rank of $K_1(\mathfrak{J}_{\overline{H(\gamma)}}^{C^*(G_i)} / \mathfrak{J}_{\overline{H(\gamma)\setminus \gamma}}^{C^*(G_i)})$, for $i=1,2$ (\cf\ Proposition~\ref{prop:first-collect-structure-1}(3)). 
Therefore, we get from the isomorphism $\varphi$ that $n_{G_1,i}-m_{G_1,i}=n_{G_2,i}-m_{G_2,i}$, for all $i\in\calP$. 
Also, since the isomorphism $\varphi$ is positive on these $K_0$-groups, it follows that $n_{G_1,i}=1$ if and only if $n_{G_2,i}=1$, for all $i\in\calP$. 

The only potential problem now is that we may not have $n_{G_1,i} = n_{G_2, i}$ when these are not $1$. 
Define the multiindices $\mathbf{m}=(m_i)_{i\in\calP}$ and $\mathbf{n}=(n_i)_{i\in\calP}$ by $m_i=\max(m_{G_1,i},m_{G_2,i})$ and $n_i=\max(n_{G_1,i},n_{G_2,i})$ for each $i\in\calP$.

Now it follows from 
Lemma~\ref{lem:we-can-put-it-in-canonical-form} that we get graphs $F_1$ and $F_2$ with finitely many vertices such that $G_i\Meq F_i$, for $i=1,2$, $\Bsf_{F_1},\Bsf_{F_2}\in\MPZccc$, $(\Bsf_{F_1}, \Bsf_{F_2} )$ is in standard form, there is a \calP-equivariant isomorphism $\Phi_i'$ between $C^{*}(G_i)\otimes\K$ and $C^{*}(F_i)\otimes\K$, for $i=1,2$.
\end{proof}

When $E$ is in canonical form, the rows of $\Bsf_E$ that are removed to form $\Bsf^{\bullet}_E$ either have all entries equal to $0$ except on the diagonal, which is $-1$ (this is the case where the corresponding vertex is a sink), or it only contains $0$ and $\infty$ except on the diagonal, which is either $-1$ or $\infty$ (this is the case where the corresponding vertex is an infinite emitter). 
It therefore follows from Proposition~\ref{prop:matrix-moves} that adding one column in $\Bsf^{\bullet}_E$ into another will preserve move equivalence, so long as it maintains the block structure and similarly for rows. 
Hence we have:

\begin{corollary} \label{cor:rc-in-bullet-old}
Let $E$ be a graph with finitely many vertices and suppose that $E$ is in canonical form. 
In $\Bsf^{\bullet}_E$ we can add column $l$ into column $k$ without changing the move equivalence class of the associated graph if the diagonal entry of column $l$ is in block $i$, the diagonal entry of column $k$ is in block $j$ and $i \preceq j$. 
Similarly we can add row $l$ into row $k$ without changing the move equivalence class if the diagonal entry of row $l$ is in block $i$, the diagonal entry of row $k$ is in block $j$ and $j \preceq i$. 
\end{corollary}

\section{Generalizing the Boyle-Huang lifting result}
\label{sec:boylehuang}
We aim to prove Theorem~\ref{thm:mainBH} below, which says that --- in certain cases --- every $K$-web isomorphism is induced by a \GLPEe. This is the main result of this section. 
To prove Theorem~\ref{thm:mainBH}, we first strengthen \cite[Theorem~4.5]{MR1990568}.
We recall the concept of Smith normal form for rectangular matrices, \cf\ \cite[Section~II.15]{MR0340283}. 
\begin{theorem}[Smith normal form]
\label{thm:smith-normal-form}
Suppose $B$ is an $m\times n$ matrix over \Z with $m,n\in\N$. 
Then there exist matrices $U\in\GLZ[m]$ and $V\in\GLZ$ such that the matrix 
$D=UBV$ satisfies the following
\begin{itemize}
\item $D(i,j)=0$ for all $i\neq j$,
\item the $\min(m,n)\times\min(m,n)$ principal submatrix of $D$ 
is a diagonal matrix $$\operatorname{diag}(d_1,d_2,\ldots,d_r,0,0,\ldots,0),$$
where $r\in\{0,1,\ldots,\min(m,n)\}$ is the rank of $B$ and $d_1,d_2,\ldots,d_r$ are positive integers such that $d_i|d_{i+1}$ for $i=1,\ldots,r-1$. 
\end{itemize}
For each matrix $B$, the matrix $D$ is unique and is called the \emph{Smith normal form} of $B$. 
\end{theorem}

\begin{lemma}\label{lem:gcd}
Let $B$ be an $m\times n$ matrix over \Z with $m,n\in\N$, and let $U\in\GLZ[m]$ and $V\in\GLZ[n]$ be given invertible matrices. 
Then $\gcd B=\gcd (UBV)$. In particular, if $D$ is the Smith normal form of $B$, then $\gcd B=D(1,1)=d_1$, whenever $B\neq 0$. 
\end{lemma}
\begin{proof}
We may assume that $B\neq 0$. 
Let $d$ be a positive integer. Then
\begin{align*}
d\text{ divides all entries of }B & \Leftrightarrow \forall i,j\colon d\,|\,e_i^\mathsf{T}Be_j \\
& \Leftrightarrow \forall x\in\Z^m,y\in\Z^n\colon d\,|\,x^\mathsf{T}By \\
& \Leftrightarrow \forall x\in\Z^m,y\in\Z^n\colon d\,|\,x^\mathsf{T}UBVy \\
& \Leftrightarrow \forall i,j\colon d\,|\,e_i^\mathsf{T}UBVe_j \\
& \Leftrightarrow d\text{ divides all entries of }UBV.
\end{align*}
Now the lemma follows.
\end{proof}

\begin{remark}
Let $B$ be an $m\times n$ matrix over \Z with $m,n\in\N$. 
Then it follows from the above, that $m$ is greater than the number of generators of $\cok B$ according to the decomposition from the Smith normal form into direct sums of nonzero cyclic groups (the invariant factor decomposition) if and only if $\gcd B=1$. 
\end{remark}

\begin{remark}
Boyle and Huang show in their paper \cite{MR1990568} a fundamental theorem about lifting automorphisms of cokernels to \GL-equivalences respectively \SL-equivalences (\cf\ \cite[Theorem~4.4]{MR1990568}). 

As we will see, it is possible to generalize the part about \GL-allowance in this theorem to rectangular matrices (with $\mathcal{R}=\Z$), but the analogous statement to the part about \SL-allowance in \cite[Theorem~4.4]{MR1990568} does not hold in general (for rectangular matrices). 
We get a counterexample if we consider the matrix
$$B=\begin{pmatrix}
3 \\ 0
\end{pmatrix}$$ 
and the automorphism $-\id$ on $\cok B\cong \Z/3\oplus\Z$ induced by the matrix 
$$M=\begin{pmatrix}
-1 & 0 \\ 0 & -1
\end{pmatrix}.$$

Although we believe it can be done, we do not investigate this further, since for our purposes we do not need to know when automorphisms can be lifted to \SL-equivalences. 
\end{remark}

Recall that 
$$\calP_{\min}=\setof{i\in\calP}{ j\preceq i\Rightarrow i=j}.$$

In \cite[Corollary~4.3]{MR3624419} there is the following strengthening of \cite[Corollary~4.7]{MR1990568}.

\begin{corollary} \label{cor:BH-4.6-B}
Let $\mathbf{n}=(n_i)_{i\in\calP}$ be a multiindex with $n_i\neq 0$, for all $i\in\calP$. 
Suppose $B$ and $B'$ are matrices in $\MPZ[\mathbf{n}]$ with $\gcd B\{i\}=\gcd B'\{i\}\in\{0,1\}$ for all $i\in\calP$.
Then for any $K$-web isomorphism $\kappa\colon K(B)\rightarrow K(B')$ together with isomorphisms $\psi_i\colon\ker B\{i\}\rightarrow\ker B'\{i\}$, for $i\in\calP_{\min}$, 
there exist matrices $U,V\in\GLPZ[\mathbf{n}]$ such that we have a \GLPEe $(U, V )\colon B\rightarrow B'$ satisfying $\kappa_{(U,V )} = \kappa$ and $V^{-1}\{i\}$ induces $\psi_i$ for each $i\in\calP_{\min}$.
\end{corollary}

The following theorem is the main result of this section, and allows us --- in certain cases --- to lift $K$-web isomorphisms to \GLPEe{s} for rectangular cases. 
Although it is possible to prove this directly, imitating the proof in \cite{MR1990568}, the present proof is much shorter since it reduces the rectangular case to the square case so that the results from \cite{MR1990568} apply. 

\begin{theorem}\label{thm:mainBH}
Let $\mathbf{m}=(m_i)_{i\in\calP},\mathbf{n}=(n_i)_{i\in\calP}\in \N_0^N$ be multiindices. 
Suppose $B$ and $B'$ are matrices in \MPZ with $\gcd B\{i\}=\gcd B'\{i\}\in\{0,1\}$ for all $i\in\calP$ with $m_i\neq 0$ and $n_i\neq 0$. 

Then for any $K$-web isomorphism $\kappa\colon K(B)\rightarrow K(B')$ there exist matrices $U\in\GLPZ[\mathbf{m}]$ and $V\in\GLPZ[\mathbf{n}]$ such that we have a \GLPEe $(U, V )\colon B\rightarrow B'$ satisfying $\kappa_{(U,V )} = \kappa$. 

If, moreover, we are given an isomorphism $\psi_i\colon\ker B\{i\}\to\ker B'\{i\}$, for every $i\in\calP_{\min}$, then we can choose the above \GLPEe $(U, V )$ such that --- in addition to the above --- also $V^{-1}\{i\}$ induces the $\psi_i$, for all $i\in\calP_{\min}$.
\end{theorem}
\begin{proof}
For each $i\in\calP$, choose $U_i,U_i'\in\GLZ[m_i]$ and $V_i,V_i'\in\GLZ[n_i]$ such that $D_i=U_iBV_i$ and $D_i'=U_i'B'V_i'$ are the Smith normal forms of $B$ and $B'$, respectively (\cf\ Theorem~\ref{thm:smith-normal-form}). 
Let $U,U'\in\GLPZ[\mathbf{m}]$ and $V,V'\in\GLPZ[\mathbf{n}]$ be the block diagonal matrices with $U_i,U_i',V_i$ and $V_i'$ in the diagonals, respectively.

Then $UBV$ and $U'B'V'$ are in $\MPZ$ and $(U,V)\colon B\rightarrow UBV$ and $(U,V)\colon B'\rightarrow U'B'V'$ are \GLPEe{s} inducing $K$-web isomorphisms $\kappa_{(U,V)}$ from $K(B)$ to $K(UBV)$ and $\kappa_{(U',V')}$ from $K(B')$ to $K(U'B'V')$, respectively. 
Moreover, Lemma~\ref{lem:gcd} ensures that we still have $\gcd(UBV)\{i\}= \gcd(U'B'V')\{i\} \in \{0,1\}$. 
Thus we can without loss of generality assume that each diagonal block is equal to its Smith normal form. 
Also note that because we have a $K$-web isomorphism from $K(B)$ to $K(B')$, now the diagonal blocks are necessarily identical. 

Let $\mathbf{r}$ be such that $r_i=\max(m_i,n_i)$ for all $i\in\calP$. 
Let, moreover, $C,C'\in\MPZ[\mathbf{r}]$ denote the matrices $B$ and $B'$ enlarged by putting zeros outside the original matrices. 
Define $\mathbf{r}^c$ and $\mathbf{r}^k$ by $r_i^c=\max(n_i-m_i,0)$ and $r_i^k=\max(m_i-n_i,0)$ for all $i\in\calP$. 
In the (reduced) $K$-web we are considering the modules $C_c(B)=\cok B(c)$ where $c$ is $\{i\}$, $\setof{j\in\calP}{j\prec i}\neq \emptyset$ or $\setof{j\in\calP }{j\preceq i}$ for $i\in\calP$ --- and similarly for $B'$. 
It is clear that when we consider $C$ and $C'$ we just add onto these cokernels 
$$\bigoplus_{j\in c}\Z^{r_j^c},$$
and that the maps between the modules are the obvious ones. 
Similarly for the modules $K_d(B)=\ker B(d)$ where $d$ is $\{i\}$ where $\setof{j\in\calP }{ j\prec i}\neq \emptyset$ --- and similarly for $B'$. 
It is also clear that when we consider $C$ and $C'$ we just add onto these kernels 
$$\Z^{r_i^k},$$
where $d=\{i\}$. And the connecting homomorphisms will be the zero maps. 

Thus we can extend the isomorphism $\kappa$ to an isomorphism $\widetilde{\kappa}\colon K(C)\rightarrow K(C')$ by setting it to be the identity on the new groups. 
By Corollary~\ref{cor:BH-4.6-B}, we see that there exist matrices $U,V\in\GLPZ[\mathbf{r}]$ such that we have a \GLPEe $(U, V )\colon C\rightarrow C'$ satisfying $\kappa_{(U,V )} = \widetilde{\kappa}$ and that, moreover, $(U,V)$ induces $\psi_i$ plus the identity on the new summands of $\ker B\{i\}$ for $i\in\calP_{\min}$. 

Now let us look at the $i$'th diagonal block. We now want to cut $U$ and $V$ down to match the original structure. Naturally there are three cases to consider. The first one, $m_i=n_i$ is trivial. 

Now consider the case $m_i<n_i$. 
In this case, $\cok C\{i\}=\cok B\{i\}\oplus \Z^{n_i-m_i}$ and $\ker C\{i\}=\ker B\{i\}$ --- and similarly for $B'$ and $C'$. 
We write $C\{i\}=C'\{i\}$ as
$$\begin{pmatrix}C_{00} & 0 & 0 \\ 0 & 0 & 0 \\ 0 & 0 & 0\end{pmatrix},$$
where $C_{00}$ is an invertible matrix over \Q and the last diagonal block has size $(n_i-m_i)\times(n_i-m_i)$. 
We write $U$ and $V$ as
$$\begin{pmatrix}U_{11} & U_{12} & U_{13} \\ U_{21} & U_{22} & U_{23} \\ U_{31} & U_{32} & U_{33}\end{pmatrix} \quad\text{and}\quad
\begin{pmatrix}V_{11} & V_{12} & V_{13} \\ V_{21} & V_{22} & V_{23} \\ V_{31} & V_{32} & V_{33}\end{pmatrix},$$
according to the block structure of $C\{i\}$ (and $C'\{i\}$). 

The condition 
$$UCV=C'$$
implies that 
$$U_{11}C_{00}V_{11}=C_{00},\quad
U_{i1}C_{00}V_{1j}=0,\text{ for all }(i,j)\neq(1,1).$$
Since $C_{00}$ is invertible as a matrix over \Q, we see that also 
$U_{11}$ and $V_{11}$ have to be invertible over \Q. 
Thus $V_{12}=0$, $V_{13}=0$, $U_{21}=0$, and $U_{31}=0$. 
Moreover, since we have to get the identity homomorphism on the new direct summand, we need to have $U_{33}=I$, $U_{23}=0$ and $U_{32}=0$. 
So now let $U_0$ be the block matrix where we erase the rows and columns corresponding to change the size of the $i$'th diagonal block from $r_i\times r_i$ to $m_i\times m_i$ --- call the new size $\mathbf{r}'$. 
Moreover, we let $C_0$ and $C_0'$ be the block matrices where we erase the rows corresponding to changing the size of the $i$'th diagonal block from $r_i\times r_i$ to $m_i\times n_i$. 
Note that the $i$'th diagonal block now is the matrix $\begin{smallpmatrix} U_{11} & U_{12} \\ 0 & U_{22} \end{smallpmatrix}$. 
This is a \GL matrix that induces the right automorphism of $\cok B\{i\}$. 
Moreover, clearly $U_0\{i\}B\{i\}V\{i\}=B\{i\}$. 
But more is true. 
We have that $U_0$ is a $\GLPZ[\mathbf{r}']$ matrix and 
that $U_0C_0V=C_0'$ and the induced $K$-web isomorphism agrees with the original on all parts except for the direct summands we cut out.

Now consider instead the case $m_i>n_i$. 
In this case, $\cok C\{i\}=\cok B\{i\}$ and $\ker C\{i\}=\ker B\{i\}\oplus \Z^{m_i-n_i}$ --- and similarly for $B'$ and $C'$. 
As above, we write $C\{i\}=C'\{i\}$, $U$ and $V$ as block matrices. 
In the same way, we see that $V_{12}=0$, $V_{13}=0$, $U_{21}=0$, and $U_{31}=0$, and that $V_{33}=I$, $V_{23}=0$ and $V_{32}=0$. 
So now let $V_0$ be the block matrix where we erase the rows and columns corresponding to changing the size of the $i$'th diagonal block from $r_i\times r_i$ to $n_i\times n_i$ --- call the new size $\mathbf{r}'$. 
Moreover, we let $C_0$ and $C_0'$ be the block matrices where we erase the rows corresponding to changing the size of the $i$'th diagonal block from $r_i\times r_i$ to $m_i\times n_i$. 
Note that the $i$'th diagonal block now is the matrix $\begin{smallpmatrix} V_{11} & 0 \\ V_{21} & V_{22} \end{smallpmatrix}$. 
This is a \GL matrix that induces the right automorphism of $\ker B\{i\}$. 
Moreover, clearly $U\{i\}B\{i\}V_0\{i\}=B\{i\}$. 
But more is true. 
We have that $V_0$ is a $\GLPZ[\mathbf{r}']$ matrix and 
that $UC_0V_0=C_0'$ and the induced $K$-web isomorphism agrees with the original on all parts except for the direct summands we cut out.

Induction finishes the proof. 
\end{proof}

\section{Generalization of Boyle's positive factorization method}
\label{sec:boyle}
In \cite{MR1907894}, Boyle proved several factorization theorems for square matrices.  These theorems are the key components to go from \SLPEe to flow equivalence.  In this section, we prove similar factorization theorems for rectangular matrices.  These are our key technical results to go from \SLPEe to move equivalence.  Although the assumptions might seem restrictive, every unital graph \ca is move equivalent to another unital graph \ca whose adjacency matrix satisfies the assumptions of the factorization theorem (\cf\ Section~\ref{sec:standardform}).  The proof for rectangular matrices will in parts closely follow the proof in \cite{MR1907894} for square matrices.  

First, we introduce a new equivalence called \emph{positive equivalence} of two matrices in \MPplusZ (see Definition~\ref{def: positive matrices}) and show that if $n_{i} \neq 0$ for all $i$, then two matrices in \MPplusZ that are \SLPE are positive equivalent.

\begin{definition}\label{def: positive matrices}
Define \MPplusZ to be the set of all $B \in \MPZ$ satisfying the following (\cf\ also the canonical form in Definition~\ref{def:CanonicalForm}):
\begin{enumerate}[(i)]
\item If $i \prec j$ and $B \{ i, j \}$ is not the empty matrix, then $B \{ i , j \} > 0$.

\item If $m_{i} = 0$, then $n_{i} =1$.

\item If $m_{i} = 1$, then $n_{i} = 1$ and $B \{ i \} = 0$.

\item If $m_{i} > 1$, then $B \{ i \} > 0$, $n_{i} , m_{i} \geq 3$, and the Smith normal form of $B\{i\}$ has at least two 1's (and thus the rank of $B \{ i \}$ is at least 2).  \end{enumerate}

Recall from \cite{MR1907894} that a \emph{basic elementary matrix} is a matrix that is equal to the identity matrix except for on at most one offdiagonal entry, where it is either $1$ or $-1$. 
Let $B, B' \in \MPplusZ$.  
An \SLPEe $\ftn{ (U,V) }{B}{B'}$ is said to be a \emph{basic positive equivalence through $\MPplusZ$} if one of $U$ and $V$ is a basic elementary matrix and the other is the identity matrix. 
An \SLPEe $\ftn{ (U,V) }{B}{B'}$ is said to be a \emph{positive equivalence through $\MPplusZ$} if it is the composition of basic positive equivalences 
\begin{align*}
B = B_{0} \to B_{1} \to B_{2} \to \cdots \to B_{k} = B'
\end{align*}
in which every $B_{i}$ is in $\MPplusZ$. 
We denote a positive equivalence $(U,V)$ through $\MPplusZ$ by writing $\xymatrix{B \ar[r]^-{(U,V)}_-{+} & B'}$ through $\MPplusZ$.
\end{definition}

Note that every element $U \in \SLPZ$ has a factorization as a product of basic elementary matrices in \SLPZ.
Therefore, a positive equivalence $\ftn{ (U,V) }{B}{B'}$ through $\MPplusZ$ is an \SLPEe-equivalence that allows one to stay in \MPplusZ in each step for some factorization of $U$ and $V$.

\subsection{Factorization: positive case}

In this section, we prove a factorization theorem similar to that of \cite[Theorem~5.1]{MR1907894} for positive rectangular matrices (in the sense that each entry is positive).  The proof is imitating the proof in \cite{MR1907894} for square matrices.   

\begin{definition}
By a \emph{signed transposition matrix}, we mean a matrix which is the matrix of a transposition, but with one of the offdiagonal $1$'s replaced by $-1$.  By a \emph{signed permutation matrix} we mean a product of signed transposition matrices.
\end{definition}

Note that for $K > 1$, any $K \times K$ permutation matrix with determinant 1 is a signed permutation matrix.  A $K \times K$ matrix $S$ is a signed permutation matrix if and only if $\det (S) = 1$ and the matrix $| S |$ is a permutation matrix (where $| S | (i,j) := | S(i,j) |$).

Let $B, B' \in \Mplus[m\times n]$. 
An equivalence $\ftn{ (U,V) }{ B }{ B' }$ is said to be a \emph{basic positive equivalence through \Mplus[m\times n]} if one of $U$ and $V$ is a basic elementary matrix and the other is the identity matrix. An equivalence $\ftn{ (U,V) }{ B }{ B' }$ is said to be a \emph{positive equivalence through \Mplus[m\times n]} if it is a composition of basic positive equivalences 
through \Mplus[m\times n]. 

An inspection of the proofs of \cite[Lemma~5.3 and Lemma~5.4]{MR1907894} shows that the proofs also hold for rectangular matrices.  Thus, we have the following lemmata.  

\begin{lemma}[{\cf\ \cite[Lemma~5.3]{MR1907894}}]\label{lem:BoyleLemma5.3}
Suppose $B \in \Mplus[m\times n]$, $E$ is a basic elementary matrix with nonzero offdiagonal entry $E(i,j)$, and the $i$'th row of $EB$ is not the zero row.  Then there exists $Q \in \SLZ[n]$ that is a product of nonnegative basic elementary matrices and there exists a signed permutation matrix $S \in \SLZ[m]$ such that $\ftn{ (SE, Q ) }{ B }{ SE B Q }$ is a positive equivalence through \Mplus[m\times n].
\end{lemma}

\begin{lemma}[\cf\ {\cite[Lemma~5.4]{MR1907894}}]\label{lem:BoyleLemma5.4}
Let $B$ be an element of $\MZ[K_1\times K_2]$ for $K_{1} ,K_{2} \geq 3$ such that the rank of $B$ is at least 2.  Suppose $U \in \SLZ[K_{1}]$ such that no row of $B$ and no row of $U B$ is the zero row.  Then $U$ is the product of basic elementary matrices $U = E_{k} \cdots E_{1}$ such that for $1 \leq j \leq k$ the matrix $E_{j} E_{j-1} \cdots E_{1} B$ has no zero rows. 
\end{lemma}

The following lemma is inspired by the reduction step in the proof of \cite[Lemma~5.5]{MR1907894}.  We give the entire proof for the convenience of the reader. 

\begin{lemma}\label{lem:BoyleLemma5.5a}
Let $B \in \Mplus[K_{1}\times K_{2}]$ with $K_{1} ,K_{2} \geq 3$.  Suppose the rank of $B$ is at least 2 and there exists $U \in \SLZ[K_{1}]$ such that $U B > 0$.  Then the equivalence $\ftn{ (U , I_{K_{2}} ) }{ B }{ UB }$ is a positive equivalence through \Mplus[K_{1}\times K_{2}]. 
\end{lemma}

\begin{proof}
By Lemma~\ref{lem:BoyleLemma5.4}, we can write $U$ as a product of basic elementary matrices $U = E_{k} E_{k-1} \cdots E_{1}$, such that for $1 \leq j \leq k$, the matrix $E_{j} \cdots E_{1}B$ has no zero rows.  By Lemma~\ref{lem:BoyleLemma5.3}, given the pair $( E_{1} , B)$, there is a nonnegative $Q_{1}$ that is a product of nonnegative basic elementary matrices and a signed permutation matrix $S_{1}$ such that $\ftn{ (S_1E_{1} , Q_{1} ) }{ B }{ S_{1} E_{1} B Q_{1} }$ is a positive equivalence through \Mplus[K_{1}\times K_{2}].  Note that 
\begin{align*}
UBQ_{1} = S_{1}^{-1} [ S_{1} E_{k} S_{1}^{-1} ] \cdots [ S_{1} E_{2} S_{1}^{-1} ] [ S_{1} E_{1} ] B Q_{1}.
\end{align*}  
Now, for $2 \leq j \leq k$, the matrix $S_{1} E_{j} S_{1}^{-1}$ is again a basic elementary matrix $E_{j}'$. Since we have that $E_{j}' \cdots E_{2}' ( S_{1} E_{1} B Q_{1} ) = S_{1} E_{j} \cdots E_{2} E_{1} B Q_{1}$, for $2 \leq j \leq k$, and since $E_{j} \cdots E_{2} E_{1} B Q_{1}$ has no zero rows, and $S_{1}$ is a signed permutation, we have that $E_{j}' \cdots E_{2}' ( S_{1} E_{1} B Q_{1} )$ has no zero rows for all $2 \leq j \leq k$.  

Using Lemma~\ref{lem:BoyleLemma5.3}, for the pair $( S_{1} E_{2} S_{1}^{-1} , S_{1} E_{1} B Q_{1} )$, we get a signed permutation matrix $S_{2}$ and a nonnegative $Q_{2}$ that is a product of nonnegative basic elementary matrices such that 
\begin{align*}
\ftn{ ( S_{2} [ S_{1} E_{2} S_{1}^{-1} ] , Q_{2} ) }{ S _{1} E_{1} B Q_{1} }{ S_{2} [ S_{1} E_{2} S_{1} ]^{-1} S_{1} E_{1} B Q_{1} Q_{2} }
\end{align*}
is a positive equivalence through \Mplus[K_{1}\times K_{2}].  Thus, we get a positive equivalence through \Mplus[K_{1}\times K_{2}] 
\begin{align*}
\ftn{ ( [S_{2} S_{1} E_{2} S_{1}^{-1} ] [ S_{1} E_{1} ] , Q_{1}Q_{2} ) }{ B }{ S_{2} S_{1} E_{2}  E_{1} B Q_{1} Q_{2} }
\end{align*}
and we observe that 
\begin{align*}
U B Q_{1} Q_{2} &= S_{1}^{-1} S_{2}^{-1} [ S_{2} S_{1} E_{k} S_{1}^{-1} S_{2}^{-1} ] \cdots \\ 
& \qquad\cdots [ S_{2} S_{1} E_{3} S_{1}^{-1} S_{2}^{-1} ][ S_{2} S_{1} E_{2} S_{1}^{-1} ] [ S_{1} E_{1} ] B Q_{1} Q_{2}.
\end{align*}
Continue this to obtain a signed permutation matrix $S = S_{k} \cdots S_{1}$ and a nonnegative matrix $Q = Q_{1} Q_{2} \cdots Q_{k}$ that is a product of nonnegative basic elementary matrices such that 
\begin{align*}
U B Q = S^{-1} [ S_{k} \cdots S_{1} E_{k} S_{1}^{-1} \cdots S_{k-1}^{-1} ] \cdots [ S_{2} S_{1} E_{2} S_{1}^{-1} ][ S_{1} E_{1} ] BQ = S^{-1} ( SU B Q )
\end{align*}
and $\ftn{ ( SU, Q ) }{ B }{ SU B Q }$ is a positive equivalence through \Mplus[K_{1}\times K_{2}].  

We claim that the equivalence $\ftn{ ( S, I_{K_{2}} ) }{ U B Q }{ SU B Q }$ is a positive equivalence through \Mplus[K_{1}\times K_{2}].  Since $S$ is a product of signed transposition matrices, it may be described as a permutation matrix in which some rows have been multiplied by $-1$.  Since $UBQ$ and $SUBQ$ are strictly positive, it must be that $S$ is a permutation matrix.  Also, $\det( S ) = 1$, so if $S \neq I_{K_{1}}$, then $S$ is a permutation matrix that is a product of $3$-cycles.  So it is enough to realize the positive equivalence through \Mplus[K_{1}\times K_{2}] in the case that $S$ is the matrix of a $3$-cycle.  For this we write the matrix
\begin{align*}
C = \begin{pmatrix} 0 & 1 & 0 \\ 0 & 0 & 1 \\ 1 & 0 & 0 \end{pmatrix}
\end{align*}
as the following product $C_{0} C_{1} C_{2} C_{3} C_{4} C_{5}$:
\begin{align*}
\begin{pmatrix}
1 & 0 & 0 \\
0 & 1 & 0 \\
0 & -1 & 1
\end{pmatrix}
\begin{pmatrix}
1 & 0 & 0 \\
-1 & 1 & 0 \\
0 & 0 & 1
\end{pmatrix}
\begin{pmatrix}
1 & 0 & -1 \\
0 & 1 & 0 \\
0 & 0 & 1
\end{pmatrix}
\begin{pmatrix}
1 & 1 & 0 \\
0 & 1 & 0 \\
0 & 0 & 1
\end{pmatrix}
\begin{pmatrix}
1 & 0 & 0 \\
0 & 1 & 0 \\
1 & 0 & 1
\end{pmatrix}
\begin{pmatrix}
1 & 0 & 0 \\
0 & 1 & 1 \\
0 & 0 & 1
\end{pmatrix}.
\end{align*}
For $0 \leq i \leq 5$, the matrix $C_{i} C_{i+1} \cdots C_{5}$ is nonnegative and has no zero rows. Therefore, the equivalence $\ftn{ ( C, I ) }{ D }{ CD }$ is a positive equivalence through \Mplus[K_{1}\times K_{2}] whenever $D \in \Mplus[K_{1}\times K_{2}]$.  Therefore, $\ftn{ ( S, I_{K_{2}} ) }{ U B Q }{ SU B Q }$ is a positive equivalence through \Mplus[K_{1}\times K_{2}] proving the claim.  Therefore, $\ftn{ ( S^{-1} , I_{K_{2} } ) }{ SU BQ }{ UBQ }$ is a positive equivalence through \Mplus[K_{1}\times K_{2}].  Since $Q$ is the product of nonnegative basic elementary matrices and $UB \in \Mplus[K_{1}\times K_{2}]$, the equivalence $\ftn{ ( I_{K_{1}} , Q ) }{ UB }{ UBQ }$ is a positive equivalence through \Mplus[K_{1}\times K_{2}].  Thus, $\ftn{ ( I_{K_{1}} , Q^{-1} ) }{ UBQ }{ UB }$ is a positive equivalence through \Mplus[K_{1}\times K_{2}].  Now the composition of positive equivalences through \Mplus[K_{1}\times K_{2}]
\begin{align*}
\xymatrix{
B \ar[rr]^-{ ( SU, Q ) }_-{+}  & & SU BQ \ar[rr]^-{ ( S^{-1} , I_{K_{2} } ) }_-{+} & & UBQ \ar[rr]^-{ ( I_{K_{1}} , Q^{-1} ) }_-{+} & & UB
}
\end{align*}
is a positive equivalence through \Mplus[K_{1}\times K_{2}] but the composition of these equivalences is equal to the equivalence $\ftn{ ( U, I_{K_{2} } ) }{ B }{ UB }$.  Hence, the equivalence $\ftn{ ( U, I_{K_{2} } ) }{ B }{ UB }$ is a positive equivalence through \Mplus[K_{1}\times K_{2}].
\end{proof}

The proof of the next lemma is similar to the proof of \cite[Lemma~5.5]{MR1907894}.  Since there are some differences between the two proofs we provide the entire argument.  

\begin{lemma}[\cf\ {\cite[Lemma~5.5]{MR1907894}}]\label{lem:BoyleLemma5.5b}
Let $B$ and $B'$ be elements of \Mplus[K_{1}\times K_{2}] with $K_{1} ,K_{2} \geq 3$, and the rank of $B$ and $B'$ at least 2.  Let $U \in \SLZ[K_{1}]$ and $W \in \SLZ[K_{2}]$ be given such that $U B$ has at least one strictly positive entry and $UB = B' W$.  Then the equivalence $\ftn{ (U , W^{-1} ) }{ B }{ B' }$ is a positive equivalence through \Mplus[K_{1}\times K_{2}]. 
\end{lemma}

\begin{proof}
We will first reduce to the case that $UB > 0$.  By assumption $(UB)(i,j) > 0$ for some $(i,j)$.  We can repeatedly add column $j$ to other columns until all entries of row $i$ are strictly positive.  This corresponds to multiplying from the right by a nonnegative matrix $Q$ in \SLZ[K_{2}], where $Q$ is the product of nonnegative basic elementary matrices, giving $UB Q = B' W Q$.  Then we can repeatedly add row $i$ of $UBQ$ to other rows until all entries are positive.  This corresponds to multiplying from the left by a nonnegative matrix $P$ in \SLZ[K_{1}], where $P$ is the product of nonnegative basic elementary matrices, giving $(PU)(BQ) = (PB')(WQ) > 0$.  We also have positive equivalences through \Mplus[K_{1}\times K_{2}] given by 
\begin{align*}
\ftn{ ( I , Q ) }{ B }{ BQ } \quad \text{and} \quad \ftn{ (P,I) }{ B' }{ PB' }.
\end{align*}

Note that the equivalence $\ftn{ (U, W^{-1}) }{ B }{ B' }$ is the composition of equivalences, $\ftn{( I, Q ) }{ B }{ BQ }$ followed by $\ftn{ (PU, (WQ)^{-1} ) }{ BQ }{ PB' }$ followed by $\ftn{ ( P^{-1} , I ) }{ PB' }{ B' }$.  Since $\ftn{( I, Q ) }{ B }{ BQ }$ and $\ftn{ ( P^{-1} , I ) }{ PB' }{ B' }$ are positive equivalences through \Mplus[K_{1}\times K_{2}], it is enough to show that the equivalence $\ftn{ (PU, (WQ)^{-1} ) }{ BQ }{ PB' }$ is a positive equivalence through \Mplus[K_{1}\times K_{2}].  Therefore, after replacing $(U, B, B' ,W)$ with $(PU, BQ, PB', WQ)$, we may assume without loss of generality that $UB > 0$.

By Lemma~\ref{lem:BoyleLemma5.5a}, the equivalence $\ftn{ ( U , I_{K_{2}} ) }{ B }{ UB }$ is a positive equivalence through \Mplus[K_{1}\times K_{2}]. Therefore, by Lemma~\ref{lem:BoyleLemma5.5a}, $\ftn{ ( ( W)^{\mathsf{T}} , I_{K_{1}} ) }{ ( B')^{\mathsf{T}} }{ W^{\mathsf{T}} ( B')^{\mathsf{T}} }$ is a positive equivalence through \Mplus[K_{2}\times K_{1}] which implies the equivalence $\ftn{ ( I_{K_{1}} , W )}{ B' }{ B' W }$ is a positive equivalence through \Mplus[K_{1}\times K_{2}].  Thus, the equivalence $\ftn{ ( I_{K_{1}} , W^{-1} ) }{ B' W }{ B' }$ is a positive equivalence through \Mplus[K_{1}\times K_{2}].  Since the equivalence $\ftn{ (U , W^{-1} ) }{ B }{ B' }$ is the composition of positive equivalences through \Mplus[K_{1}\times K_{2}] via $\ftn{ ( U, I_{K_{2}} ) }{ B }{ UB }$ followed by $\ftn{ (I_{K_{1}} , W^{-1} ) }{ B' W }{ B' }$, the equivalence $\ftn{ (U , W^{-1} ) }{ B }{ B' }$ a positive equivalence through \Mplus[K_{1}\times K_{2}].  
\end{proof}

\begin{theorem}[{\cf\ \cite[Theorem~5.1]{MR1907894}}]\label{thm:BoyleTheorem5.1}
Let $K_{1} ,K_{2} \geq 3$ and let $B \in \Mplus[K_{1}\times K_{2}]$.  Suppose $U \in \SLZ[K_{1}]$ and $V \in \SLZ[K_{2}]$ such that $U B V \in \Mplus[K_{1}\times K_{2}]$ and suppose that $X \in \SLZ[K_{1}]$ and $Y \in \SLZ[K_{2}]$ such that 
\begin{align*}
X B Y = 
\begin{pmatrix}
\begin{matrix} 1 & 0  \\
0 & 1 
\end{matrix} & 0 \\ 
0 & F
\end{pmatrix}.
\end{align*}
Then the equivalence $\ftn{ ( U , V ) }{ B }{ UB V }$ is a positive equivalence through \Mplus[K_{1}\times K_{2}].
\end{theorem}

\begin{proof}
Note that for any $H \in \SLZ[2]$, the $K_{1} \times K_{1}$ matrix $G_{H, 1}$ and the $K_{2} \times K_{2}$ matrix $G_{H, 2}$, given by 
$$G_{H, 1} = \begin{pmatrix} H & 0 \\ 0 & I_{K_{1}-2}  \end{pmatrix}\quad\text{and}\quad G_{H, 2} = \begin{pmatrix} H & 0 \\ 0 & I_{K_{2}-2}  \end{pmatrix},$$ 
give a self-equivalence $\ftn{ ( X^{-1} G_{H,1} X , Y G_{H,2}^{-1} Y^{-1} ) }{ B }{ B}$.

For a matrix $Q$, we let $Q( 12; * )$ denote the submatrix consisting of the first two rows.  Since $( XBY)( 12; * )$ has rank $2$ and $Y$ is invertible, we have that $(XB) ( 12; * )$ has rank two.  Therefore, there exists $H' \in \SLZ[2]$ such that the first row $r = \begin{pmatrix} r_{1}, \dots, r_{K_{2}} \end{pmatrix}$ of $H' [ ( XB ) ( 12; * ) ]$ has both a positive entry and a negative entry.   

Let $c = \left(\begin{smallmatrix} c_{1} \\ \vdots \\ c_{K_{1}} \end{smallmatrix}\right)$ denote the first column of $X^{-1}$, and note that it is nonzero.  Since $cr$ is the $K_{1} \times K_{2}$ matrix with $(i,j)$ entry equal to $c_{i} r_{j}$, we have that $cr$ has a positive and a negative entry.  For each $m \in \N$, set $H_{m} = \left(\begin{smallmatrix} m & -1 \\ 1 & 0 \end{smallmatrix}\right) H'$.  Choose $m$ large enough such that the entries of the two matrices $X^{-1} G_{H_{m}, 1} X B$ and $mc r$ will have the same sign whenever the entries of $mcr$ are nonzero.  In particular, $X^{-1} G_{H_{m} , 1 } X B$ will have a positive entry.  From Lemma~\ref{lem:BoyleLemma5.5b} it follows that $\ftn{ ( X^{-1} G_{ H_{m} , 1} X ,   Y G_{ H_{m} , 2}^{-1} Y^{-1} ) }{ B }{ B }$ is a positive equivalence through \Mplus[K_{1}\times K_{2}].      

Similarly for large enough $m$, the entries of $U X^{-1} G_{H_{m} , 1} X B$ will agree in sign with the entries $Ucr$ whenever the entries of the latter matrix are nonzero.  Since $U$ is invertible, the matrix $Ucr$ is nonzero, and thus contains positive and negative entries, because $r$ does.  Therefore, $U X^{-1} G_{H_{m} , 1} X B$ contains a positive entry.  By Lemma~\ref{lem:BoyleLemma5.5b}, 
\begin{align*}
\ftn{ ( U X^{-1} G_{ H_{m} , 1} X ,   Y G_{ H_{m} , 2}^{-1} Y^{-1} V) }{ B }{ B' }
\end{align*}
gives a positive equivalence through \Mplus[K_{1}\times K_{2}] with $B' = UBV$.  Hence, the equivalence $\ftn{ ( U,V) }{ B }{ B' }$ is a positive equivalence through \Mplus[K_{1}\times K_{2}] since it is the composition of positive equivalences through \Mplus[K_{1}\times K_{2}]: 
\begin{align*}
\ftn{ ( X^{-1} G_{ H_{m} , 1}^{-1} X  , Y G_{ H_{m} , 2} Y^{-1} ) }{ B }{ B }
\end{align*}
followed by 
\begin{equation*}
\ftn{ ( U X^{-1} G_{ H_{m} , 1} X ,   Y G_{ H_{m} , 2}^{-1} Y^{-1} V) }{ B }{ B' } \qedhere
\end{equation*}
\end{proof}

\subsection{Factorization:  general case}

We now use the above results to prove a factorization for general $B, B' \in \MPplusZ$ with $n_{i}\neq 0$ that are \SLPE.  Again, many of the arguments follow \cite{MR1907894}.

\begin{lemma}[{\cf\ \cite[Lemma~4.6]{MR1907894}}]\label{lem:BoyleLemma4.6}
Let $B, B' \in\MPplusZ$ with $n_{i} \neq 0$ for all $i$.  If $\ftn{ (U, V) }{ B }{ B' }$ is an \SLPEe such that $U\{ i \}$ and $V \{ j\}$ are the identity matrices of the appropriate sizes whenever they are not the empty matrix, then $\ftn{ (U,V) }{ B }{ B' }$ is a positive equivalence through \MPplusZ.
\end{lemma}

\begin{proof}
We will first find $Q$ in \SLPZ that is a product of nonnegative basic elementary matrices in \SLPZ such that $\ftn{ ( U, Q ) }{ B }{ U B Q }$ is a positive equivalence through \MPplusZ. 
We may assume that $U$ is not the identity matrix. 
Factor $U = U_{n} \cdots U_{1}$ where for each $U_{t}$ there is an associated pair $( i_{t}, j_{t} )$ with $i_t\prec j_t$ such that the following hold
\begin{itemize}
\item  $U_{t} = I$ except in the block $U_{t} \{ i_{t} , j_{t} \}$, where it is nonzero
\item if $s \neq t$, then $( i_{s} , j_{s} ) \neq ( i_{t} , j_{t} )$.
\end{itemize} 
Factor $U_{1} = U_{1}^{-} U_{1}^{+}$, where $U_{1}^{-}$ and $U_{1}^{+}$ are equal to $I$ outside the block $\{ i_{1} , j_{1} \}$, $U_{1}^{-} \{ i_{1} , j_{1} \}$ is the nonpositive part of $U_{1} \{ i_{1} , j_{1} \}$ and $U_{1}^{+} \{ i_{1} , j_{1} \}$ is the nonnegative part of $U_{1} \{ i_{1} , j_{1} \}$. Note that $U_{1}^{+}$ is a product of nonnegative basic elementary matrices in \SLPZ[\mathbf{m}] and $U_{1}^{-}$ is a product of nonpositive basic elementary matrices in \SLPZ[\mathbf{m}]. It is now clear that $\ftn{ ( U_{1}^{+} , I ) }{ B }{ U_{1}^{+}B }$ is a positive equivalence through \MPplusZ.  

Now, note that $U_{1}^{-} U_{1}^{+} B = U_{1}^{+} B$ outside the blocks $\{ i_{1} , k \}$ such that $i_{1} \prec j_{1} \preceq k$.  Also note that $m_{i_{1}} \neq 0$ (since $U_1\neq I$).   

Suppose $m_{i_{1}} > 1$.  Then $B \{ i_{1} \}$ is not the empty matrix and $B \{ i_{1} \} > 0$ since $B\in\MPplusZ$.  Hence, $(U_{1}^{+} B) \{ i_{1} \} > 0$ since $\ftn{ ( U_{1}^{+} , I ) }{ B }{ U_{1}^{+}B }$ is a positive equivalence through \MPplusZ. We can now add columns of  $(U_{1}^{+}B )\{ i_{1} \}$ to columns of $(U_{1}^{+}B) \{ i_{1} , k\}$ for all $i_{1} \prec j_{1} \preceq k$ enough times to obtain a $Q_{1}$ that is a product of nonnegative basic elementary matrices in \SLPZ such that $\ftn{ ( U_{1}^{-} , Q_{1} ) }{ U_{1}^{+} B }{ U_{1}^{-} U_{1}^{+} B Q_{1} }$ is a positive equivalence through \MPplusZ. Since $\ftn{ (U_{1} , Q_{1} ) }{ B }{ U_{1} B Q_1 }$ is the composition of positive equivalences 
\begin{align*}
\xymatrix{
B \ar[rr]^-{ ( U_{1}^{+} , I ) }_-{+} & & U_{1}^{+}B \ar[rr]^-{ ( U_{1}^{-} , Q_{1} ) }_-{+} & & U_{1} B Q_{1}
}
\end{align*}
through \MPplusZ, we get that the equivalence $\ftn{ (U_{1} , Q_{1} ) }{ B }{ U_{1} B Q_{1} }$ is a positive equivalence through \MPplusZ.  

Suppose $m_{i_{1}} = 1$.  Then $B\{i_{1} \}$ and $B'\{i_{1} \}$ are the $1\times 1$ zero matrix.  Since  $B\in\MPplusZ$, we have that $B \{ i_{1} , j \} > 0$ and $(U_{1}^{+}B) \{ i_{1} , j \} > 0$ for all $i_{1} \prec j$.  Therefore, for all $j_{1} \prec k$, we can add columns of $(U_{1}^{+}B) \{ i_{1} , j_{1} \}$ to columns of $U_{1}^{+} B \{ i_{1} , k \}$ enough times to obtain a $Q_{1}'$ that is a product of nonnegative basic elementary matrices in \SLPZ such that $(U_{1}^{-} U_{1}^{+} B Q_{1}' ) \{i_{1} , k \} > 0$ for all $j_{1} \prec k$.  Suppose there exists $i_{1} \prec j \prec j_{1}$.  Then we can add columns of $( U_{1}^{+} B Q_{1}' )\{ i_{1} , j \}$ to columns of $( U_{1}^{+} B Q_{1}') \{ i_{1} , j_{1} \}$ enough times to obtain a $Q_{1}''$ that is a product of nonnegative basic elementary matrices in \SLPZ such that  $(U_{1}^{-}U_{1}^{+} B Q_{1}' Q_{1}'' )\{ i_{1} ,k \} > 0$ for all $i_{1} \prec j_{1} \preceq k$.  Suppose there is no $j$ such that $i_{1} \prec j \prec j_{1}$.  Then $j_{1}$ is an immediate successor of $i_{1}$.  Since $U\{ i_{1}, j_{1} \} = ( U_{n} \cdots U_{2} ) \{ i_{1} , j_{1} \} + U_{1} \{ i_{1} , j_{1} \} = U_{1} \{ i_{1} , j_{1} \}$, we have that 
\begin{align*}
&(U_{1}^{-} (U_{1}^{+} B Q_{1}'))\{ i_{1} , j_{1} \} \\
&\qquad = U_{1} \{i_{1} \} (B Q_{1}')\{i_{1} , j_{1} \} + U_{1} \{ i_{1} , j_{1} \} (BQ_{1}') \{ j_{1} \}  \\
&\qquad = B\{ i_{1} \} Q_{1}' \{i_{1}, j_{1} \} + B \{ i_{1}, j_{1} \} Q_{1}' \{ j_{1} \} + U_{1} \{ i_{1} , j_{1} \} B \{ j_{1} \} Q_{1}' \{ j_{1} \} \\
&\qquad =  B \{ i_{1}, j_{1} \} + U_{1} \{ i_{1} , j_{1} \} B \{ j_{1} \} \\
&\qquad = U\{i_{1}\}B \{ i_{1}, j_{1} \} + U \{ i_{1} , j_{1} \} B \{ j_{1} \} \\
&\qquad = (U B) \{ i_{1} , j_{1} \} \\
&\qquad = ( B'V^{-1}) \{ i_{1} , j_{1} \} \\
&\qquad = B' \{ i_{1} \} V^{-1} \{ i_{1} , j_{1} \} + B' \{ i_{1} , j_{1} \} V^{-1}\{j_{1} \} \\
&\qquad = B' \{ i_{1} , j_{1} \} > 0.
\end{align*}
Therefore, we get $Q_{1}$ in \SLPZ that is a product of nonnegative basic elementary matrices in \SLPZ such that $\ftn{ (U_{1}^{-} , Q_{1} ) }{ U_{1}^{+}B }{ U_{1} B }$ is a positive equivalence through \MPplusZ.  Since $\ftn{ (U_{1} , Q_{1} ) }{ B }{ U_{1} B Q_1 }$ is the composition of positive equivalences
\begin{align*}
\xymatrix{
B \ar[rr]^-{ ( U_{1}^{+} , I ) }_-{+} & & U_{1}^{+}B \ar[rr]^-{ ( U_{1}^{-} , Q_{1} ) }_-{+} & & U_{1} B Q_{1}
}
\end{align*}
through \MPplusZ, we get that the equivalence $\ftn{ (U_{1} , Q_{1} ) }{ B }{ U_{1} B Q_{1} }$ is a positive equivalence through \MPplusZ.

In all cases, we get a $Q_{1}$ in \SLPZ that is a product of nonnegative basic elementary matrices in \SLPZ such that the equivalence $\ftn{ (U_{1} , Q_{1} ) }{ B }{ U_{1} B Q_{1} }$ is a positive equivalence through \MPplusZ.

By repeating the process for the matrices $U_{1} B Q_{1}$ and $U_{2} U_{1} B Q_{1}$, we get $Q_{2}$ that is the product of nonnegative basic elementary matrices in \SLPZ such that the equivalence $\ftn{ ( U_{2} , Q_{2} ) }{ U_{1} B Q_{1} }{ U_{2} U_{1} B Q_{1} Q_{2} }$ is a positive equivalence through \MPplusZ. We continue this process to get $Q_{i}$ that is the product of nonnegative basic elementary matrices in \SLPZ for $1 \leq i \leq n$ such that $\ftn{ ( U, Q ) }{ B }{ U B Q }$ is a positive equivalence through \MPplusZ, where $Q = Q_{1}\cdots Q_{n}$. 

We now show that there exists $P$ that is a product of nonnegative basic elementary matrices in \SLPZ[\mathbf{m}] such that $\ftn{ ( P , V^{-1} ) }{ B' }{P B' V^{-1} }$ is a positive equivalence through \MPplusZ. 
Throughout the rest of the proof, if $M \in \MPplusZ$, then $M \{ \{1, 2, \dots, i \} \}$ will denote the block matrix whose $\{ s, r \}$ block is $M \{ s, r \}$ for all $1 \leq s, r \leq i$.  First note that there are matrices $V_2, \dots, V_N$ in \SLPZ such that $V^{-1} = V_2 V_3 \cdots V_N$, each $V_i$ is the identity matrix except for possibly on blocks $V_i \{ l, i \}$ with $l\prec i$, and 
\[
V_2 \cdots V_i = 
\begin{pmatrix}
V^{-1} \{ \{ 1, \dots, i \} \} & 0 \\
0 & I
\end{pmatrix}.
\]
Let $V_i^-$ be the matrix in \SLPZ that is the identity matrix except for the blocks $V_i \{ l, i \}$ and $V_i^- \{ l, i \}$ is the nonpositive part of $V_i\{ l , i\}$ and  let $V_i^+$ be the matrix in \SLPZ that is the identity matrix except for the blocks $V_i \{ l, i \}$ and $V_i^+ \{ l, i \}$ is the nonnegative part of $V_i\{ l , i\}$.  Note that $V_i^+ V_i^-$ is equal to the identity matrix except for the blocks $V_i \{ l, i \}$ and $(V_i^+ V_i^-)\{ l,i\} = V_i^+\{l, i\} + V_i^-\{l, i \} = V_i \{ l , i \}$.  Therefore, $V_i = V_i^+ V_i^-$

We will inductively construct matrices $P_2 , P_3, \dots, P_N$ in \SLPZ[\mathbf{m}] such that each $P_i$ is the product of nonnegative basic matrices such that each $P_i$ is the identity outside of the blocks $\{l, i \}$ for $l \prec i$ and for each $2 \leq i \leq N$, we have that $\ftn{ ( P_i  ,  V_i ) }{ P_{i-1} \cdots P_2 B' V_2 \cdots V_{i-1} }{  P_i \cdots P_2 B' V_2 \dots V_i }$ is a positive equivalence through \MPplusZ.  Note that if we have constructed $P_2$, $P_3,\ldots,P_N$, then the composition of these positive equivalences through \MPplusZ gives a positive equivalence $\ftn{ ( P , V^{-1}) }{ B' }{ P B' V^{-1} }$ through \MPplusZ, where $P = P_k \cdots P_2$.  Thus, once we have constructed $P_2 , P_3, \dots, P_N$ the lemma holds.

We first construct $P_2$.  Note that if $1$ is not a predecessor of $2$, then $V_2^+ = V_2^- = I$.  Therefore, $\ftn{ ( I , V_2 ) }{ B' }{ B' V_2 }$ is a positive equivalence through \MPplusZ.  Suppose $1 \preceq 2$.  Suppose $m_1 = 0$.  Then $B' V_2^+ V_2^- = B' V_2^- = B'$ which implies that $\ftn{ ( I, V_2^- ) }{ B' V_2^+ }{ B' V_2 }$ is a positive equivalence through \MPplusZ.  So, $\ftn{ ( I, V_2 ) }{ B' }{ B'V_2 }$ is a positive equivalence through \MPplusZ since it is the composition of the positive equivalences $( I, V_2^+ )$ and $( I, V_2^-)$ through \MPplusZ.  Suppose $m_1 \neq 0$.  In this situation, we have three cases, $m_2 > 1$, $m_{2} = 1$, and $m_2 = 0$.  

Suppose $m_2 > 1$.  Note that $B' V_2^+ V_2^-$ is equal to $B'$ except for the $\{1,2\}$ block. It is clear that $\ftn{ ( I , V_2^+ ) }{ B' }{ B'V_2^+ }$ is a positive equivalence through \MPplusZ.  Since $m_{2} > 1$, we have that $(B'V_{2}^{+}) \{ 2 \} > 0$.  Hence, we may add rows of $(B'V_2^+) \{2\}$ to rows of $(B'V_2^+)\{1,2\}$ to get a matrix $P_2$ in \SLPZ[\mathbf{m}] that is the product of nonnegative basic matrices and is the identity outside of the block $\{1,2\}$ such that $\ftn{ (P_2, V_2^- ) }{ B'V_2^+ }{ P_2 B' V_2 }$ is a positive equivalence through \MPplusZ.  Composing the positive equivalences $(I , V_2^+)$ and $( P_2 , V_2^- )$ through \MPplusZ, we get a positive equivalence $\ftn{ ( P_2 , V_2 ) }{ B' }{ P_2 B V_2 }$ through \MPplusZ.  

Suppose $m_{2} \leq 1$.  Then $B\{ 2 \}$ is either the $1 \times 1$ zero matrix or the empty matrix.
\begin{align*}
(B' V_{2}) \{ 1, 2 \} &= B' \{1\} V_{2} \{1,2\} + B' \{1,2 \} V_{2} \{ 2 \} \\
			&= B' \{1\} V^{-1} \{1, 2\} + B' \{1,2\} V^{-1} \{2\} \\
			&= (B' V^{-1}) \{1,2\} \\
			&= (UB)\{1,2\} \\
			&= U \{1 \} B\{1,2\} + U \{1,2\} B \{2 \} \\
			&= B\{1,2\}  > 0
\end{align*}  
Therefore $\ftn{ ( I , V_2^- ) }{ B' V_2^+ }{ B' V_2 }$ is a positive equivalence through \MPplusZ and by composing the positive equivalences $( I, V_2^+ )$ and $(I , V_2^-)$ through \MPplusZ, we get a positive equivalence $\ftn{ (I, V_2 ) }{ B' }{ B' V_2 }$ through \MPplusZ.

So, in all cases, we have found a matrix $P_2$ in \SLPZ[\mathbf{m}] that is the product of nonnegative basic elementary matrices and is the identity outside of the block $\{1,2\}$ such that $\ftn{ ( P_2 , V_2 ) }{ B' }{ P_2 B' V_2 }$ is a positive equivalence through \MPplusZ.  

Let $2 \leq n \leq N-1$ and suppose we have constructed $P_2 , P_3, \dots, P_n$ in \SLPZ[\mathbf{m}] such that each $P_i$ is the product of nonnegative basic matrices and $P_i$ is the identity outside of the blocks $\{l, i \}$ with $l \prec i$ and for each $2 \leq i \leq n$, we have that $\ftn{ ( P_i  ,V_i ) }{ P_{i-1} \cdots P_2 B' V_2 \cdots V_{i-1} }{  P_i \cdots P_2 B' V_2 \dots V_i }$ is a positive equivalence through \MPplusZ.  

To simplify the notation, we set $B_i' = P_i \cdots P_2 B' V_2 \dots V_i$.  Since $B_n'$ is in $\MPplusZ$, we get a positive equivalence $\ftn{ ( I, V_{n+1}^+ ) }{ B_n' }{ B_n'V_{n+1}^+ }$ through \MPplusZ.  Note that $B_n'V_{n+1}^+V_{n+1}^-$ is equal to $B_n'$ except for the blocks $\{ i , n+1 \}$ with $i \preceq n+1$.  

Suppose $m_{n+1} > 1$.  Then $(B_n' V_{n+1}^+) \{ n+1 \}  > 0$.  Hence, we may add rows of $( B_n' V_{n+1}^+) \{ n+1 \}$ to rows of $(B_n' V_{n+1}^+)\{i, n+1 \}$ for all $i \prec n+1$, to obtain a matrix $P_{n+1}$ in \SLPZ[\mathbf{m}] that is the product of nonnegative basic matrices and is the identity outside of the blocks $\{ i, n+1\}$ for $i \prec n+1$ such that $\ftn{ ( P_{n+1}, V_{n+1}^- )}{B_n'V_{n+1}^+ }{ P_{n+1}B_n'V_{n+1}}$ is a positive equivalence through \MPplusZ.  Composing the positive equivalences $( I , V_{n+1}^+ )$ and $( P_{n+1} , V_{n+1}^-)$ through \MPplusZ, we get that $\ftn{ ( P_{n+1}, V_{n+1} ) }{ B_n'}{  P_{n+1} B_n' V_{n+1} }$ is a positive equivalence through \MPplusZ.

Suppose $m_{n+1} \leq 1$.  Let $J = \{ i_0, \dots, i_t \}$ be the set of elements $i_s\in\mathcal{P}$ that satisfy $i_s \prec n+1$, $m_{i_s} \neq 0$, and if $i_s \prec l \prec n+1$, then $m_l = 0$.
Note that for all distinct $i_s,i_r\in J$, $i_s$ is not a predecessor of $i_r$.
Note that if $J = \emptyset$, then $B_n' V_{n+1}^+ V_{n+1}^- = B_n' V_{n+1}^- = B_n'$.  This would imply that $\ftn{ ( I, V_{n+1}^- )}{ B_n' V_{n+1}^+ }{ B_n' V_{n+1}}$ is a positive equivalence through \MPplusZ and hence $\ftn{ ( I , V_{n+1} ) }{ B_n' }{ B_n' V_{n+1} }$ is a positive equivalence through \MPplusZ. 

Suppose $J \neq \emptyset$.  Note that for each $i_s \in J$, 
\begin{align*}
(B_n'V_{n+1} )\{ i_s, n+1 \} &= \sum_{ i_s \preceq l \preceq n+1 }  (P_n \cdots P_2 B' ) \{ i_s , l \}  (V_2 \dots V_nV_{n+1})\{ l, n+1 \}
 \\  				&=\sum_{ i_s \preceq l \preceq n+1 }  (P_n \cdots P_2 B' ) \{ i_s , l \}  V^{-1}\{ l, n+1 \} \\ 
 				&= ( ( P_n \cdots P_2 B' ) V^{-1} ) \{ i_s , n+1 \}
\end{align*}
 since $(V_2 \dots V_nV_{n+1})\{\{1, \dots n+1\}\} = V^{-1} \{ \{1, \dots, n+1 \} \}$.  Since $P_n \cdots P_2 U B = P_n \cdots P_2 B' V^{-1}$,
\begin{align*}
 ( ( P_n \cdots P_2 B' ) V^{-1} ) \{ i_s , n+1 \} &= ((P_n \cdots P_2 U) B) \{i_s, n+1 \} \\
 				&= \sum_{i_s \preceq l \preceq n+1 } (P_n \cdots P_2 U ) \{ i_s , l \}  B\{ l, n+1 \}. 
\end{align*}
Using the fact that $m_l = 0$ for all $i_s \prec l \prec n+1$, $B\{n+1\}$ is either the $1 \times 1$ zero matrix (if $m_{n+1} = 1$) or the empty matrix (if $m_{n+1} = 0$), and $(P_n \cdots P_2 U ) \{ i_s  \} = I$, we get that 
\begin{align*}
(B_n'V_{n+1} )\{ i_s, n+1 \} &=  B\{ i_{s}, n+1 \} + (P_n \cdots P_2 U ) \{ i_s , n+1 \} B \{n+1\} \\
				&= B \{ i_s, n+1 \}.
\end{align*}
Moreover, $(B_n'V_{n+1} )\{ i_s, n+1 \} = B \{ i_s, n+1 \} > 0$ because $m_{i_s} \neq 0$ and $B \in \MPplusZ$.

For each $l\prec n+1$ with $m_l\neq 0$, there exists an $s$ such that $l\preceq i_s$. 
Recall that $(B_n' V_{n+1}^+ )\{ i_s, n+1 \} > 0$, so if $l\prec i_s$, we may add rows of $(B_n'V_{n+1}^+ )\{ i_s, n+1 \}$ to rows of $( B_n'V_{n+1}^+)\{l , n+1 \}$, to get a matrix $P_{n+1}^l$ in \SLPZ[\mathbf{m}] that is the product of nonnegative basic elementary matrices and is the identity outside of the block $\{l,n+1\}$ such that $( P_{n+1}^l B_n' V_{n+1})\{ l, n+1 \} > 0$.  
Doing this for all $l \prec n+1$ with $m_l\neq 0$, we get a matrix $P_{n+1}$ in \SLPZ[\mathbf{m}] that is the product of nonnegative basic elementary matrices and is the identity outside of the blocks $\{l,n+1\}$ for $l \prec n+1$ such that $\ftn{  ( P_{n+1} , V_{n+1}^- )}{ B_n' V_{n+1}^+}{ P_{n+1} B_n' V_{n+1} }$ is a positive equivalence through \MPplusZ.  Composing the positive equivalences $( I, V_{n+1}^+ )$ and $( P_{n+1} , V_{n+1}^- )$ through \MPplusZ, we get that $\ftn{  ( P_{n+1} , V_{n+1})}{ B_n' }{ P_{n+1} B_n' V_{n+1} }$ is a positive equivalence through \MPplusZ. 

In all cases, we get a matrix $P_{n+1}$ in \SLPZ[\mathbf{m}] that is the product of nonnegative basic elementary matrices and is the identity outside of the blocks $\{l,n+1\}$ for $l \prec n+1$ such that $\ftn{  ( P_{n+1} , V_{n+1})}{ B_n' }{ P_{n+1} B_n' V_{n+1} }$ is a positive equivalence through \MPplusZ.  The claim now follows by induction.
\end{proof}

The next lemma allows us to reduce the general case to the case that the diagonal blocks $U\{ i \}$ and $V\{j\}$ are the identity matrices of the appropriate sizes when they are not the empty matrices.  This will allow us to use Lemma~\ref{lem:BoyleLemma4.6} to get the desired positive equivalence through \MPplusZ.

\begin{lemma}[{\cf\ \cite[Lemma~4.9]{MR1907894}}]\label{lem:BoyleLemma4.9}
Let $B\in \MPplusZ$ with $n_{l} \neq 0$ for all $l$. Fix $i$ with $m_{i}  > 1$.
\begin{enumerate}[(1)]
\item  \label{lem:BoyleLemma4.9-item1}
Suppose $E$ is a basic elementary matrix in \SLPZ[\mathbf{m}] such that $E \{ j, k \} = I \{ j, k \}$ when $(j,k) \neq (i, i)$ and 
\begin{align*}
\ftn{ ( E\{ i \} , I ) }{ B \{ i  \} }{ EB \{  i \}}
\end{align*}
is a positive equivalence through \Mplus[m_i\times n_i].  Then there exists $V \in \SLPZ$ that is the product of nonnegative basic elementary matrices in \SLPZ such that $V \{  k \} = I$ for all $k$ and 
\begin{align*}
\ftn{ ( E, V ) }{ B }{ E B V }
\end{align*}
is a positive equivalence through \MPplusZ.  

\item  \label{lem:BoyleLemma4.9-item2}
Suppose $E$ is a basic elementary matrix in \SLPZ such that $E \{ j, k \} = I \{ j, k \}$ when $(j,k) \neq (i, i)$ and 
\begin{align*}
\ftn{ (  I, E\{ i \} ) }{ B \{  i \} }{ BE \{  i \}}
\end{align*}
is a positive equivalence through \Mplus[m_i\times n_i].
Then there exists $U \in \SLPZ$ that is the product of nonnegative basic elementary matrices in \SLPZ such that $U \{ k \} = I$ for all $k$ and 
\begin{align*}
\ftn{ ( U, E ) }{ B }{ U B E }
\end{align*}
is a positive equivalence through \MPplusZ.
\end{enumerate}
\end{lemma}

\begin{proof}
We prove \ref{lem:BoyleLemma4.9-item1}, the proof of \ref{lem:BoyleLemma4.9-item2} is similar.  Let $E(s,t)$ be the nonzero offdiagonal entry of $E$.  If $E(s,t) =1$, then set $V = I$.  Suppose $E( s, t ) = -1$.  So, $E$ acts from the left to subtract row $t$ from row $s$.  Since $m_{i} > 1$ and 
\begin{align*}
\ftn{ ( E\{ i \} , I ) }{ B \{ i  \} }{ EB \{ i\}},
\end{align*}
is a positive equivalence through \Mplus[m_i\times n_i], we have that $(EB) \{ i \} > 0$.  Thus, there exists $r$ an index for a column through the $\{ i, i \}$ block such that $B( s, r ) > B( t, r )$.  Let $V$ be the matrix in \SLPZ that acts from the right to add column $r$ to column $q$, $M$ times, for every $q$ indexing a column through an $\{ i , j \}$ block for which $i \prec j$.  Choosing $M$ large enough, we have that $\ftn{ (E, I ) }{ B V }{ E B V }$ is a positive equivalence through \MPplusZ.  Therefore, $\ftn{ ( E, V ) }{ B }{ E B V  }$ is a positive equivalence through \MPplusZ since it is the composition of two positive equivalences through \MPplusZ: $\ftn{ ( I, V ) }{ B }{ B V }$ followed by $\ftn{ ( E , I ) }{ B V }{ E B V }$.
\end{proof}

We are now ready to prove the main result of this section.  This result will be used to show that if the adjacency matrices of $E$ and $F$ are \SLPE, then $E$ is move equivalent to $F$.  Consequently, $C^{*} (E)$ is Morita equivalent to $C^{*} (F)$.

\begin{theorem}[{\cf\ \cite[Theorem~4.4]{MR1907894}}]\label{thm:BoyleTheorem4.4}
Let $B, B' \in \MPplusZ$ with $n_{i} \neq 0$ for all $i$.  Suppose there exist $U \in \SLPZ[\mathbf{m}]$ and $V \in \SLPZ$ such that $U B V = B'$.  Then $\ftn{ ( U, V ) }{ B }{ B' }$ is a positive equivalence through \MPplusZ. 
\end{theorem}

\begin{proof}
By Theorem~\ref{thm:BoyleTheorem5.1}, for each $i$ with $m_{i} > 1$, we have that $\ftn{ ( U\{ i \} , V \{ i \} ) }{ B \{ i \} }{B' \{ i \} }$ is a positive equivalence through \Mplus[m_i\times n_i].  So, we may find a string of basic positive equivalences $( E_{1}, F_{1} ), \dots, ( E_{t} , F_{t} )$, with every $E_{l} \{ i, j \} = I\{ i,j \}$, $F_{l} \{ i, j \} = I \{i,j\}$ unless $i = j$ with $m_{i} > 1$, which accomplishes the decomposition inside the diagonal blocks as basic positive equivalences through \Mplus[m_i\times n_i], for $i\in\calP$ with $m_i>1$.  Note that in the case $m_{i} = 1$, then $U \{ i \} = V\{i\} = 1$.  By Lemma~\ref{lem:BoyleLemma4.9}, we may find $( U_{1} , V_{1} ), \dots, ( U_{t} , V_{t} )$ such that $U_{s} \in \SLPZ[\mathbf{m}]$, $V_{s} \in \SLPZ$, $U_{s} \{ k \} = I$, $V_{s} \{ k  \} = I$, for all $k\in\calP$ with $m_k\neq 0$, and such that we have the following positive equivalences 
\begin{align*}
\xymatrix{
B \ar[r]_-{+}^-{( U_{1} , F_{1} )} & \cdot \ar[r]_-{+}^-{( E_{1} , V_{1} )} & \cdots \ar[r]_-{+}^-{( U_{t} , F_{t} )} & \cdot \ar[r]_-{+}^-{( E_{t} , V_{t} )} & B''.
}
\end{align*}
through \MPplusZ. 
Let $X = E_{t} U_{t} \cdots E_{2} U_{2} E_{1} U_{1}$ and $Y = F_{1} V_{1} F_{2} V_{2} \cdots F_{t} V_{t}$.  Then for all $i$, we have that $X\{ i \} = U \{ i \}$ and $Y \{ i \} = V \{ i \}$.  Therefore, $( UX^{-1} ) \{ i \} = I$ and $( Y^{-1} V ) \{ i\} = I$ for all $i$.  Then by Lemma~\ref{lem:BoyleLemma4.6},
\begin{align*}
\xymatrix{
B'' \ar[rr]_-{+}^-{( U X^{-1} , Y^{-1} V ) } & & B'
}
\end{align*}
is a positive equivalence through \MPplusZ.  Thus, $\ftn{ ( U, V ) }{ B }{ B' }$ is a positive equivalence through \MPplusZ since it is the composition of two positive equivalences
\begin{equation*}
\xymatrix{
B \ar[r]_-{+}^-{ ( X , Y ) } & B'' \ar[rr]^-{( U X^{-1} , Y^{-1} V ) }_-{+} & & B'.
} 
\end{equation*}
through \MPplusZ.
\end{proof}

We will here give an important result that is a key element in proving the main lifting results. The result is also interesting in its own right, since it allows us to get a stable isomorphism for the corresponding graph algebras from a very concrete equation involving the adjacency matrices of two graphs. 

\begin{theorem}\label{thm:SLP-equivalence-implies-stable-isomorphism}
Let $E$ and $E'$ be graphs with finitely many vertices, and assume that $\Bsf_{E}\in\MPZccc[\mathbf{m}\times\mathbf{n}]$, $\Bsf_{E'}\in\MPZccc[\mathbf{m}'\times\mathbf{n}']$, $\mathbf{n}-\mathbf{m}=\mathbf{n}'-\mathbf{m}'$ and the component corresponding to $\Asf_{E}\{i\}$ is a noncyclic strongly connected component if and only if the component corresponding to $\Asf_{E'}\{i\}$ is a noncyclic strongly connected component (\ie, they have the same cyclic components and the same singular components). 
Let $\mathbf{r}=(r_i)_{i\in\calP}$ and $\mathbf{r}'=(r_i')_{i\in\calP}$ be multiindices such that $\mathbf{m}+\mathbf{r}=\mathbf{m}'+\mathbf{r}'$ and $r_i=0=r_i'$ whenever $i$ does not correspond to a noncyclic strongly connected component. 
Let there be given an \SLPEe $(U,V)$ from $-\iota_{\mathbf{r}}(-\Bsf_{E}^\bullet)$ to $-\iota_{\mathbf{r}'}(-\Bsf_{E'}^\bullet)$. 
Then $E\Meq E'$ and there exists a \calP-equivariant isomorphism $\Phi$ from $C^*(E)\otimes\K$ to $C^*(E')\otimes\K$ such that $\FKR(\calP;\Phi)=\FKRs(U,V)$. 
\end{theorem}
\begin{proof}
We may assume that $r_i\geq 3$ for every $i\in\calP$ corresponding to a noncyclic strongly connected component. 
If follows from Lemma~\ref{lem:we-can-put-it-in-canonical-form} and its proof, that there exist graphs $E_1$ and $E_1'$ in canonical form with $E\Meq E_1$, $E'\Meq E_1'$, $\Bsf_{E_1},\Bsf_{E_1'}\in \MPZccc[(\mathbf{m}+\mathbf{r})\times(\mathbf{n}+\mathbf{r})]$ for the multiindex $\mathbf{r}$, and a \calP-equivariant isomorphisms $\Phi_1$ from $C^*(E)\otimes\K$ to $C^*(E_1)\otimes\K$ and $\Phi_2$ from $C^*(E')\otimes\K$ to $C^*(E_1')\otimes\K$ such that $\FKR( \Phi_1)$ and $\FKR( \Phi_2)$ are induced by \SLPEe{s} from $-\iota_\mathbf{r}(-\Bsf_E^\bullet)$ to $\Bsf_{E_1}^\bullet$ and from $-\iota_{\mathbf{r}'}(-\Bsf_{E'}^\bullet)$ to $\Bsf_{E_1'}^\bullet$ respectively, where $\mathbf{r}'=\mathbf{m}+\mathbf{r}- \mathbf{m}'$.
Composing the \SLPEe{s} above, we get an \SLPEe from $\Bsf_{E_1}^\bullet$ to  $\Bsf_{E_1'}^\bullet$, so from Theorem~\ref{thm:BoyleTheorem4.4} we have that this \SLPEe is a positive equivalence from $\Bsf_{E_1}^\bullet$ to  $\Bsf_{E_1'}^\bullet$ through \MPplusZ.
According to Proposition~\ref{prop:toke}, there exists a series of stable \calP-equivariant isomorphisms corresponding to the series of basic positive equivalences through \MPplusZ and inducing exactly the same on reduced filtered $K$-theory. 
\end{proof}

\section{Cuntz splice and \texorpdfstring{\FKR}{reduced filtered K-theory}}
\label{sec:cuntzsplice}

We know that the moves \OO, \II, \RR, \SSS imply stable isomorphism, \cf\ Theorem~\cite[Theorem~2.7]{MR3759003} and the references there. 
It was also recently shown in \cite[Theorem~4.8]{MR3713535} that the move \CC implies stable isomorphism. 
But it is very important for our proof to know what the isomorphism from $C^*(E)\otimes\K$ to $C^*(E_{u,-})\otimes\K$ induces on $K$-theory. 
This was done in the case of the Cuntz-Krieger algebras in \cite{MR3624419}. In fact, the induced map on $K$-theory was described in general for the isomorphism from  $C^*(E_{u,-})\otimes\K$ to $C^*(E_{u,--})\otimes\K$, but the description between 
$C^*(E)\otimes\K$ and $C^*(E_{u,--})\otimes\K$ was only proved for Cuntz-Krieger algebras. Thus we generalize that part.   
This is an important result needed in Section~\ref{sec:crelle-trick} and as part of the proof of the other main theorems.

In \cite{MR3713535} it was shown that $E\Meq E_{u,--}$ and thus $C^*(E)\otimes\K\cong C^*(E_{u,--})\otimes\K$. We show how this isomorphism acts on $K$-theory. 

\begin{proposition} \label{prop:cuntzsplicetwice}
Let $E$ be a graph with finitely many vertices, and let $u$ be a regular vertex in $E^0$ that supports at least two distinct return paths.  Assume that $\Bsf_E \in \MPZccc$ and $u$ belongs to the block $j\in\calP$.  
Then $\Bsf_{E_{u,--}}$ is an element of $\MPZccc[(\mathbf{m}+4\mathbf{e}_j)\times(\mathbf{n}+4\mathbf{e}_j)]$, there is an \SLPEe $(U,V)$ from $-\iota_{4\mathbf{e}_j}(-\Bsf_E^\bullet)$ to  $\Bsf_{E_{u,--}}^\bullet$, and there exists a \calP-equivariant isomorphism $\Phi$ from $C^*(E)\otimes\K$ to $C^*(E_{u,--} )\otimes\K$ such that $\FKR(\calP;\Phi)=\FKRs(U,V)$. 
\end{proposition}
\begin{proof}
The $j$'th diagonal block of $\Bsf_{E_{u,--}}^\bullet$ can be described as the matrix 
$$\begin{pmatrix}
\Bsf_{E}^\bullet\{j\} & \begin{smallpmatrix} 0&0&0&0 \\ \vdots & \vdots & \vdots & \vdots \\ 1 & 0 & 0 & 0 \end{smallpmatrix} \\ 
\begin{smallpmatrix}0 & \cdots & 1 \\ 0 & \cdots & 0 \\ 0 & \cdots & 0 \\ 0 & \cdots & 0  \end{smallpmatrix} & \begin{smallpmatrix}0 & 1 & 1 & 0 \\ 1 & 0 & 0 & 0 \\ 1 & 0 & 0 & 1 \\ 0 & 0 & 1 & 0\end{smallpmatrix}
\end{pmatrix}$$
Let $U\in\SLPZ[\mathbf{m}+4\mathbf{e}_j]$ and $V\in\SLPZ[\mathbf{n}+4\mathbf{e}_j]$ be the identity matrix everywhere except for the $j$'th diagonal block where they are 
\begin{align*}
\begin{pmatrix}
I& \begin{smallpmatrix} 0&0&0&0 \\ \vdots & \vdots & \vdots & \vdots \\ 0 & 1 & 0 & 0 \end{smallpmatrix} \\ 
\begin{smallpmatrix}0 & \cdots & 0 \\ 0 & \cdots & 0 \\ 0 & \cdots & 0 \\ 0 & \cdots & 0  \end{smallpmatrix} & \begin{smallpmatrix}1 & 0 & 0 & 1 \\ 0 & 1 & 0 & 0 \\ 0 & 0 & 1 & 0 \\ 0 & 0 & 0 & 1\end{smallpmatrix}
\end{pmatrix} \\
\begin{pmatrix}
I& \begin{smallpmatrix} 0&0&0&0 \\ \vdots & \vdots & \vdots & \vdots \\ 0 & 0 & 0 & 0 \end{smallpmatrix} \\ 
\begin{smallpmatrix}0 & \cdots & -1 \\ 0 & \cdots & 0 \\ 0 & \cdots & 0 \\ 0 & \cdots & 0  \end{smallpmatrix} & \begin{smallpmatrix}0 & -1 & 0 & 0 \\ -1 & 0 & 0 & 0 \\ -1 & 0 & 0 & -1 \\ 0 & 0 & -1 & 0\end{smallpmatrix}
\end{pmatrix} ,
\end{align*}
respectively. It is easy to verify that these blocks have indeed determinant $1$ and that $(U,V)$ is an \SLPEe from $-\iota_{4\mathbf{e}_j}(-\Bsf_E^\bullet)$ to $\Bsf_{E_{u,--}}^\bullet$.  The existence of $\Phi$ follows from Theorem~\ref{thm:SLP-equivalence-implies-stable-isomorphism}.
\end{proof}

The existence of the isomorphism in the following theorem is shown in \cite{MR3713535}, while $K$-theory computation is carried out in \cite{MR3624419}.

\begin{theorem}\label{thm:cuntz-splice-1}
Let $E$ be a graph with finitely many vertices, and let $u$ be a regular vertex in $E^0$.  Assume that $\Bsf_E \in \MPZccc$ and that $u$ belongs to the block $j\in\calP$.  Let $\mathbf{r}=2\mathbf{e}_j$ and $\mathbf{r}'=4\mathbf{e}_j$.  Then there exist $U\in\SLPZ[\mathbf{m}+\mathbf{r}']$ and $V\in\GLPZ[\mathbf{n}+\mathbf{r}']$ that are the identity matrix everywhere except for the $j$'th diagonal block, all the diagonal blocks of $U$ and $V$ have determinant $1$ except for the $j$'th diagonal block of $V$, which has determinant $-1$, and $(U,V)$ is a \GLPEe from $-\iota_{\mathbf{r}}(-\Bsf_{E_{u,-}}^\bullet)$ to  $\Bsf_{E_{u,--}}^\bullet$. 
Moreover, there exists a $\calP$-equivariant isomorphism $\Phi$ from $C^{*}(E_{u,-})$ to $C^{*}(E_{u,--})$ such that $\FKR(\calP;\Phi)=\FKR(U,V)$. 
\end{theorem}

\begin{corollary}\label{cor:cuntzspliceinvariant}
Let $E$ be a graph with finitely many vertices, let $u$ be a regular vertex in $E^0$ that supports at least two distinct return paths, and assume that $\Bsf_E\in\MPZccc$. 
Assume, moreover, that $u$ belongs to the block $j\in\calP$. 
Let $\mathbf{r}=2\mathbf{e}_j$.  Then $\Bsf_{E_{u,-}}\in\MPZccc[(\mathbf{m}+\mathbf{r})\times(\mathbf{n}+\mathbf{r})]$, and there exist $U\in\SLPZ[\mathbf{m}+\mathbf{r}]$ and $V\in\GLPZ[\mathbf{n}+\mathbf{r}]$ such that the following holds. The matrices $U$ and $V$ are equal to the identity matrices everywhere except for the $j$'th diagonal block. All the diagonal blocks of $U$ and $V$ have determinant $1$ except for the $j$'th diagonal block of $V$, which has determinant $-1$ and $(U,V)$ is a \GLPEe from $-\iota_{\mathbf{r}}(-\Bsf_{E}^\bullet)$ to  $\Bsf_{E_{u,-}}^\bullet$. 
Moreover, there exists a $\calP$-equivariant isomorphism $\Phi$ from $C^{*}(E) \otimes \K$ to $C^{*}(E_{u,-}) \otimes \K$ such that $\FKR(\calP;\Phi)=\FKRs(U,V)$. 
\end{corollary}
\begin{proof}
It follows from Theorem~\ref{thm:cuntz-splice-1} and Proposition~\ref{prop:cuntzsplicetwice} that the statement holds if we say that the \calP-equivariant isomorphism $\Phi$ is induced by such a \GLPEe from $-\iota_{\mathbf{r}}(\iota_{\mathbf{r}}(-\Bsf_{E}^\bullet))$ to  $-\iota_{\mathbf{r}}(-\Bsf_{E_{u,-}}^\bullet)$. 
Let $(U_1,V_1)$ be the \SLPEe from the proof of Proposition~\ref{prop:cuntzsplicetwice} and let $(U_2,V_2)$ be the \GLPEe from the proof of Theorem~\ref{thm:cuntz-splice-1}. 
By computing $U_2^{-1}U_1$ and $V_1V_2^{-1}$ we see that these are in fact of the form $\iota_\mathbf{r}(U)$ and $\iota_\mathbf{r}(V)$, respectively, for matrices $U\in\SLPZ[\mathbf{m}+\mathbf{r}]$ and $V\in\GLPZ[\mathbf{n}+\mathbf{r}]$ of the desired form. 
\end{proof}

\section{\texorpdfstring{$\GLP$}{GLP}-equivalence to \texorpdfstring{$\SLP$}{SLP}-equivalence}
\label{sec:crelle-trick}

In this section we are concerned with \hyperlink{mainthm-step-GLtoSL}{Step 4} in our proof outline in Section~\ref{sec:main} of the proof of \ref{thm:main-1-item-3} implies \ref{thm:main-1-item-1} in Theorem \ref{thm:main-1}. 

The next proposition will be our first step in going from a \GLPEe to an \SLPEe (provided that only diagonal blocks corresponding to noncyclic strongly connected components can be non-\SL). 
The idea is to alter our graphs in such a way that their $\Bsf^{\bullet}$ matrices are \GLPE, say by $(U,V)$, but where all the diagonal blocks of $U$ have determinant one --- provided that the \GLPEe has the right determinant on the diagonal blocks corresponding to the cyclic components and on the diagonal blocks corresponding to the not strongly connected components. Consequently, all sign problems are moved to $V$. 

\begin{proposition}\label{prop:GLtoSL1}
Let $E_{1}$ and $E_{2}$ be graphs with finitely many vertices such that $(\Bsf_{E_{1}}, \Bsf_{E_{2}} )$ is in standard form.
Suppose $\ftn{ (U_{1},V_{1}) }{ \Bsf^{\bullet}_{E_{1}} }{ \Bsf^{\bullet}_{E_{2}} }$ is a \GLPEe, where $U_{1} \in \GLPZ[\mathbf{m}]$, $V_1 \in \GLPZ[\mathbf{n}]$. 
If, for some $i\in\calP$, $m_i> 1$, and
\begin{equation}
\label{eq:prop:GLtoSL1-eq1} \det ( U_{1} \{ i\} ) = -1,
\end{equation}
then there exist graphs $F_{1}$ and $F_{2}$ with finitely many vertices, there exist $U_{2} \in \GLPZ[\mathbf{m}'], V_2 \in \GLPZ[\mathbf{n}']$, where $\mathbf{m}' =\mathbf{m}+ 2\mathbf{e}_i$ and $\mathbf{n}' =\mathbf{n}+ 2\mathbf{e}_i$, and there exist $\calP$-equivariant isomorphisms $\Phi_1\colon C^*(E_1)\otimes\K\rightarrow C^*(F_1)\otimes\K$ and $\Phi_2\colon C^*(E_2)\otimes\K\rightarrow C^*(F_2)\otimes\K$, such that 
\begin{enumerate}[(1)]
 \item\label{enum:prop:GLtoSL1-1} $E_{k} \Meq F_{k}$, $k=1,2$;
 \item\label{enum:prop:GLtoSL1-2} $(\Bsf_{F_{1}} , \Bsf_{F_{2}} )$ is in standard form;
 \item\label{enum:prop:GLtoSL1-3} $(U_2,V_2)$ is a \GLPEe from $\Bsf^{\bullet}_{F_{1}}$ to $\Bsf^{\bullet}_{F_{2}}$;
 \item\label{enum:prop:GLtoSL1-4} $\det ( U_{2} \{ i\}  )= 1$, $\det ( V_{2} \{ i\} ) = - \det( V_{1} \{ i \} )$;
 \item\label{enum:prop:GLtoSL1-5} $\det( U_{2} \{ j\} ) = \det( U_{1} \{ j\} )$ and $\det( V_{2} \{ j\} ) = \det( V_{1} \{ j\} )$ for all $j \neq i$; and
 \item\label{enum:prop:GLtoSL1-6} $\FKR(\calP;\Phi_2)\circ\FKRs(U_1,V_1) =\FKRs(U_2,V_2)\circ\FKR(\calP;\Phi_1)$.
\end{enumerate}
\end{proposition}

\begin{proof}
Let $\mathbf{r}=2\mathbf{e}_i$.  For each $j=1,2$ use the last part of Proposition~\ref{prop:edge-expansion} to find $U_j'\in\SLPZ[\mathbf{m}']$ and $V_j'\in\SLPZ[\mathbf{n}']$ that are identity matrices everywhere except in the $i$'th diagonal block such that $(U_j',V_j')$ is an \SLPEe from $-\iota_{\mathbf{r}}(-\Bsf_{E_j}^\bullet)$ to $\Bsf_{E_j'}^\bullet$ and a $\calP$-equivariant isomorphism $\Phi_j'$ from $C^*(E_j)\otimes\K$ to $C^*(E_j')\otimes\K$ satisfying $\FKR(\calP;\Phi_j')=\FKRs(U_j',V_j')$, where we let $E_j'$ be the graph obtained by simple expansions of $E_j$ twice. 

Now we let $U\in\GLPZ[\mathbf{m}']$ and $V\in\GLPZ[\mathbf{n}']$ be equal to the identity matrices except for the $i$'th diagonal block where they are 
$$\begin{pmatrix}
I_{m_i} & 0 \\
0 & \begin{smallpmatrix} 0 & 1 \\ 1 & 0 \end{smallpmatrix}
\end{pmatrix}
\quad\text{and}\quad
\begin{pmatrix}
I_{n_i} & 0 \\
0 & \begin{smallpmatrix} 0 & 1 \\ 1 & 0 \end{smallpmatrix}
\end{pmatrix},$$
respectively. 
Now $(U\iota_{\mathbf{r}}(U_1),\iota_{\mathbf{r}}(V_1)V)$ is a \GLPEe from $-\iota_{\mathbf{r}}(-\Bsf_{E_1})$ to $-\iota_{\mathbf{r}}(-\Bsf_{E_2})$ inducing exactly $\FKR(U,V)$. So this changes the signs of the determinant as wanted. So using an argument similar to the argument for the proof of the last part of Lemma~\ref{lem:we-can-put-it-in-canonical-form} we can get a pair of graphs $(F_1, F_2)$ with the desired properties.
\end{proof}

We now use the Cuntz splice to fix potential sign problems on $V$ --- for the strongly connected components that are not cyclic.

\begin{proposition}\label{prop:GLtoSL2}
In Proposition~\ref{prop:GLtoSL1}, if we replace Equation~\eqref{eq:prop:GLtoSL1-eq1} by 
\begin{equation}
\det ( U_{1} \{ i\} ) = 1\quad\text{and}\quad\det( V_{1} \{i\} ) = -1
\end{equation} 
in the assumptions, then the proposition still holds with Conditions~\ref{enum:prop:GLtoSL1-1} and~\ref{enum:prop:GLtoSL1-4} replaced by 
\begin{enumerate}[(1)]
\item[(1')] $E_{k} \MCeq F_{k}$, for $k=1,2$, and 
\item[(4')] $\det ( U_{2} \{ i \} )  = \det ( V_{2} \{ i \} ) = 1$,
\end{enumerate} 
respectively, in the conclusion. 
\end{proposition}

\begin{proof}
Let $\mathbf{r}=2\mathbf{e}_i$.  We use the last part of Proposition~\ref{prop:edge-expansion} to find $U_2'\in\SLPZ[\mathbf{m}']$ and $V_2'\in\SLPZ[\mathbf{n}']$ that are identity matrices everywhere except in the $i$'th diagonal block such that $(U_2',V_2')$ is an \SLPEe from $-\iota_{\mathbf{r}}(-\Bsf_{E_2}^\bullet)$ to $\Bsf_{E_2'}^\bullet$ and a $\calP$-equivariant isomorphism $\Phi_2'$ from $C^*(E_2)\otimes\K$ to $C^*(E_2')\otimes\K$ satisfying $\FKR(\calP;\Phi_2')=\FKRs(U_2',V_2')$, where we let $E_2'$ be the graph obtained by simple expansions of $E_2$ twice. 

We can choose a regular vertex $u$ in the component of $E_1$ corresponding to $i\in\calP$ that supports at least two distinct return paths. 
From Corollary~\ref{cor:cuntzspliceinvariant} it follows that $\Bsf_{(E_1)_{u,-}}\in\MPZccc[(\mathbf{m}+\mathbf{r})\times(\mathbf{n}+\mathbf{r})]$, and that there exist $U_1'\in\SLPZ[\mathbf{m}+\mathbf{r}]$ and $V_1'\in\GLPZ[\mathbf{n}+\mathbf{r}]$ such that the following holds. The matrices $U_1'$ and $V_1$ are equal to the identity matrices everywhere except for the $i$'th diagonal block. All the diagonal blocks of $U_1'$ and $V_1'$ have determinant $1$ except for the $i$'th diagonal block of $V$, which has determinant $-1$ and $(U_1',V_1')$ is a \GLPEe from $-\iota_{\mathbf{r}}(-\Bsf_{E_1}^\bullet)$ to  $\Bsf_{(E_1)_{u,-}}^\bullet$. 
Moreover, there exists a $\calP$-equivariant isomorphism $\Phi_1'$ from $C^{*}(E_1) \otimes \K$ to $C^{*}((E_1)_{u,-}) \otimes \K$ such that $\FKR(\calP;\Phi)=\FKRs(U_1',V_1')$. 

So this changes the signs of the determinant as wanted. So using an argument similar to the argument for the proof of the last part of Lemma~\ref{lem:we-can-put-it-in-canonical-form} we can get a pair of graphs $(F_1, F_2)$ with the desired properties. 
\end{proof}

We now have all we need to modify a \GLPEe to an \SLPEe{} --- if the determinants of all the components that do not correspond to a non-cyclic strongly component are $1$.

\begin{theorem}\label{thm:GLtoSL}
Let $E_{1}$ and $E_{2}$ be graphs with finitely many vertices such that the pair $(\Bsf_{E_{1}} , \Bsf_{E_{2}} )$ is in standard form.  
Suppose $(U,V)$ is a \GLPEe from $\Bsf^{\bullet}_{E_{1}}$ to $\Bsf^{\bullet}_{E_{2}}$ satisfying that $U\{i\}=1$ whenever $m_i=1$ and $V\{i\}=1$ whenever $n_i=1$.
Then there exist graphs $F_{1}$ and $F_{2}$, $U'\in\SLPZ[\mathbf{m}']$,$V'\in\SLPZ[\mathbf{n}']$, and $\calP$-equivariant isomorphisms $\Phi_1\colon C^*(E_1)\otimes\K\rightarrow C^*(F_1)\otimes\K$ and $\Phi_2\colon C^*(E_2)\otimes\K\rightarrow C^*(F_2)\otimes\K$ such that $E_{j} \MCeq F_{j}$, for $j=1,2$, the pair $(\Bsf_{F_1}, \Bsf_{F_2})$ is in standard form, $(U',V')$ is an \SLPEe from $\Bsf^{\bullet}_{F_{1}}$ to $\Bsf^{\bullet}_{F_{2}}$, and 
 \begin{equation*}
 \FKR(\calP;\Phi_2)\circ\FKRs(U,V) =\FKRs(U',V')\circ\FKR(\calP;\Phi_1). 
 \end{equation*}
\end{theorem} 

\begin{proof}
The theorem follows from an argument similar to the argument in \cite[Theorem~6.8]{MR2270572} with Propositions~\ref{prop:GLtoSL1} and~\ref{prop:GLtoSL2} in place of \cite[Lemma~6.7]{MR2270572}. 

Briefly, the idea is that we are given a \GLPEe, say $(U,V)$. 
We go down the diagonal blocks and for each of them use Proposition~\ref{prop:GLtoSL1} if necessary to make sure the $U$ has positive determinants of the diagonal blocks. 
Then we go down the diagonal blocks again this time using Proposition~\ref{prop:GLtoSL2} to fix the determinant of the diagonal blocks of $V$ when necessary. 
\end{proof}

Using Theorem~\ref{thm:SLP-equivalence-implies-stable-isomorphism}, we get the following immediate consequence. 

\begin{corollary}\label{cor:GLtoSL}
Let $E_{1}$ and $E_{2}$ be graphs with finitely many vertices such that the pair $(\Bsf_{E_{1}} , \Bsf_{E_{2}} )$ is in standard form.  
Suppose $(U,V)$ is a \GLPEe from $\Bsf^{\bullet}_{E_{1}}$ to $\Bsf^{\bullet}_{E_{2}}$ satisfying that $U\{i\}=1$ whenever $m_i=1$ and $V\{i\}=1$ whenever $n_i=1$.

Then $E_1\MCeq E_2$ and there exists a \calP-equivariant isomorphism $\Phi$ from $C^*(E_1)\otimes\K$ to $C^*(E_2)\otimes\K$ such that $\FKR(\calP;\Phi)=\FKRs(U,V)$. 
\end{corollary}

\section{The \Pul move implies stable isomorphism}
\label{sec:MoveP}

In this section we will show that Move \PP\ implies stable isomorphism of the associated graph \cas. 
We do that by first showing that it implies  stable isomorphism of the associated graph \cas in the case where the graph is assumed to be in a very specific form. Then we use the invariance of the Cuntz splice to show that the general Move \PP\ implies stable isomorphism. Along the way, we also partly keep track of the induced maps on $K$-theory (on the gauge simple subquotients corresponding to the cyclic components).
The strategy of proving the invariance of Move \PP\ (in the case when the graph is assumed to be in a very specific form) is --- in principle --- the same as the proof of the invariance of the Cuntz splice. We prove that doing Move \PP\ twice yields a graph that is move equivalent to the original graph. Moreover, we prove that doing Move \PP\ once and twice yields isomorphic graph \cas. This is done by gluing two different relative graph \cas onto a graph \ca and see that the resulting graph \cas are isomorphic. 
As for the Cuntz splice, we need to show that these two relative graph \cas are isomorphic in such a way that certain projections are Murray-von~Neumann equivalent. 
The entire Section~\ref{subsec:classification-result} is devoted to proving this by appealing to various general classification results and techniques --- this is a very crucial step in applying the above strategy. 
In Section~\ref{subsec:inv-spec-pulelehua} we then prove the invariance of the specialized Move \PP. 
In Section~\ref{subsec:transform-to-apply-spec-pulelehua} we show that we can bring our graphs on the special form, so that we can apply the specialized Move \PP. In Section~\ref{subsec:invariance-of-pulelehua} we show that the general Move \PP\ can be reduced to applying a specialized  Move \PP\ and a series of Cuntz move equivalences --- thus proving the invariance of the general Move \PP.

\begin{assumption}\label{assumption:1-hash}
In this section, we will frequently assume that we have a graph satisfying the conditions in this assumption, by referring to it. Let $E = ( E^{0}, E^{1} , r_{E}, s_{E} )$ be a graph with finitely many vertices and we let $u_0$ be a vertex of $E$ satisfying the following conditions. 
\begin{enumerate}[(1)]
\item $\Bsf_{E}\in\MPZcc$, 
\item $\{u_0\}$ is a cyclic component,
\item $u_0$ emits at least two edges, 
\item $u_0$ only emits edges to non-cyclic strongly connected components, 
\item $u_0$ only emits edges to components that are immediate successors of $\{u_0\}$, 
\item $u_0$ emits exactly one edge to each component that is an immediate successor of $\{u_0\}$, 
\item each vertex in $r_E(s_E^{-1}(u_0))\setminus\{u_0\}$ has exactly one loop and it has exactly one other edge going out (to a vertex that is part of the same component),
\item each vertex in $r_E(s_E^{-1}(u_0))\setminus\{u_0\}$ has only incoming edges from other vertices belonging to the same component (except for the edge from $u_0$). 
\end{enumerate}
\end{assumption}

\begin{notation}\label{notation:OnceAndTwice-hash}
For each vertex $u$ in a graph, we let $\mathbf{E}^{u}_*$ and $\mathbf{E}^{u}_{**}$ denote the graphs: 
\begin{align*}
\mathbf{E}^u_* \  = \ \ \ \ \xymatrix{
  \bullet^{v_1^u} \ar@(ul,ur)[]^{e_{1}^u} \ar@/^/[r]^{e_{2}^u} & \bullet^{v_2^u} \ar@(ul,ur)[]^{e_{4}^u} \ar@/^/[l]^{e_{3}^u}
}
\end{align*}
\begin{align*}
\mathbf{E}_{**}^u \  =  \ \ \ \ \xymatrix{
	\bullet^{ w_{1}^u } \ar@(ul,ur)[]^{f_{1}^u}  \ar@/^/[r]^{ f_{2}^u } & \bullet^{ w_{2}^u } \ar@(ul,ur)[]^{f_{4}^u} \ar@/^/[r]^{ f_{5}^u }  \ar@/^/[l]^{f_{3}^u} &  \bullet^{w_3^u} 				\ar@(ul,ur)[]^{f_{7}^u} \ar@/^/[r]^{f_{8}^u} \ar@/^/[l]^{f_{6}^u}
	& \bullet^{w_4^u} \ar@(ul,ur)[]^{f_{10}^u} \ar@/^/[l]^{f_{9}^u}
	}
\end{align*}
Note that these graphs are isomorphic to the graphs $\mathbf{E}_{*}$ and $\mathbf{E}_{**}$, respectively, introduced in \cite[Notation~4.1]{MR3713535}. Here we have relabelled the vertices and edges and we have made a family of distinct copies of these graphs indexed over the vertices of the given graph. 
\end{notation}

\begin{notation}\label{notation:OnceAndTwice-hash-2}
Let $E = ( E^{0}, E^{1} , r_{E}, s_{E} )$ and $u_0$ be as in Assumption~\ref{assumption:1-hash}. 
Then we let $E_{u_0,\#}$ denote the graph described as follows:
\begin{align*}
E_{u_0,\#}^{0} &= E^{0} \sqcup\bigsqcup_{u\in r_E(s_E^{-1}(u_0))\setminus\{u_0\}} (\mathbf{E}^u_{*})^{0} \\
E_{u_0,\#}^{1} &= E^{1} \sqcup\bigsqcup_{u\in r_E(s_E^{-1}(u_0))\setminus\{u_0\}} \left((\mathbf{E}_{*}^u)^{1} \sqcup \{ e_5^u, e_6^u,e_7^u,e_8^u \}\right)
\end{align*}
with $r_{E_{u_0,\#}} \vert_{E^{1}} = r_{E}$, $s_{E_{u_0,\#}} \vert_{ E^{1} } = s_{E}$, $r_{E_{u_0,\#}} \vert_{(\mathbf{E}_{*}^u)^{1}} = r_{\mathbf{E}_{*}^u}$, $s_{E_{u_0,\#}} \vert_{(\mathbf{E}_{*}^u)^{1}} = s_{\mathbf{E}_{*}^u}$, and
\begin{align*}
	s_{E_{u_0,\#}}(e_5^u) &= u	& r_{E_{u_0,\#}}(e_5^u) &= v_{1}^u \\
	s_{E_{u_0,\#}}(e_6^u) &= v_1^u	& r_{E_{u_0,\#}}(e_6^u) &= u, \\
	s_{E_{u_0,\#}}(e_7^u) &= u_0	& r_{E_{u_0,\#}}(e_7^u) &= v_{2}^u, \\
	s_{E_{u_0,\#}}(e_8^u) &= u_0	& r_{E_{u_0,\#}}(e_8^u) &= v_{2}^u,
	\end{align*}
for every $u\in r_E(s_E^{-1}(u_0))\setminus\{u_0\}$. 
Moreover, we let $E_{u_0,\#\#}$ denote the graph described as follows:
\begin{align*}
E_{u_0,\#\#}^{0} &= E^{0} \sqcup \bigsqcup_{u\in r_E(s_E^{-1}(u_0))\setminus\{u_0\}} (\mathbf{E}^u_{**})^{0} \\
E_{u_0,\#\#}^{1} &= E^{1} \sqcup \bigsqcup_{u\in r_E(s_E^{-1}(u_0))\setminus\{u_0\}} \left((\mathbf{E}_{**}^u)^{1} \sqcup \{ f_{11}^u, f_{12}^u,f_{13}^u,f_{14}^u,f_{15}^u,f_{16}^u \}\right)
\end{align*}
with $r_{E_{u_0,\#\#}} \vert_{E^{1}} = r_{E}$, $s_{E_{u_0,\#\#}} \vert_{ E^{1} } = s_{E}$, $r_{E_{u_0,\#\#}} \vert_{(\mathbf{E}_{**}^u)^{1}} = r_{\mathbf{E}_{**}^u}$, $s_{E_{u_0,\#\#}} \vert_{(\mathbf{E}_{**}^u)^{1}} = s_{\mathbf{E}_{**}^u}$, and
\begin{align*}
	s_{E_{u_0,\#\#}}(f_{11}^u) &= u		& r_{E_{u_0,\#\#}}(f_{11}^u) &= w_{1}^u \\
	s_{E_{u_0,\#\#}}(f_{12}^u) &= w_1^u	& r_{E_{u_0,\#\#}}(f_{12}^u) &= u, \\
	s_{E_{u_0,\#}}(f_{13}^u) &= u_0	& r_{E_{u_0,\#}}(f_{13}^u) &= w_{2}^u, \\
	s_{E_{u_0,\#}}(f_{14}^u) &= u_0	& r_{E_{u_0,\#}}(f_{14}^u) &= w_{2}^u, \\
    s_{E_{u_0,\#}}(f_{15}^u) &= u_0	& r_{E_{u_0,\#}}(f_{15}^u) &= w_{4}^u, \\
	s_{E_{u_0,\#}}(f_{16}^u) &= u_0	& r_{E_{u_0,\#}}(f_{16}^u) &= w_{4}^u,
\end{align*}
for every $u\in r_E(s_E^{-1}(u_0))\setminus\{u_0\}$. 
For each $u\in r_E(s_E^{-1}(u_0))\setminus\{u_0\}$, we let $c_0^u$, $c_1^u$ and $c_2^u$ denote the unique arrow from $u_0$ to $u$, the loop based at $u$ and the unique arrow out of $u$, respectively. 
We let $e_0$ denote the loop based at $u_0$.

Note that $E_{u_0,\#}$ is (isomorphic to) the graph $E_{u_0,P}$ --- so this is a specialised move, that is only allowed in the much more restrictive settings of Assumption~\ref{assumption:1-hash} compared to the settings in Definition~\ref{def:hashmove}. 
\end{notation}

\begin{notation}\label{notation:OnceAndTwice-secondtime-hash}
Let $E = ( E^{0}, E^{1} , r_{E}, s_{E} )$ and $u_0$ be as in Assumption~\ref{assumption:1-hash}. 
Then we let $\underline{\mathbf{E}}_{u_0,\#}$ denote the graph where we remove all the vertices in $E^0\setminus r_E(s_E^{-1}(u_0))$ from $E_{u_0,\#}$ (and all corresponding edges going in or out from these vertices). 
Moreover, we let $\uuline{\mathbf{E}}_{u_0,\#\#}$ denote the graph where we remove all the vertices in $E^0\setminus r_E(s_E^{-1}(u_0))$ from $E_{u_0,\#\#}$ (and all corresponding edges going in or out from these vertices). 
We denote the vertices and edges in these graph by underlining them (in order to tell them apart from the others) --- like $\uline{u}_0$, $\uuline{u}_0$, $\uline{u}$, $\uuline{u}$, $\uline{v}_i^u$, $\uuline{w}_i^u$, $\uline{e}_0$, $\uuline{e}_0$, $\uline{e}_i^u$, $\uuline{f}_i^u$, $\uline{c}_0^u$, $\uline{c}_1^u$, $\uuline{c}_0^u$, $\uuline{c}_1^u$, for $u\in r_E(s_E^{-1}(u_0))\setminus\{u_0\}$. 
\end{notation}

\begin{example}\label{example:hash}
Consider the graph $E$:
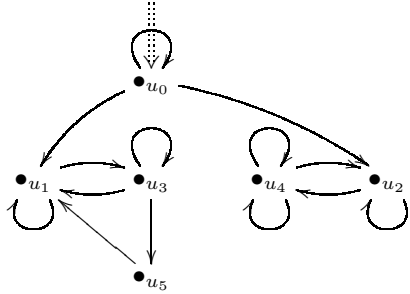
\begin{align*}
\xymatrix{&\ar@{:>}[d] \\
&\bullet_{u_0} \ar@(ul,ur)\ar@/^/[drr]\ar@/_/[dl] \\
\ar@/^/[r] \bullet_{u_1}\ar@(dr,dl)[] & \bullet_{u_3} \ar@(ul,ur) \ar@/^/[l]\ar[d] &  \bullet_{u_4} \ar@(ul,ur) \ar@(dr,dl)\ar@/^/[r] & \ar@/^/[l] \bullet_{u_2}\ar@(dr,dl)[] \\
& \bullet_{u_5}\ar[ul]
}
\end{align*}
Then $E_{u_0,\#}$ is the graph:
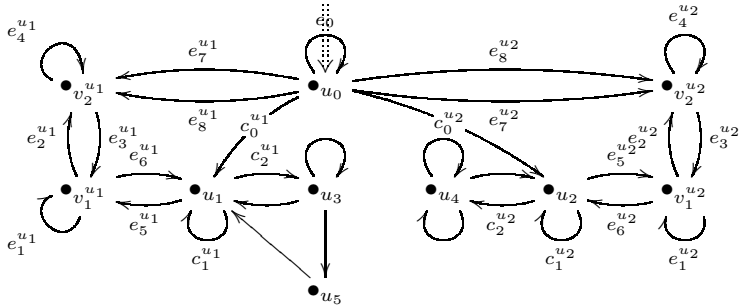
\begin{align*}
\xymatrix{
&&\ar@{:>}[d] \\
\bullet_{v_2^{u_1}}\ar@(l,u)[]^{e_4^{u_1}}\ar@/^/[d]^{e_3^{u_1}} && \bullet_{u_0}\ar@/^/[ll]^{e_8^{u_1}} \ar@/_/[ll]_{e_7^{u_1}} \ar@(ul,ur)[]^{e_0}\ar@/^/[drr]|{c_0^{u_2}}\ar@/_/[dl]|{c_0^{u_1}} \ar@/^/[rrr]^{e_8^{u_2}} \ar@/_/[rrr]_{e_7^{u_2}} &&& \bullet_{v_2^{u_2}} \ar@(ul,ur)[]^{e_4^{u_2}} \ar@/^/[d]^{e_3^{u_2}}  \\
\bullet_{v_1^{u_1}}\ar@(d,l)[]^{e_1^{u_1}}\ar@/^/[u]^{e_2^{u_1}}\ar@/^/[r]^{e_6^{u_1}}&\ar@/^/[r]^{c_2^{u_1}} \bullet_{u_1}\ar@(dr,dl)[]^{c_1^{u_1}}\ar@/^/[l]^{e_5^{u_1}} & \bullet_{u_3} \ar@(ul,ur) \ar@/^/[l]\ar[d] &  \bullet_{u_4} \ar@(ul,ur) \ar@(dr,dl)\ar@/^/[r] & \ar@/^/[l]^{c_2^{u_2}} \bullet_{u_2}\ar@(dr,dl)[]^{c_1^{u_2}}\ar@/^/[r]^{e_5^{u_2}} & \bullet_{v_1^{u_2}}\ar@(dr,dl)[]^{e_1^{u_2}} \ar@/^/[l]^{e_6^{u_2}}\ar@/^/[u]^{e_2^{u_2}} \\
&& \bullet_{u_5}\ar[ul]
}
\end{align*}
and $E_{u_0,\#\#}$ is the graph:
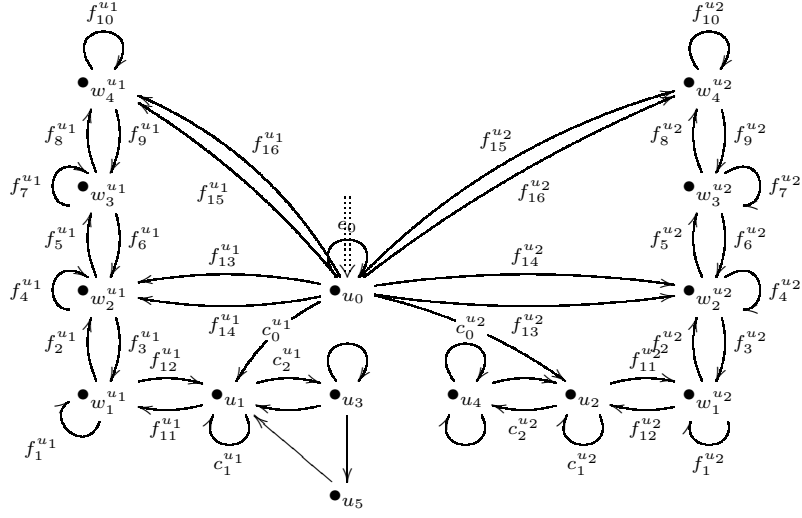
\begin{align*}
\xymatrix{
\bullet_{w_4^{u_1}}\ar@(ul,ur)[]^{f_{10}^{u_1}}\ar@/^/[d]^{f_9^{u_1}} &&&&& \bullet_{w_4^{u_2}}\ar@(ul,ur)[]^{f_{10}^{u_2}}\ar@/^/[d]^{f_9^{u_2}} \\
\bullet_{w_3^{u_1}}\ar@(dl,ul)[]^{f_{7}^{u_1}}\ar@/^/[u]^{f_8^{u_1}}\ar@/^/[d]^{f_6^{u_1}} &&\ar@{:>}[d] &&& \bullet_{w_3^{u_2}}\ar@(ur,dr)[]^{f_{7}^{u_2}}\ar@/^/[u]^{f_8^{u_2}}\ar@/^/[d]^{f_6^{u_2}} \\
\bullet_{w_2^{u_1}}\ar@/^/[u]^{f_5^{u_1}}\ar@(dl,ul)[]^{f_4^{u_1}}\ar@/^/[d]^{f_3^{u_1}} && \bullet_{u_0} \ar@/_/[uull]^{f_{15}^{u_1}} \ar@/_1pc/[uull]_{f_{16}^{u_1}} \ar@/^1pc/[uurrr]^{f_{15}^{u_2}} \ar@/^/[uurrr]_{f_{16}^{u_2}} \ar@/^/[ll]^{f_{14}^{u_1}} \ar@/_/[ll]_{f_{13}^{u_1}} \ar@(ul,ur)[]^{e_0}\ar@/^/[drr]|{c_0^{u_2}}\ar@/_/[dl]|{c_0^{u_1}} \ar@/^/[rrr]^{f_{14}^{u_2}} \ar@/_/[rrr]_{f_{13}^{u_2}} &&& \bullet_{w_2^{u_2}} \ar@(ur,dr)[]^{f_4^{u_2}} \ar@/^/[d]^{f_3^{u_2}} \ar@/^/[u]^{f_5^{u_2}} \\
\bullet_{w_1^{u_1}}\ar@(d,l)[]^{f_1^{u_1}}\ar@/^/[u]^{f_2^{u_1}}\ar@/^/[r]^{f_{12}^{u_1}}&\ar@/^/[r]^{c_2^{u_1}} \bullet_{u_1}\ar@(dr,dl)[]^{c_1^{u_1}}\ar@/^/[l]^{f_{11}^{u_1}} & \bullet_{u_3} \ar@(ul,ur) \ar@/^/[l]\ar[d] &  \bullet_{u_4} \ar@(ul,ur) \ar@(dr,dl)\ar@/^/[r] & \ar@/^/[l]^{c_2^{u_2}} \bullet_{u_2}\ar@(dr,dl)[]^{c_1^{u_2}}\ar@/^/[r]^{f_{11}^{u_2}} & \bullet_{w_1^{u_2}}\ar@(dr,dl)[]^{f_1^{u_2}} \ar@/^/[l]^{f_{12}^{u_2}}\ar@/^/[u]^{f_2^{u_2}} \\
&& \bullet_{u_5}\ar[ul]
}
\end{align*}
Moreover, $\uline{\mathbf{E}}_{u_0,\#}$ is the graph:
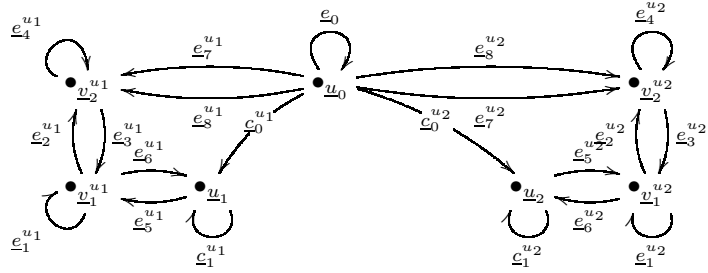
\begin{align*}
\xymatrix{
\bullet_{\uline{v}_2^{u_1}}\ar@(l,u)[]^{\uline{e}_4^{u_1}}\ar@/^/[d]^{\uline{e}_3^{u_1}} && \bullet_{\uline{u}_0}\ar@/^/[ll]^{\uline{e}_8^{u_1}} \ar@/_/[ll]_{\uline{e}_7^{u_1}} \ar@(ul,ur)[]^{\uline{e}_0}\ar@/^/[drr]|{\uline{c}_0^{u_2}}\ar@/_/[dl]|{\uline{c}_0^{u_1}} \ar@/^/[rrr]^{\uline{e}_8^{u_2}} \ar@/_/[rrr]_{\uline{e}_7^{u_2}} &&& \bullet_{\uline{v}_2^{u_2}} \ar@(ul,ur)[]^{\uline{e}_4^{u_2}} \ar@/^/[d]^{\uline{e}_3^{u_2}}  \\
\bullet_{\uline{v}_1^{u_1}}\ar@(d,l)[]^{\uline{e}_1^{u_1}}\ar@/^/[u]^{\uline{e}_2^{u_1}}\ar@/^/[r]^{\uline{e}_6^{u_1}} & \bullet_{\uline{u}_1}\ar@(dr,dl)[]^{\uline{c}_1^{u_1}}\ar@/^/[l]^{\uline{e}_5^{u_1}} &  &  &  \bullet_{\uline{u}_2}\ar@(dr,dl)[]^{\uline{c}_1^{u_2}}\ar@/^/[r]^{\uline{e}_5^{u_2}} & \bullet_{\uline{v}_1^{u_2}}\ar@(dr,dl)[]^{\uline{e}_1^{u_2}} \ar@/^/[l]^{\uline{e}_6^{u_2}}\ar@/^/[u]^{\uline{e}_2^{u_2}} 
}
\end{align*}
and $\uuline{\mathbf{E}}_{u_0,\#\#}$ is the graph:
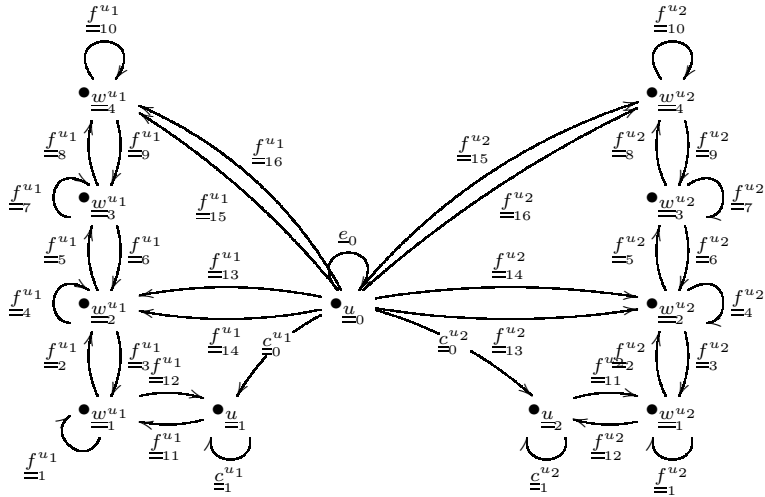
\begin{align*}
\xymatrix{
\bullet_{\uuline{w}_4^{u_1}}\ar@(ul,ur)[]^{\uuline{f}_{10}^{u_1}}\ar@/^/[d]^{\uuline{f}_9^{u_1}} &&&&& \bullet_{\uuline{w}_4^{u_2}}\ar@(ul,ur)[]^{\uuline{f}_{10}^{u_2}}\ar@/^/[d]^{\uuline{f}_9^{u_2}} \\
\bullet_{\uuline{w}_3^{u_1}}\ar@(dl,ul)[]^{\uuline{f}_{7}^{u_1}}\ar@/^/[u]^{\uuline{f}_8^{u_1}}\ar@/^/[d]^{\uuline{f}_6^{u_1}} &&&&& \bullet_{\uuline{w}_3^{u_2}}\ar@(ur,dr)[]^{\uuline{f}_{7}^{u_2}}\ar@/^/[u]^{\uuline{f}_8^{u_2}}\ar@/^/[d]^{\uuline{f}_6^{u_2}} \\
\bullet_{\uuline{w}_2^{u_1}}\ar@/^/[u]^{\uuline{f}_5^{u_1}}\ar@(dl,ul)[]^{\uuline{f}_4^{u_1}}\ar@/^/[d]^{\uuline{f}_3^{u_1}} && \bullet_{\uuline{u}_0} \ar@/_/[uull]^{\uuline{f}_{15}^{u_1}} \ar@/_1pc/[uull]_{\uuline{f}_{16}^{u_1}} \ar@/^1pc/[uurrr]^{\uuline{f}_{15}^{u_2}} \ar@/^/[uurrr]_{\uuline{f}_{16}^{u_2}} \ar@/^/[ll]^{\uuline{f}_{14}^{u_1}} \ar@/_/[ll]_{\uuline{f}_{13}^{u_1}} \ar@(ul,ur)[]^{\uline{e}_0}\ar@/^/[drr]|{\uuline{c}_0^{u_2}}\ar@/_/[dl]|{\uuline{c}_0^{u_1}} \ar@/^/[rrr]^{\uuline{f}_{14}^{u_2}} \ar@/_/[rrr]_{\uuline{f}_{13}^{u_2}} &&& \bullet_{\uuline{w}_2^{u_2}} \ar@(ur,dr)[]^{\uuline{f}_4^{u_2}} \ar@/^/[d]^{\uuline{f}_3^{u_2}} \ar@/^/[u]^{\uuline{f}_5^{u_2}} \\
\bullet_{\uuline{w}_1^{u_1}}\ar@(d,l)[]^{\uuline{f}_1^{u_1}}\ar@/^/[u]^{\uuline{f}_2^{u_1}}\ar@/^/[r]^{\uuline{f}_{12}^{u_1}}&  \bullet_{\uuline{u}_1}\ar@(dr,dl)[]^{\uuline{c}_1^{u_1}}\ar@/^/[l]^{\uuline{f}_{11}^{u_1}} &  &  &  \bullet_{\uuline{u}_2}\ar@(dr,dl)[]^{\uuline{c}_1^{u_2}}\ar@/^/[r]^{\uuline{f}_{11}^{u_2}} & \bullet_{\uuline{w}_1^{u_2}}\ar@(dr,dl)[]^{\uuline{f}_1^{u_2}} \ar@/^/[l]^{\uuline{f}_{12}^{u_2}}\ar@/^/[u]^{\uuline{f}_2^{u_2}} }
\end{align*}
The idea for the \emph{Pulelehua move} came during a research workshop that the
first three authors had in Hawaii. To us the graph $\uuline{\mathbf{E}}_{u_0,\#\#}$ above looks like a butterfly, so we chose the name \emph{Pulelehua}, which is the Hawaiian name for the \emph{Kamehameha butterfly}, a certain Hawaiian butterfly.
\end{example}

\begin{proposition} \label{prop:hashmovetwice}
Let $E$ be a graph with finitely many vertices, and let $u_0$ be a vertex of $E$ such that these satisfy Assumption~\ref{assumption:1-hash}. 
So $\Bsf_E\in\MPZcc$ for some $\mathbf{m},\mathbf{n}\in\N_0^N$ and some $\calP$ satisfying Assumption~\ref{ass:preorder} (\cf\ Assumption~\ref{assumption:1-hash}).
Assume, moreover, that $u_0$ belongs to the block $j\in\calP$. 

Let $\mathbf{r}=(r_i)_{i=1}^N$, where $r_i=4$ for all immediate successors $i$ of $j$, and $r_i=0$ otherwise. 
Then 
$\Bsf_{E_{u_0,\#\#}}\in\MPZcc[(\mathbf{m}+\mathbf{r})\times(\mathbf{n}+\mathbf{r})]$ 
and there is an \SLPEe from $-\iota_{\mathbf{r}}(-\Bsf_E^\bullet)$ to $\Bsf_{E_{u_0,\#\#}}^\bullet$. 
If, moreover, $\Bsf_E\in\MPZccc$, then $\Bsf_{E_{v,\#\#}}\in\MPZccc[(\mathbf{m}+\mathbf{r})\times(\mathbf{n}+\mathbf{r})]$. 
\end{proposition}
\begin{proof}
It is clear that 
$\Bsf_{E_{u_0,\#\#}}\in\MPZcc[(\mathbf{m}+\mathbf{r})\times(\mathbf{n}+\mathbf{r})]$. 
It is also clear, that $\Bsf_{E_{u_0,\#\#}}\in\MPZccc[(\mathbf{m}+\mathbf{r})\times(\mathbf{n}+\mathbf{r})]$ whenever $\Bsf_E\in\MPZccc$.

We construct matrices $U\in\SLPZ[\mathbf{m}+\mathbf{r}]$ and  $V\in\SLPZ[\mathbf{n}+\mathbf{r}]$ as follows.
For each immediate successor $i\in\calP$ of $j$, we set
\begin{align*}
U\{i\}&:= \begin{pmatrix}
I & 
\begin{smallpmatrix} 
0 & 0 & 0 & 0 \\ 
\vdots & \vdots & \vdots & \vdots \\ 
0 & 1 & 0 & 0 
\end{smallpmatrix} \\ 
\begin{smallpmatrix} 
0 & \cdots & 0 \\ 
0 & \cdots & 0 \\ 
0 & \cdots & 0 \\ 
0 & \cdots & 0 
\end{smallpmatrix} & 
\begin{smallpmatrix} 
1 & 0 & 0 & 0 \\ 
0 & 1 & 0 & 1 \\ 
0 & 0 & 1 & 0 \\ 
0 & 0 & 0 & 1 
\end{smallpmatrix}
\end{pmatrix}, \\ 
U\{j,i\}&:=\begin{pmatrix}
\begin{smallpmatrix} 
0&\cdots & 0 
\end{smallpmatrix} & 
\begin{smallpmatrix} 
0 & 0 & 2 & 0
\end{smallpmatrix} 
\end{pmatrix}, \\
V\{i\}&:= \begin{pmatrix}
I & 
\begin{smallpmatrix} 
0 & 0 & 0 & 0 \\ 
\vdots & \vdots & \vdots & \vdots \\ 
0 & 0 & 0 & 0 
\end{smallpmatrix} \\ 
\begin{smallpmatrix} 
0 & \cdots & -1 \\ 
0 & \cdots & 0 \\ 
0 & \cdots & 0 \\ 
0 & \cdots & 0 
\end{smallpmatrix} & 
\begin{smallpmatrix} 
0 & -1 & 0 & 0 \\ 
-1 & 0 & 0 & 0 \\ 
0 & -1 & 0 & -1 \\ 
0 & 0 & -1 & 0 
\end{smallpmatrix}
\end{pmatrix},
\end{align*}
where we assume that each $u\in r_E(s_E^{-1}(u_0))\setminus\{u_0\}$ corresponds to the last vertex in its component when writing the adjacency matrix in block form. 
Everywhere else, we let $U$ and $V$ be equal to the identity. 
An elementary computation shows that $(U,V)$ is an \SLPEe from $-\iota_{\mathbf{r}}(-\Bsf_E^\bullet)$ to $\Bsf_{E_{u_0,\#\#}}^\bullet$. 
\end{proof}

\begin{corollary} \label{cor:hashmovetwice}
Let $E$ be a graph with finitely many vertices, and let $u_0$ be a vertex in $E^0$ such that these satisfy Assumption~\ref{assumption:1-hash}, and assume, moreover, that every cyclic strongly connected component is a singleton (\ie, $\Bsf_E\in\MPZccc$). 
Then $E \Meq E_{u_0,\#\#}$ and $C^*(E)\otimes\K\cong C^*(E_{u_0,\#\#})\otimes\K$.  
\end{corollary}
\begin{proof}
This follows directly from Proposition~\ref{prop:hashmovetwice}, Theorem~\ref{thm:SLP-equivalence-implies-stable-isomorphism}, and \cite[Theorem~2.7]{MR3759003}. 
\end{proof}

\subsection{A classification result}
\label{subsec:classification-result}
By classification of simple, purely infinite graph \cas, \ie, by Kirchberg-Phillips classification, the graph \cas $C^*(\mathbf{E}_*)$ and $C^*(\mathbf{E}_{**})$ are isomorphic (this important case is actually due to R{\o}rdam, \cf\ \cite{MR1340839}). 
This is a key component in the proof of $C^*(E_{u,-})\cong C^*(E_{u,--})$ (\cf\ Theorem~\ref{thm:cuntz-splice-1}). Similarly to that proof, to show that $C^*(E_{u_0,\#})$ is isomorphic to $C^*(E_{u_0,\#\#})$, we would like to know that $C^*(\uline{\mathbf{E}}_{u_0,\#})$ and $C^*(\uuline{\mathbf{E}}_{u_0,\#\#})$ are still isomorphic if we do not enforce the summation relation at $\uline{v}_1$ and $\uuline{w}_1$, respectively. 

\begin{lemma}\label{lem: multiple strands}
Let $\A_{1}$ and $\A_{2}$ be unital, separable \cas.  For each $i = 1,2$, let $\{ \B_{k,i} \}_{ k = 1}^{n}$ be a collection of ideals of $\A_{i}$ such that 
\begin{enumerate}
\item each $\B_{k,i}$ is a separable, stable \ca and satisfies the corona factorization property,

\item $\B_{k,i} \cap \B_{l,i} = 0$ for all $k \neq l$,

\item $\A_{i} / \left( \bigoplus_{k=1}^{n} \B_{k,i} \right) \cong C(S^{1})$, and

\item the extension $\mathfrak{e}_{l, i}$ obtained by pushing forward the extension 
\[
\mathfrak{e}_{i} \colon 0 \to \bigoplus_{ k = 1}^{n} \B_{k,i} \to \A_{i} \to C(S^{1}) \to 0
\] via the surjective \starhomo $\pi_{l} \colon \bigoplus_{ k = 1}^{n} \B_{k,i} \rightarrow \B_{l,i}$ is a full extension.
\end{enumerate}
Suppose for each $k$, there exists a \stariso $\Psi_{k} \colon \B_{k,1} \rightarrow \B_{k,2}$ and there exists a \stariso $\Phi \colon C(S^{1}) \rightarrow C(S^{1})$ such that 
\[
K_{1-j} ( \Psi_{k} ) \circ K_{j} ( \tau_{\mathfrak{e}_{k,1}} ) =  K_{j} ( \tau_{\mathfrak{e}_{k,2}} ) \circ K_{j} ( \Phi ), 
\]
where we have identified $K_{j} ( \mathcal{Q} ( \B_{k,i} ))$ with $K_{1-j} ( \B_{k,i} )$, for $j=0,1$.  Then there exists a unitary $u \in \mathcal{M}(\B_{2})$ such that 
\[
\mathrm{Ad} ( \pi (u) ) \circ \overline{ \Psi } \circ \tau_{ \mathfrak{e}_{1} } = \tau_{ \mathfrak{e}_{2} } \circ \Phi
\]
where $\overline{\Psi}$ is the \stariso from $\mathcal{Q}( \B_{1} )$ to $\mathcal{Q}( \B_{2} )$ induced by the \stariso $\Psi = \bigoplus_{ k =1}^{n} \Psi_{k} \colon \B_1=\bigoplus_{ k  =1}^{n} \B_{k,1} \rightarrow \B_2=\bigoplus_{ k  =1}^{n}\B_{k,2}$ and $\pi$ is the canonical surjective \starhomo from $\mathcal{M}(\B_2)$ to $\mathcal{Q}(\B_2)$.
\end{lemma}

\begin{proof}
We claim that $\mathfrak{e}_{i}$ is an absorbing extension (note that the quotient is $C(S^1)$, and thus nuclear).  We first show that $\mathfrak{e}_{i}$ is a full extension.  Let $\Xi \colon \mathcal{Q}( \B_{i} ) \rightarrow \bigoplus_{ k = 1}^{n} \mathcal{Q} ( \B_{k,i} )$ be the canonical \stariso such that $\Xi \circ \tau_{ \mathfrak{e}_{i} } = \bigoplus_{ k = 1}^{ n } \tau_{ \mathfrak{e}_{k,i} }$.  Let $f \in C(S^{1})$ with $f \neq 0$.  Then $\tau_{ \mathfrak{e}_{k,i} } (f)$ is full in $\mathcal{Q} ( \B_{k,i} )$ by assumption.  Hence, $\Xi \circ \tau_{ \mathfrak{e}_{i} } (f)$ is full in $\bigoplus_{ k = 1}^{n} \mathcal{Q} ( \B_{k,i} )$ which implies that $\tau_{ \mathfrak{e}_{i} } (f)$ is full in $\mathcal{Q} ( \B_{i} )$.  Thus, showing that $\mathfrak{e}_{i}$ is a full extension.  Since $\B_{k,i}$ satisfies the corona factorization property, $\B_{i}$ satisfies the corona factorization property.  These two observations imply that $\mathfrak{e}_{i}$ is a purely large extension, and by Elliott-Kucerovsky \cite{MR1835515}, $\mathfrak{e}_{i}$ is an absorbing extension.

The assumption that 
\[
K_{1-j} ( \Psi_{k} ) \circ K_{j} ( \tau_{\mathfrak{e}_{k,1}} ) =  K_{j} ( \tau_{\mathfrak{e}_{k,2}} ) \circ K_{j} ( \Phi ), 
\]
for $j=0,1$, implies that for $j=0,1$
\begin{equation}\label{eq: index maps}
K_{1-j} ( \Psi ) \circ K_{j} ( \tau_{ \mathfrak{e}_{1} } ) = K_{j} ( \tau_{ \mathfrak{e}_{2} } ) \circ K_{j} ( \Phi ) . 
\end{equation}

Let $\mathfrak{f}_{1}$ be the extension obtained by pushing forward the extension $\mathfrak{e}_{1}$ via the \stariso $\Psi \colon \B_{1} \rightarrow \B_{2}$ and let $\mathfrak{f}_{2}$ be the extension obtained by pulling back the extension $\mathfrak{e}_2$ via the \stariso $\Phi \colon C(S^{1}) \rightarrow C(S^{1})$.  By Equation~\eqref{eq: index maps}, we have that $K_{0} ( \tau_{ \mathfrak{f}_{1} } ) = K_{0} ( \tau_{ \mathfrak{f}_{2} } )$ and $K_{1} ( \tau_{ \mathfrak{f}_{1} } ) = K_{1} ( \tau_{ \mathfrak{f}_{2} } )$.  Since the $K$-theory of $C(S^{1})$ is free, by the UCT, $[ \tau_{ \mathfrak{f}_{1} } ] = [ \tau_{ \mathfrak{f}_{2} } ]$ in $\kk^{1} ( C(S^{1} ), \B_{2} ) = \mathrm{Ext}^{1} ( C(S^{1}), \B_{2} )$.

Since $\tau_{\mathfrak{f}_{1}}$ and $\tau_{\mathfrak{f}_{2}}$ are unital extensions, by \cite[Page 112, 2nd paragraph]{MR1446202}, there are unital, trivial extensions $\sigma_{1} \colon C(S^{1} ) \rightarrow \mathcal{Q} ( \B_{2} )$ and $\sigma_{2} \colon C(S^{1} ) \rightarrow \mathcal{Q} ( \B_{2} )$ and a unitary $w$ in $\mathcal{M}( \mathfrak{B}_{2} )$ such that 
\[
\mathrm{Ad} ( \pi(w) ) \circ ( \tau_{ \mathfrak{f}_{1} } \oplus \sigma_{1} ) = \tau_{ \mathfrak{f}_{2}} \oplus \sigma_{2}.
\]

We claim that $\sigma_i$ is strongly unital, \ie, there exists a unital \starhomo $\widetilde{\sigma}_i \colon C( S^1 ) \rightarrow \mathcal{M} ( \B_2 )$ such that $\pi \circ \widetilde{\sigma}_i = \sigma_i$.  Since $\sigma_i$ is a trivial extension, there exists a \starhomo $\sigma_i' \colon C(S^1) \rightarrow \mathcal{M}( \B_2 )$ such that $\pi \circ \sigma_i' = \sigma_i$.  If $\sigma_i'$ is a unital \starhomo, then set $\widetilde{\sigma}_i = \sigma_i'$.  Suppose $\sigma_i'$ is not a unital \starhomo.  Since $\pi \circ \sigma_i' = \sigma_i$ and $\sigma_i$ is unital, we have that $p_i = 1_{ \mathcal{M}( \B_2 ) } - \sigma_i' ( 1_{ C(S^1) } ) \in \B_2$.  Let $\sigma_i'' \colon C(S^1) \rightarrow \B_2$ be the \starhomo defined by $\sigma_i'' ( f ) = f(1) p_i$.  Set $\widetilde{ \sigma}_i = \sigma_i' + \sigma_i''$.  Since $\sigma_i' (f) \sigma_i''(g) = \sigma_i'' (f) \sigma_i' (g) = 0$ for all $f,g \in C(S^1)$, we have that $\widetilde{ \sigma }_i$ is a unital \starhomo from $C(S^1)$ to $\mathcal{M}( \B_2 )$.  Since the image of $\sigma_i''$ is contained in $\B_2$, we have that $\pi \circ \widetilde{\sigma}_i = \sigma_i$.  This proves our claim.

Since $\mathfrak{e}_i$ is a unital, absorbing extension, $\mathfrak{f}_i$ is a unital, absorbing extension.  Hence, there exists a unitary $v_i$ in $\mathcal{M} ( \B_2 )$ such that 
\[
\mathrm{Ad} ( \pi( v_i ) ) \circ ( \tau_{\mathfrak{f}_{i}} \oplus \sigma_i ) = \tau_{\mathfrak{f}_{i}}
\]
since $\sigma_i$ is a strongly unital, trivial extension.  Set $u = v_2 w v_1^*$.  Then $u$ is a unitary in $\mathcal{M}( \B_2 )$ such that 
\begin{align*}
\left( \mathrm{Ad} ( \pi (u) ) \circ \tau_{ \mathfrak{f}_1 } \right)(f) &= \pi(v_2 w v_1^*) \tau_{\mathfrak{f}_{1}} (f) \pi ( v_1 w^* v_2^* ) \\
 												&= \pi ( v_2  w ) ( \tau_{\mathfrak{f}_{1}} \oplus \sigma_1)(f) \pi ( w^* v_2^* ) \\
												&= \pi ( v_2 ) ( \tau_{\mathfrak{f}_2} \oplus \sigma_2 )(f) \pi ( v_2^* ) \\
										&= \tau_{ \mathfrak{f}_2 } (f).
\end{align*} 
By \cite{MR1725815}, $\tau_{ \mathfrak{f}_{1} } = \overline{ \Psi } \circ \tau_{ \mathfrak{e}_{1} }$ and $\tau_{ \mathfrak{f}_{2} } = \tau_{ \mathfrak{e}_{2} } \circ \Phi$.  Thus, 
\[
\mathrm{Ad} ( \pi (u) ) \circ \overline{ \Psi } \circ \tau_{ \mathfrak{e}_{1} } = \mathrm{Ad} ( \pi (u) ) \circ \tau_{ \mathfrak{f}_{1} } =  \tau_{ \mathfrak{f}_{2} } =\tau_{ \mathfrak{e}_{2} } \circ \Phi. \qedhere 
\]
\end{proof}

\begin{lemma}\label{lem: K-theory ideal}
Consider the graphs 
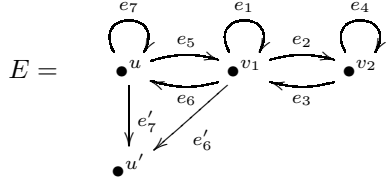
\[
E = \ \ \ \ \xymatrix{
  \bullet^{u} \ar@(ul,ur)[]^{e_{7}}  \ar@/^/[r]^{e_{5}} \ar[d]^{ e_{7}' }& \bullet^{v_1} \ar@(ul,ur)[]^{e_{1}} \ar@/^/[r]^{e_{2}} \ar@/^{}/[l]^{e_{6}} \ar[ld]^{ e_{6}'} & \bullet^{v_2} \ar@(ul,ur)[]^{e_{4}} \ar@/^/[l]^{e_{3}} \\
  \bullet^{u'}
}
\]
and 
\[
F \  =  \ \ \ \ 
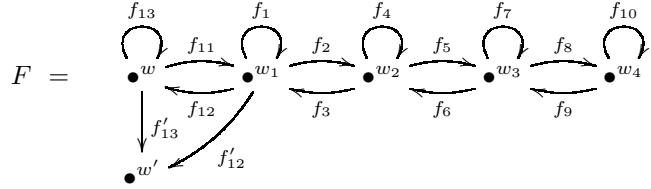
\xymatrix{
\bullet^{w} \ar@(ul,ur)[]^{f_{13}}  \ar@/^/[r]^{f_{11}} \ar[d]^{ f_{13}' }& 	\bullet^{ w_{1} } \ar@(ul,ur)[]^{f_{1}}  \ar@/^/[r]^{ f_{2} } \ar@/^{}/[l]^{f_{12}} \ar@/^{}/[ld]^{f_{12}'}  & \bullet^{ w_{2} } \ar@(ul,ur)[]^{f_{4}} \ar@/^/[r]^{ f_{5} }  \ar@/^/[l]^{f_{3}} &  \bullet^{w_3} 				\ar@(ul,ur)[]^{f_{7}} \ar@/^/[r]^{f_{8}} \ar@/^/[l]^{f_{6}}
	& \bullet^{w_4} \ar@(ul,ur)[]^{f_{10}} \ar@/^/[l]^{f_{9}} \\
	\bullet^{w'}
	}
\]
Then there exists a \stariso $\Phi \colon C^{*} ( E ) \otimes \K \rightarrow C^{*} (F) \otimes \K$ such that $\Phi ( q_{u} \otimes e_{11} ) \sim Q_{w} \otimes e_{11}$ and $\Phi ( q_{u'}  \otimes e_{11}) \sim Q_{w'} \otimes e_{11}$, where $\setof{ q_{u} , t_{e} }{ u \in E^{0} , e \in E^{1} }$ is a Cuntz-Krieger $E$-family generating $C^{*} (E)$, $\setof{ Q_{u} , T_{e} }{ u \in F^{0} , e \in F^{1} }$ is a Cuntz-Krieger $F$-family generating $C^{*} (F)$, and $\{ e_{ij} \}$ is a system of matrix units for $\K$.
\end{lemma}

\begin{proof}
For a graph $G=(G^0,G^1,r_G,s_G)$ and a subset $S \subsetneq G^0$, we let $G/S$ denote the subgraph with vertices $G^0\setminus S$, edges $G^1\setminus(r_G^{-1}(S)\cup s_G^{-1}(S))$ and range and source maps being the restrictions. 
A computation shows that the six-term exact sequence in $K$-theory induced by $\mathfrak{e} \colon 0 \to \mathfrak{J}_{\{ u' \} } \otimes \K \to C^{*} (E) \otimes \K \to C^{*} ( E / \{ u' \} ) \otimes \K \to 0$ is isomorphic to 
\[
\xymatrix{
\Z\langle [ q_{u'}]  \rangle \ar[r]^{0} & \Z \langle [ q_{u} ] \rangle \ar[r] & \Z \langle [\overline{q}_{u}] \rangle \ar[d]^{0}  \\
\Z\langle a\rangle \ar[u]^{1} & 0\ar[l] & 0 \ar[l]  
}
\]
where $\overline{q}_{u}$ is the image of $q_{u}$ in $C^{*} ( E / \{ u' \} )$ and $a$ is the generator that is sent to $[q_{u'}]$.  Another computation shows that the six-term exact sequence in $K$-theory induced by $\mathfrak{f} \colon 0 \to \mathfrak{J}_{\{ w' \} } \otimes \K \to C^{*} (F) \otimes \K \to C^{*} ( F / \{ w' \} ) \otimes \K \to 0$ is isomorphic to 
\[
\xymatrix{
\Z\langle [ Q_{w'}]  \rangle \ar[r]^{0} & \Z \langle [ Q_{w} ] \rangle \ar[r] & \Z \langle [\overline{Q}_{w}] \rangle \ar[d]^{0}  \\
\Z\langle b \rangle \ar[u]^{1} & 0\ar[l] & 0 \ar[l]  
}
\]
where $\overline{Q}_{w}$ is the image of $Q_{w}$ in $C^{*} ( F / \{ w' \} )$  and $b$ is the generator that is sent to $[Q_{w'}]$.  

Since $E / \{ u' \}$ and $F / \{w' \}$ are strongly connected graphs satisfying Condition~(K), the Kirchberg-Phillips classification result applies, 
so there exists a \stariso $\Psi_{2} \colon C^{*} ( E / \{u\} ) \otimes \K \to C^{*} ( F / \{ w' \} ) \otimes \K$ such that $K_{0} ( \Psi_2 ) ( [ \overline{ q }_{u}]  ) = [ \overline{Q}_{w} ]$ and $K_{1} ( \Psi_2 )$ is equal to the isomorphism that sends $a$ to $b$.  Moreover, since $\mathfrak{J}_{\{u'\}} \otimes \K \cong \K$ and $\mathfrak{J}_{\{w'\}} \otimes \K \cong \K$, we have that there exists a \stariso $\Psi_{0} \colon \mathfrak{J}_{\{u'\}} \otimes \K \rightarrow \mathfrak{J}_{\{ w' \} } \otimes \K$ such that $K_{0} ( \Psi_{0} ) ( [ q_{u'} ] ) = [ Q_{w'} ]$.  So, $K_{0} ( \overline{ \Psi }_{0} ) \circ K_{0} ( \tau_{ \mathfrak{e} } )= K_{0} ( \tau_{ \mathfrak{f}} ) \circ K_{0} ( \Psi_{2} )$ and $K_{1} ( \overline{ \Psi }_{0} ) \circ K_{1} ( \tau_{ \mathfrak{e} } )= K_{1} ( \tau_{ \mathfrak{f}} ) \circ K_{1} ( \Psi_{2} )$, where $\overline{\Psi}_0$ is the \stariso from $\mathcal{Q} ( \mathfrak{J}_{ \{ u' \} } \otimes \K)$ to $\mathcal{Q} ( \mathfrak{J}_{ \{ w' \} } \otimes \K )$ induced by the \stariso $\Psi_0$ .  Since the $K$-theory of $C^{*} ( E / \{ u' \} )$ is free, by the UCT, we have that $\kk ( \Psi_{2} ) \times [ \tau_{\mathfrak{f}} ] = [ \tau_{ \mathfrak{e} } ] \times \kk ( \Psi_{0} )$ in $\kk^{1} ( C^{*} (E / \{ u' \} ) , \mathfrak{J}_{\{w'\} } )$.  By \cite[Proposition~2]{MR3425906}, there exists a \stariso $\Phi \colon C^{*} (E) \otimes \K \rightarrow C^{*} (F) \otimes \K$ such that $\Phi_{0} = \Phi \vert_{\mathfrak{J}_{\{u'\}}} \colon \mathfrak{J}_{\{u'\}} \otimes \K \rightarrow \mathfrak{J}_{\{ w'\} } \otimes \K$ is a \stariso, the diagram
\[
\xymatrix{
0 \ar[r] & \mathfrak{J}_{\{u'\}} \otimes \K \ar[d]^{ \Phi_{0}} \ar[r] & C^{*} (E) \otimes \K \ar[d]^{ \Phi } \ar[r] & C^{*} ( E / \{ u' \} ) \otimes \K \ar[r] \ar[d]^{ \Psi_{2}} & 0 \\
0 \ar[r] & \mathfrak{J}_{ \{ w' \} } \otimes \K \ar[r] & C^{*} (F ) \otimes \K \ar[r] & C^{*} (F / \{w'\}) \otimes \K \ar[r] & 0.
}
\]
is commutative, and $\kk ( \Phi_{0} ) = \kk ( \Psi_{0} )$.  So, in particular, $K_{0} ( \Phi_0 ) ( [ q_{u'} ] ) = [ Q_{w'} ]$ in $K_{0} ( \mathfrak{J}_{\{w'\}} )$.  Since $\Phi(q_{u'} \otimes e_{11})$ and $Q_{w'} \otimes e_{11}$ generate the ideal $\mathfrak{J}_{\{w'\}} \otimes \K$ and every graph \ca has the stable weak cancellation property, we have that $\Phi(q_{u'} \otimes e_{11}) \sim Q_{w'} \otimes e_{11}$.

By the above six-term exact sequences in $K$-theory, $K_{0} ( \Phi )$ is completely determined by $K_{0} ( \Psi_{2} )$.  Since $K_{0} ( \Psi_{2} ) ( [ \overline{q}_{u} ] ) = [ \overline{Q}_{w} ]$, we have that $K_{0} ( \Phi ) ( [ q_{u} ] ) = [ Q_{w} ]$ in $K_{0} ( C^{*} (F) )$.  Since $\Phi ( q_{u} \otimes e_{11} )$ and $Q_{w}\otimes e_{11}$ are full projections in $C^{*} (F)$ and every graph \ca has the stable weak cancellation property, we have that $\Phi(q_{u} \otimes e_{11}) \sim Q_{w} \otimes e_{11}$.
\end{proof}

\begin{lemma}\label{lem: full extensions}
Let $E$ be a finite graph with exactly two nontrivial hereditary saturated subsets $H_{1}$ and $H_{2}$ such that $H_{1} = \{ w \}$, $w \in H_{2}$, $( H_{2} \setminus H_{1} , s^{-1} ( H_{2} ) \setminus r^{-1}(H_{1} ) , r, s ) $ is a strongly connected graph satisfying Condition~(K), $E^{0} \setminus ( H_{1} \cup H_{2} ) = \{ v \}$, $v$ supports exactly one loop and $w$ does not support any cycle.  Then $\mathfrak{J}_{H_{2}}$ is a stable \ca satisfying the corona factorization property, $C^{*} (E) / \mathfrak{J}_{ H_{2}} = C(S^{1})$, and the extension 
\[
\mathfrak{e} \colon 0 \to \mathfrak{J}_{ H_{2} } \to C^{*} (E) \to C^{*} (E) / \mathfrak{J}_{H_{2}} \to 0
\]
is a full extension.
\end{lemma}
 
\begin{proof}
By \cite{MR3254772}, $\mathfrak{J}_{H_{2}} \cong C^{*} (F)$ where $F$ is the graph obtained from the subgraph $( H_{2} , s^{-1} ( H_{2} ) , s, r )$ by adding infinitely many regular sources to each vertex $u \in H_{2}$ with $| s^{-1} ( v ) \cap r^{-1} (u) | \geq 1$.  It is clear that $F$ has no nonzero bounded graph trace and every vertex that is on a cycle is left infinite (as defined in \cite[Definitions~2.2 and~3.1]{MR2051143}).  By \cite[Theorem~3.2]{MR2051143}, $C^{*} ( F)$ is a stable \ca and hence $\mathfrak{J}_{H_{2}}$ is a stable \ca.  That $\mathfrak{J}_{H_{2}}$ satisfies the corona factorization property follows as in \cite[Proposition 6.1]{MR3056712}.  The fact that $C^{*} (E) / \mathfrak{J}_{ H_{2}} = C( S^{1} )$ is clear.    

We now show that  $\mathfrak{e}$ is a full extension.  Note that $\mathfrak{J}_{H_{1}} \cong \K$ and $\mathfrak{J}_{H_{2}} / \mathfrak{J}_{H_{1}}$ is a purely infinite simple \ca (since $H_1$ only consists of a sink, $E_{ H_{2} \setminus H_{1} } = ( H_{2} \setminus H_{1} , s^{-1} ( H_{2} ) \setminus r^{-1}(H_{1} ) , r, s )$ is a strongly connected graph satisfying Condition~(K) and $\mathfrak{J}_{H_{2}} / \mathfrak{J}_{H_{1}}$ is stably isomorphic to $C^{*} ( E_{H_{2} \setminus H_{1} } )$).  By \cite[Lemma~6.6]{MR3759003}, $\mathcal{Q} ( \mathfrak{J}_{H_{2}} )$ has exactly one nontrivial ideal and it is the kernel of $\overline{\pi} \colon \mathcal{Q} ( \mathfrak{J}_{H_{2}} ) \rightarrow \mathcal{Q} ( \mathfrak{J}_{H_{2}} / \mathfrak{J}_{H_{1}} )$, where $\overline{\pi}$ is the surjective \starhomo induced by $\pi \colon \mathfrak{J}_{H_{2}} \rightarrow \mathfrak{J}_{H_{2}} / \mathfrak{J}_{H_{1}}$.  Therefore, to show that $\mathfrak{e}$ is a full extension, it is enough to show that $\overline{\pi} ( \tau_{ \mathfrak{e} } (f) ) \neq 0$ for all $f \neq 0$.  

Consider the extension $\overline{\mathfrak{e}} \colon 0 \to \mathfrak{J}_{H_{2}} / \mathfrak{J}_{H_{1}} \to C^{*} ( E ) / \mathfrak{J}_{ H_{1}} \to C(S^{1} ) \to 0$.  So, $\overline{\mathfrak{e}}$ is the extension obtained by pushing forward the extension $\mathfrak{e}$ via the surjective \starhomo $\pi$.  Note that $C^{*} ( E ) / \mathfrak{J}_{H_{1}} \cong C^{*} ( E' )$ where $E'$ is the graph $( E^{0} \setminus H_{1} , r^{-1} ( E^{0} \setminus H_{1} ), s, r )$.  Note that $E'$ satisfies Condition~(L) (\ie, every cycle has an exit).  If $\mathfrak{I}$ is a nonzero ideal of $C^{*} (E')$ then there exists $u \in (E')^{0}$ such that $p_{u} \in \mathfrak{I}$.  Otherwise by Szyma\'nski's general Cuntz-Krieger uniqueness theorem (\cite[Theorem~1.2]{MR1914564}), the \starhomo $C^{*} (E') \rightarrow C^{*} (E') / \mathfrak{I}$ would be injective but this can not happen if $\mathfrak{I}$ is a nonzero ideal.  So, the unique nontrivial gauge invariant ideal of $C^{*} (E')$ corresponding to the ideal $\mathfrak{J}_{H_{2}} / \mathfrak{J}_{H_{1}}$ is an essential ideal of $C^{*} (E')$.  So, $\mathfrak{J}_{H_{2}} / \mathfrak{J}_{H_{1}}$ is an essential ideal of $C^{*} ( E ) / \mathfrak{J}_{ H_{1}}$.  This implies that $\overline{\mathfrak{e}}$ is an essential extension.  So, $\tau_{ \overline{\mathfrak{e}} } (f) \neq 0$ for all $f \neq 0$.  Since $\overline{ \pi } ( \tau_{ \mathfrak{e}} (f) ) =\tau_{ \overline{\mathfrak{e} } } (f)$ by \cite{MR1725815}, we have that $\overline{\pi} ( \tau_{ \mathfrak{e} } (f) ) \neq 0$ for all $f \neq 0$.  So, by the above paragraph, $\mathfrak{e}$ is a full extension.
\end{proof}

\begin{proposition} \label{prop:csdouble-hash}
Let $E=(E^0,E^1,r_E,s_E)$ be a graph with finitely many vertices, and let $u_0$ be a vertex in $E^0$ such that these satisfy Assumption~\ref{assumption:1-hash}. Let, moreover, 
$$\mathcal{E} = C^*\left(\uline{\mathbf{E}}_{u_0,\#}, (\uline{\mathbf{E}}_{u_0,\#})^0\setminus\left(\bigcup_{u\in r_E(s_E^{-1}(u_0))\setminus\{u_0\}}\uline{u}\right)\right)$$ 
and 
$$\mathcal{E}'=C^*\left(\uuline{\mathbf{E}}_{u_0,\#}, (\uuline{\mathbf{E}}_{u_0,\#})^0\setminus\left(\bigcup_{u\in r_E(s_E^{-1}(u_0))\setminus\{u_0\}}\uuline{u}\right)\right)$$ 
be the relative graph \cas (in the sense of Muhly-Tomforde \cite{MR2054981}). 
Let $p_{\uline{u}_0}$, $p_{\uline{u}}$, $p_{\uline{v}_i^u}$, $i=1,2$, $s_{\uline{e}_0}$,$s_{\uline{c}_j^u}$, $j=0,1$, $s_{\uline{e}_k^u}$, $k=1,2,\ldots,8$ for $u\in r_E(s_E^{-1}(u_0))\setminus\{u_0\}$ denote the canonical generators of $\mathcal{E}$ and let $p_{\uuline{u}_0}$, $p_{\uuline{u}}$, $p_{\uuline{w}_i^u}$, $i=1,2,3,4$, $\uuline{e}_0$,$s_{\uuline{c}_j^u}$, $j=0,1$, $s_{\uuline{f}_k^u}$, $k=1,2,\ldots,16$ for $u\in r_E(s_E^{-1}(u_0))\setminus\{u_0\}$ denote the canonical generators of $\mathcal{E}'$. 

Then there exists a \stariso $\Psi$ from $\mathcal{E}'$ to $\mathcal{E}$ such that 
\begin{align*}
p_{\uline{u}_0} &\sim \Psi\left( p_{\uuline{u}_0}\right),
\intertext{in $\mathcal{E}$, and for each $u\in r_E(s_E^{-1}(u_0))\setminus\{u_0\}$,}
p_{\uline{u}} - \left(s_{\uline{e}_5^u}s_{\uline{e}_5^u}^*  + s_{\uline{c}_1^u} s^*_{\uline{c}_1^u}\right) 
&\sim \Psi\left( p_{\uuline{u}} - \left(s_{\uuline{f}_{11}^u} s_{\uuline{f}_{11}^u}^* + s_{\uuline{c}_1^u} s^*_{\uuline{c}_1^u}\right)\right), \text{ and }\\
p_{\uline{u}} &\sim \Psi\left( p_{\uuline{u}}\right), 
\end{align*}
in $\mathcal{E}$, where $\sim$ denotes Murray-von~Neumann equivalence. Moreover, the induced map between $K_1$ of the quotients $C(S^1)$ (corresponding to the cyclic component $\{u_0\}$) is $-\id$. 
\end{proposition}

\begin{proof}
Let $G$ be the graph obtained from $\uline{\mathbf{E}}_{u_0,\#}$ by adding a sink $\uline{u}'$ for each $u \in r_E(s_E^{-1}(u_0))\setminus\{u_0\}$ and adding one edge from $\uline{u}_{0}$ to $\uline{u}'$, one edge from $\uline{u}$ to $\uline{u}'$, and one edge from $\uline{v}_{1}^{u}$ to $\uline{u}'$.  Let $F$ be the graph obtained from $\uuline{\mathbf{E}}_{u_0,\#\#}$ by adding a sink $\uuline{u}'$ for each $u \in r_E(s_E^{-1}(u_0))\setminus\{u_0\}$ and adding one edge from $\uuline{u}_{0}$ to $\uuline{u}'$, one edge from $\uuline{u}$ to $\uuline{u}'$, and one edge from $\uuline{v}_{1}^{u}$ to $\uuline{u}'$. 

By \cite[Theorem~3.7 and its proof]{MR2054981}, there exists a \stariso $\Lambda \colon C^{*} (G) \rightarrow \mathcal{E}$ such that $\Lambda ( q_{\uline{u}_{0} } ) = p_{\uline{u}_{0} }$, $\Lambda( q_{\uline{u} } + q_{ \uline{u}' } ) = p_{ \uline{u} }$ and $\Lambda ( q_{ \uline{ u }' } ) = p_{ \uline{u} } - ( s_{ \uline{c}_{1}^{u } } s_{ \uline{c}_{1}^{u } }^{*} + s_{ \uline{e}_{5}^{u } }  s_{ \uline{e}_{5}^{u } }^{*} )$.  Similarly, there exists a \stariso $\Gamma\colon C^{*} (F) \rightarrow \mathcal{E}'$ such that
 $\Gamma ( q_{\uuline{u}_{0} } ) = p_{\uuline{u}_{0} }$, $\Gamma( q_{\uuline{u} } + q_{ \uuline{u}' } ) = p_{ \uuline{u} }$ and $\Gamma ( q_{ \uuline{ u }' } ) = p_{ \uuline{u} } - ( s_{ \uuline{c}_{1}^{u } } s_{ \uuline{c}_{1}^{u } }^{*} + s_{ \uuline{f}_{11}^{u } }  s_{ \uuline{f}_{11}^{u } }^{*} )$.   With the canonical identifications, it is also clear that the induced maps on $K_1$ of the quotient $C(S^1)$ is the identity. Therefore, it is enough to construct a \stariso $\Phi \colon C^{*} (G) \rightarrow C^{*} (F)$ such that $\Phi ( q_{\uline{u}_{0}} ) \sim q_{ \uuline{u}_{0}}$, $\Phi( q_{\uline{u}} ) \sim q_{ \uuline{u}}$, and $\Phi ( q_{ \uline{u}' } ) \sim q_{ \uuline{u}' }$ for all $u \in r_E(s_E^{-1}(u_0))\setminus\{u_0\}$. 
 
For each $u \in r_E(s_E^{-1}(u_0))\setminus\{u_0\}$, set $\uline{H}_{u} = \{ \uline{u}, \uline{u}' , \uline{v}_{1}^{u} , \uline{ v}_{2}^{u} \}$.  Then $\uline{H}_{u}$ is a hereditary and saturated subset of $G^{0}$ and $\uline{H}_{u} \cap \uline{H}_{ w } = \emptyset$ if and only if $u \neq w$.  Note that $\mathfrak{J}_{ \uline{H}_u }$ canonically sits as an ideal of $C^* ( G_u )$, where $G_u$ is the subgraph $( \uline{H}_u \sqcup \{ \uline{u}_0 \} , s_G^{-1} ( \uline{H}_u \sqcup \{ \uline{u}_0 \} ) \cap r_G^{-1} ( \uline{H}_u \sqcup \{ \uline{u}_0 \} ), r, s)$.  By Lemma~\ref{lem: full extensions}, $\mathfrak{J}_{\uline{H}_u}$ is stable and the extension $0 \to \mathfrak{J}_{\uline{H}_u} \to C^* ( G_u ) \to C(S^1) \to 0$ is a full extension.

For each $u \in r_E(s_E^{-1}(u_0))\setminus\{u_0\}$, set $\uuline{H}_{u} = \{ \uuline{u}, \uuline{u}' , \uuline{w}_{1}^{u} , \uuline{w}_{2}^{u}, \uuline{w}_{3}^{u}, \uuline{w}_{4}^{u} \}$.  Then $\uuline{H}_{u}$ is a hereditary and saturated subset of $F^{0}$ and $\uuline{H}_{u} \cap \uuline{H}_{ w } = \emptyset$ if and only if $u \neq w$.  Note that $\mathfrak{J}_{ \uuline{H}_u }$ canonically sits as an ideal of $C^* ( F_u )$, where $F_u$ is the subgraph $( \uuline{H}_u \sqcup \{ \uuline{u}_0 \} , s_F^{-1} ( \uuline{H}_u \sqcup \{ \uuline{u}_0 \} ) \cap r_F^{-1} ( \uuline{H}_u \sqcup \{ \uuline{u}_0 \} ), r, s)$.  By Lemma~\ref{lem: full extensions}, $\mathfrak{J}_{\uuline{H}_{u}}$ is stable and the extension $0 \to \mathfrak{J}_{\uuline{H}_u} \to C^* ( F_u ) \to C(S^1) \to 0$ is a full extension.

Set $\uline{H} = \bigcup_{ u \in r_E(s_E^{-1}(u_0))\setminus\{u_0\} } \uline{H}_u$ and $\uuline{H} = \bigcup_{ u \in r_E(s_E^{-1}(u_0))\setminus\{u_0\} }\uuline{H}_u$.  Note that $C^{*} (G) / \mathfrak{J}_{\uline{H}} = C(S^{1})$ and $C^{*} (F) / \mathfrak{J}_{\uuline{H}} = C(S^{1})$.  Let $\uline{\mathfrak{e}}$ be the extension 
\[
0 \to \mathfrak{J}_{ \uline{H}} \to C^{*} (G) \to C(S^{1} ) \to 0
\] 
and let $\uuline{\mathfrak{e}}$ be the extension 
\[
0 \to \mathfrak{J}_{ \uuline{H} }  \to C^{*} (F) \to C(S^{1} ) \to 0.
\] 
By pushing forward the extension $\uline{\mathfrak{e}}$ via the surjective \starhomo $\uline{\pi}_{u}$ from $\bigoplus_{ v \in r_E(s_E^{-1}(u_0))\setminus\{u_0\}  } \mathfrak{J}_{ \uline{H}_{v} }$ to $\mathfrak{J}_{ \uline{H}_{u} }$, gives the extension
\[
\uline{\mathfrak{e}}_{u} \colon 0 \to \mathfrak{J}_{\uline{H}_u} \to C^* ( G_u ) \to C(S^1) \to 0.
\]
By pushing forward the extension $\uuline{\mathfrak{e}}$ via the surjective \starhomo $\uuline{\pi}_{u}$ from $\bigoplus_{ v \in r_E(s_E^{-1}(u_0))\setminus\{u_0\}  } \mathfrak{J}_{ \uuline{H}_{v} }$ to $\mathfrak{J}_{ \uuline{H}_{u} }$, gives the extension
\[
\uuline{\mathfrak{e}}_{u} \colon 0 \to \mathfrak{J}_{\uuline{H}_u} \to C^* ( F_u ) \to C(S^1) \to 0.
\]

A computation shows that the six-term exact sequence in $K$-theory induced by $\uline{\mathfrak{e}}_u$ is 
\[
\xymatrix{
\Z\langle [ q_{\uline{u}}]  \rangle \ar[r]^{0} & \Z\langle a_1\rangle \ar[r]^\cong & \Z\langle a_2\rangle \ar[d]^{0}  \\
\Z\langle a_3\rangle \ar[u]^{1} & 0\ar[l] & 0 \ar[l]  
}
\]
where $a_1,a_2,a_3$ are suitably chosen generators.  Another computation shows that the six-term exact sequence in $K$-theory induced by $\uuline{\mathfrak{e}}_u$ is  
\[
\xymatrix{
\Z\langle [ q_{\uuline{u}}]  \rangle \ar[r]^{0} & \Z \langle b_1 \rangle \ar[r]^\cong & \Z \langle b_2 \rangle \ar[d]^{0}  \\
\Z\langle b_3 \rangle \ar[u]^{1} & 0\ar[l] & 0 \ar[l]  
}
\]
where $b_1,b_2,b_3$ are suitably chosen generators. 

For every $u \in r_E(s_E^{-1}(u_0))\setminus\{u_0\}$, by Lemma~\ref{lem: K-theory ideal}, there exists a \stariso $\Gamma_{u} \colon \mathfrak{J}_{\uline{H}_{u}} \rightarrow \mathfrak{J}_{\uuline{H}_{u}}$ such that $\Gamma_{u} ( q_{\uline{u} } ) \sim q_{ \uuline{u} }$ and $\Gamma_{u} ( q_{\uline{u}'} ) \sim q_{ \uuline{u}' }$ for all $u \in r_E(s_E^{-1}(u_0))\setminus\{u_0\}$.  For every $u \in r_E(s_E^{-1}(u_0))\setminus\{u_0\}$, choose a \stariso $\Lambda \colon C(S^1) \rightarrow C(S^1)$ such that 
\[
K_0 ( \Gamma_u ) \circ K_1 ( \tau_{ \uline{\mathfrak{e}}_u } )  = K_1 ( \tau_{ \uuline{\mathfrak{e}}_u }) \circ K_1 ( \Lambda ),
\]
(\ie, $K_1(\Lambda)$ sends $a_3$ to $b_3$), and note that
\[
K_1 ( \Gamma_u ) \circ K_0 ( \tau_{ \uline{\mathfrak{e}}_u } )  = K_0 ( \tau_{ \uuline{\mathfrak{e}}_u } ) \circ K_0 ( \Lambda )
\]
trivially holds (since both sides are zero). Also note that by symmetry, we can use the same $\Lambda$ for all $u \in r_E(s_E^{-1}(u_0))\setminus\{u_0\}$. In fact, using the canonical identifications and the description of the cyclic six-term exact sequences in $K$-theory for graph \cas (\cf\ \cite[Remark~4.2]{MR2922394}), we easily see that we have that $K_1(\Lambda)=-\id$ (while, trivially, $K_0(\Lambda)=\id$). 

By Lemma~\ref{lem: multiple strands}, there exists a unitary $W \in \mathcal{M}( \mathfrak{J}_{\uuline{H} } )$ such that 
\[
\mathrm{Ad} (\pi(W) ) \circ \overline{\Gamma} \circ \tau_{ \uline{\mathfrak{e}} } = \tau_{ \uuline{\mathfrak{e}} } \circ \Lambda,
\]
where $\Gamma = \bigoplus_{u \in r_E(s_E^{-1}(u_0))\setminus\{u_0\}} \Gamma_u$ and $\overline{\Gamma}$ is the \stariso from $\mathcal{Q} ( \mathfrak{J}_{\uline{H} }  )$ to $\mathcal{Q} ( \mathfrak{J}_{\uuline{H}} )$ induced by $\Gamma$.  By \cite[Theorem~2.2]{MR1725815}, there exists a \stariso $\Phi \colon C^* (G) \rightarrow C^* (F)$ such that $\Phi \vert_ { \mathfrak{J}_{\uline{H}} }= \mathrm{Ad} ( W ) \circ \Gamma \vert_ { \mathfrak{J}_{\uline{H}} }$ and the induced map between the quotients is $\Lambda$.  Note that 
\[
\Phi ( q_{ \uline{u} } ) = \mathrm{Ad} ( W ) \circ \Gamma ( q_ {\uline{u} }  ) = \mathrm{Ad} (W) \circ \Gamma_u ( q_{\uline{u} } ) \sim \mathrm{Ad}(W) ( q_{ \uuline{u} } )  \sim q_{\uuline{u} }
\]
and
\[
\Phi ( q_{ \uline{u'} } ) = \mathrm{Ad} ( W ) \circ \Gamma ( q_ {\uline{u'} }  ) = \mathrm{Ad} (W) \circ \Gamma_u ( q_{\uline{u'} } ) \sim \mathrm{Ad}(W) ( q_{ \uuline{u'} } ) \sim q_{\uuline{u'} }
\]
in $C^*(F)$ for $u \in r_E(s_E^{-1}(u_0))\setminus\{u_0\}$.

Using $K$-theory computations in $K_0 ( \mathfrak{J}_{\uline{H}_u } )$, it is easy to see that $[ q_{ \uline{v}_1^u} ] = 0$, $[ q_{\uline{u}'} ] = 0$, and $[ q_{ \uline{v}_2^u} ] + [ q_{ \uline{u} } ] = 0$ in $K_0(C^*(G))$ (since $[ q_{\uline{v}_1^u} ] = [ q_{\uline{v}_1^u} ] + [ q_{\uline{u}'} ] + [ q_{ \uline{v}_2}^u ] + [ q_{\uline{u} } ]$ in $K_0 ( \mathfrak{J}_{\uline{H}_u } )$) for $u \in r_E(s_E^{-1}(u_0))\setminus\{u_0\}$.  Therefore,
\[
[ 1_{ C^* (G) } ] = [ q_{\uline{u}_0 } ] + \sum_{u \in r_E(s_E^{-1}(u_0))\setminus\{u_0\}} \left( [ q_{ \uline{v}_1^u} ] + [ q_{ \uline{v}_2^u} ] + [ q_{ \uline{u} } ]  + [ q_{\uline{u}'} ]  \right) = [ q_{\uline{u}_0} ]
\]
in $K_0(C^*(G))$. A similar computation shows that 
\[
[ 1_{ C^* (F) } ] =  [ q_{\uuline{u}_0 } ] + \sum_{u \in r_E(s_E^{-1}(u_0))\setminus\{u_0\}} [ q_{\uuline{u} } ]
\]
in $K_0 ( C^* (F) )$.  

Although, $[ q_{\uuline{u} } ]$ is not necessarily $0$ in $K_0 ( C^*(F) )$, $\sum_{u \in r_E(s_E^{-1}(u_0))\setminus\{u_0\}} [ q_{\uuline{u} } ]$ is in the image of the map from $K_1 ( C(S^1) )$ to $K_0 ( \mathfrak{J}_{ \uuline{H} } )$.  Hence, $\sum_{u \in r_E(s_E^{-1}(u_0))\setminus\{u_0\}} [ q_{\uuline{u} } ] = 0$ in $K_0 ( C^* (F) )$.  So, $[ 1_{ C^* (F) } ] =  [ q_{\uuline{u}_0 } ]$ in $K_0(C^*(F))$.  Since $\Phi$ is a \stariso, we have that $\Phi ( 1_{C^* (G) } ) = 1_{ C^*(F) }$ which implies that $[ \Phi ( q_{\uline{u}_0} ) ] = [ q_{\uuline{u}_{0}} ]$ in $K_0 ( C^* ( F) )$.  Since $q_{\uline{u}_0 }$ is full in $C^* (G)$ and $\Phi$ is a \stariso, $\Phi ( q_{\uline{u}_0} )$ is full in $C^* (F)$.  Note that $q_{ \uuline{u}_0}$ is full in $C^* (F)$.  Since every graph \ca has the stable weak cancellation property, we have that $\Phi ( q_{\uline{u}_0} ) \sim q_{\uuline{u}_0}$.
\end{proof}

\subsection{Invariance of the specialized Pulelehua move}

In this subsection, we show that the specialized Move \PP\ induces a stable isomorphism between the corresponding graph \cas{} --- and we keep to a certain extent track of what happens to the induced maps on $K$-theory. 
The idea is somewhat inspired by the proofs in \cite{MR3713535} respectively \cite{MR3624419}, but are more complicated. 

\label{subsec:inv-spec-pulelehua}
\begin{theorem}\label{thm:hash-1}
Let $E$ be a graph with finitely many vertices, and let $u_0$ be a vertex of $E$ such that these satisfy Assumption~\ref{assumption:1-hash}. 
Let $\Bsf_E\in\MPZcc$ for some $\mathbf{m},\mathbf{n}\in\N_0^N$ and some $\calP$ satisfying Assumption~\ref{ass:preorder} (\cf\ Assumption~\ref{assumption:1-hash}).
Assume, moreover, that $u_0$ belongs to the block $j\in\calP$. 

Let $\mathbf{r}=(r_i)_{i=1}^N$ and $\mathbf{r}'=(r_i')_{i=1}^N$, where $r_i=2$ and $r_i'=4$ for all immediate successors $i$ of $j$, and $r_i=0=r_i'$ otherwise. 
Then 
$\Bsf_{E_{u_0,\#}}\in\MPZcc[(\mathbf{m}+\mathbf{r})\times(\mathbf{n}+\mathbf{r})]$ 
and $\Bsf_{E_{u_0,\#\#}}\in\MPZcc[(\mathbf{m}+\mathbf{r}')\times(\mathbf{n}+\mathbf{r}')]$.

Then there exists a $\calP$-equivariant isomorphism $\Phi$ from $C^{*}(E_{u_0,\#})$ to $C^{*}(E_{u_0,\#\#})$. 

If all cyclic components are singletons and $\gcd(\Bsf_E\{i\})=1$ for all $i\in\calP$ corresponding to non-cyclic strongly connected components, then we can choose $\Phi$ such that 
there exists a \GLPEe $(U,V)$ from $-\iota_{\mathbf{r}}(-\Bsf_{E_{u_0,\#}})$ to $\Bsf_{E_{u_0,\#\#}}$ such that $U\{i\}$ and $V\{i\}$ are the $1\times 1$ matrix $1$ for all cyclic components $i\neq j$, $U\{j\}=-1$ and $V\{j\}=1$ and $\FKR(\calP;\Phi)=\FKR(U,V)$. 
\end{theorem}

\begin{proof}
We let $r$ and $s$ denote the range and source maps of $E$.
As above, we let $$\mathcal{E} = C^*\left(\uline{\mathbf{E}}_{u_0,\#}, (\uline{\mathbf{E}}_{u_0,\#})^0\setminus\left(\{\uline{u}_0\}\cup\bigcup_{u\in r(s^{-1}(u_0))\setminus\{u_0\}}\uline{u}\right)\right)$$ 
and choose an isomorphism between $\mathcal{E}$ and 
$$\mathcal{E}'=C^*\left(\uuline{\mathbf{E}}_{u_0,\#\#}, (\uuline{\mathbf{E}}_{u_0,\#\#})^0\setminus\left(\{\uuline{u}_0\}\cup\bigcup_{u\in r(s^{-1}(u_0))\setminus\{u_0\}}\uuline{u}\right)\right),$$ according to Proposition~\ref{prop:csdouble-hash}. 

Since $C^*(E_{u_0,\#})$ and $\mathcal{E}$ are unital, separable, nuclear \cas, it follows from Kirchberg's embedding theorem that there exists a unital embedding 
\[
	C^*(E_{u_0,\#}) \oplus \mathcal{E} \hookrightarrow \mathcal{O}_2. 	
\]
We will suppress this embedding in our notation. 
In $\mathcal{O}_2$, we denote the vertex projections and the partial isometries coming from $C^*(E_{u_0,\#})$ by $p_v, v\in E_{u_0,\#}^0$ and $s_e,e\in E_{u_0,\#}^1$, respectively, and we denote the vertex projections and the partial isometries coming from $\mathcal{E}$ by $p_{\uline{u}_0}$, $p_{\uline{u}}$, $p_{\uline{v}_1^u}$, $p_{\uline{v}_2^u}$ and $s_{\uline{e}_0}$, $s_{\uline{c}_0^u}$, $s_{\uline{c}_1^u}$, $s_{\uline{e}_i^u}$, $i=1,2,\ldots,8$, for $u\in r(s^{-1}(u_0))\setminus\{u_0\}$, respectively.
Since we are dealing with an embedding, it follows from 
\cite[Theorem~1.2]{MR1914564} that for any vertex-simple cycle $\alpha_1 \alpha_2 \cdots \alpha_n$ in $E_{u_0,\#}$ without any exit, we have that the spectrum of $s_{\alpha_1} s_{\alpha_2} \cdots s_{\alpha_n}$ contains the entire unit circle.

We will define a new Cuntz-Krieger $E_{u_0,\#}$-family. 
We let
\begin{align*}
q_{u_0}&=p_{\uline{u}_0}, \\ 
q_v &= p_v, &&\text{for each }v \in E^0\setminus r(s^{-1}(u_0)), 
\intertext{and}
q_{u} &= p_{\uline{u}}, \\
q_{v_i^u} &= p_{\uline{v}_i^u}, &&\text{for }i=1,2,
\end{align*}
for all $u\in r(s^{-1}(u_0))\setminus\{u_0\}$. 
Since any two nonzero projections in $\mathcal{O}_2$ are Murray-von~Neumann equivalent, we can choose partial isometries $x_1^e \in \mathcal{O}_2$ such that 
\begin{align*}
	x_1^e (x_1^e)^* &= s_{e} s_{e}^*				
	& (x_1^e)^* x_1^e &= 
	\begin{cases}
	p_{\uline{u}_0} & \text{if } r(e)=u_0, \\
	p_{\uline{r(e)}} & \text{otherwise,}
	\end{cases}	
\intertext{for all $e\in r^{-1}(r(s^{-1}(u_0)))\setminus s^{-1}(r(s^{-1}(u_0)))$ and partial isometries $x_2^{c_2^u} \in \mathcal{O}_2$}	
	x_2^{c_2^u} (x_2^{c_2^u})^* &= p_{\uline{u}} - (s_{\uline{e}_5^{u}} s_{\uline{e}_5^{u}}^* + s_{\uline{c}_1^{u}} s_{\uline{c}_1^{u}}^*)	
	& (x_2^{c_2^u})^* x_2^{c_2^u} &= p_{r({c_2^u})},
\end{align*} 
for all $u\in r(s^{-1}(u_0))\setminus\{u_0\}$.

We let 
\begin{align*}
t_{e_0}&=s_{\uline{e}_0}, \\ 
t_e &= s_e, &&\text{for all edges }e\text{ between vertices in }E^0\setminus r(s^{-1}(u_0)) \\
t_{e} &= x_1^e, &&\text{for all }e\in r^{-1}(r(s^{-1}(u_0)))\setminus s^{-1}(r(s^{-1}(u_0))), 
\intertext{and}
t_{c_i^u} &= s_{\uline{c}_i^u}, &&i=0,1, \\
t_{c_2^u} &= x_2^{c_2^u}, \\
t_{e_i^u} &= s_{\uline{e}_i^u}, &&i=1,2,\ldots,8,
\end{align*}
for all $u\in r(s^{-1}(u_0))\setminus\{u_0\}$. 

By construction $\setof{ q_v }{ v \in E_{u_0,\#}^0 }$ is a set of orthogonal projections, and the collection $\setof{ t_e }{ e \in E_{u_0,\#}^1 }$ is a set of partial isometries.
Furthermore, by choice of $t_e$ for all edges not touching $r(s^{-1}(u_0))$, the relations are satisfied at all vertices not connected to a vertex in $r(s^{-1}(u_0))$ by an edge. 
The choice of $x_1^e, x_2^{c_2^u}$ ensures that the relations hold at the other vertices as well. 
Hence $\{q_v, t_e\}$ does indeed form a Cuntz-Krieger $E_{u_0,\#}$-family. 
Denote this family by $\mathcal{S}$.

Using the universal property of graph \cas, we get a \starhomo from $C^*(E_{u_0,\#})$ onto $C^*(\mathcal{S}) \subseteq \mathcal{O}_2$.
Let $\alpha_1 \alpha_2 \cdots \alpha_n$ be a vertex-simple cycle in $E_{u_0,\#}$ without any exit. 
It is evident that no vertex-simple cycle without any exit uses edges connected to $u_0$, $u$, $v_1^u$ or $v_2^u$ for $u\in r(s^{-1}(u_0))\setminus\{u_0\}$. 
Hence $t_{\alpha_1} t_{\alpha_2} \cdots t_{\alpha_n} = s_{\alpha_1} s_{\alpha_2} \cdots s_{\alpha_n}$ and so its spectrum contains the entire unit circle. 
It now follows from 
\cite[Theorem~1.2]{MR1914564} that the \starhomo from $C^*(E_{u_0,\#})$ to $C^*(\mathcal{S})$ is in fact a \stariso.

Let $\A$ be the subalgebra of $\mathcal{O}_2$ generated by $\setof{ p_v }{ v \in E^0\setminus r(s^{-1}(u_0)) }$ and $\mathcal{E}$.
Note that $\A$ has a unit, and although it does not coincide with the unit of $\mathcal{O}_2$ it does coincide with the unit of $C^*(\mathcal{S})$. 
In fact $\A$ is  a unital subalgebra of $C^*(\mathcal{S})$.

Let us denote by 
$p_{\uuline{u}_0}$, $p_{\uuline{u}}$, $p_{\uuline{w}_1^u}$, $p_{\uuline{w}_2^u}$, $p_{\uuline{w}_3^u}$, $p_{\uuline{w}_4^u}$ and $s_{\uuline{e}_0}$, $s_{\uuline{c}_0^u}$, $s_{\uuline{c}_1^u}$, $s_{\uuline{f}_i^u}$, $i=1,2,\ldots,16$, for $u\in r(s^{-1}(u_0))\setminus\{u_0\}$ the image of the canonical generators of $\mathcal{E}'$ in $\mathcal{O}_2$ under the chosen isomorphism between $\mathcal{E}'$ and $\mathcal{E}$ composed with the embedding into $\mathcal{O}_2$. 

By Proposition~\ref{prop:csdouble-hash}, certain projections in $\mathcal{E}$ are Murray-von~Neumann equivalent, hence we can find partial isometries $y_{e} \in \mathcal{E}$ such that 
\begin{align*}
    y_{c_2^u}^*y_{c_2^u} &= p_{\uline{u}}-\left(s_{\uline{e}_5^u}s_{\uline{e}_5^u}^* +s_{\uline{c}_1^u}s^*_{\uline{c}_1^u}\right), & 
    y_{c_2^u}y_{c_2^u}^* &= p_{\uuline{u}}-\left(s_{\uuline{f}_{11}^u}s_{\uuline{f}_{11}^u}^* +s_{\uuline{c}_1^u}s^*_{\uuline{c}_1^u}\right), 
\intertext{when $e=c_2^u$ for $u\in r(s^{-1}(u_0))\setminus\{u_0\}$, and }
    y_{e}^*y_{e} &= p_{\uline{r(e)}}, &
    y_{e}y_{e}^* &= p_{\uuline{r(e)}}, 
\end{align*}
when $e\in r^{-1}(r(s^{-1}(u_0)))\setminus s^{-1}(r(s^{-1}(u_0)))$ (here $\uline{r(e)}$ and $\uuline{r(e)}$ denotes $\uline{u}_0$ and $\uuline{u}_0$, respectively, if $r(e)=u_0$). 

For all these $e\in \setof{c_2^u }{u\in r(s^{-1}(u_0))\setminus\{u_0\}}\cup (r^{-1}(r(s^{-1}(u_0)))\setminus s^{-1}(r(s^{-1}(u_0))))$, we let 
$$    z_{e} =y_{e} + p, $$
where
$$p=\sum_{v\in E^0\setminus r(s^{-1}(u_0))}p_v.$$
Note that these are partial isometries and $p_v z_e =z_e p_v = p_v$, $v\in E^0\setminus r(s^{-1}(u_0))$, and $s_{e'} z_e =z_e s_{e'} = s_{e'}$ and $s_{e'}^* z_e =z_e s_{e'}^* = s_{e'}^*$, for all edges $e'$ between vertices in $E^0\setminus r(s^{-1}(u_0))$, for all such edges $e$. 

We will define a Cuntz-Krieger $E_{u_0,\#\#}$-family in $\mathcal{O}_2$. 
We let
\begin{align*}
P_{u_0}&=p_{\uuline{u}_0}, \\ 
P_v&=q_v = p_v, &&\text{for each }v \in E^0\setminus r(s^{-1}(u_0)), 
\intertext{and}
P_{u} &= p_{\uuline{u}}, \\
P_{w_i^u} &= p_{\uuline{w}_i^u}, &&\text{for }i=1,2,3,4,
\end{align*}
for all $u\in r(s^{-1}(u_0))\setminus\{u_0\}$. Let moreover 
\begin{align*}
S_{e_0}&=s_{\uuline{e}_0}, \\ 
S_e &= t_e=s_e, &&\text{for all edges }e\text{ between vertices in }E^0\setminus r(s^{-1}(u_0)) \\
S_{e} &= z_e t_e z_e^* = z_e x_1^e z_e^*, &&\text{for all }e\in r^{-1}(r(s^{-1}(u_0)))\setminus s^{-1}(r(s^{-1}(u_0))), 
\intertext{and}
S_{c_i^u} &= s_{\uuline{c}_i^u}, &&i=0,1, \\
S_{c_2^u} &= z_{c_2^u} t_{c_2^u} z_{c_2^u}^* =z_{c_2^u} x_2^{c_2^u} z_{c_2^u}^*, \\
S_{f_i^u} &= s_{\uuline{f}_i^u}, &&i=1,2,\ldots,16,
\end{align*}
for all $u\in r(s^{-1}(u_0))\setminus\{u_0\}$. 
Denote this family by $\mathcal{T}$.

By construction, $\setof{ P_v }{ v \in E_{u_0,\#\#}^0}$ is a set of orthogonal projections and $\setof{ S_e }{ e \in E_{u_0,\#\#}^1}$ is a set of partial isometries.
Furthermore, by the choice of $S_e$ for all edges not touching $r(s^{-1}(u_0))$, the relations are satisfied at all vertices not connected to a vertex in $r(s^{-1}(u_0))$ by an edge. 

Moreover, 
\begin{align*}
S_{e_0}^*S_{e_0} &= s_{\uuline{e}_0}^*s_{\uuline{e}_0} = p_{\uuline{u}_0}=P_{u_0}=P_{r(e_0)},&
S_{e_0}S_{e_0}^* &= s_{\uuline{e}_0}s_{\uuline{e}_0}^* \leq p_{\uuline{u}_0}=P_{u_0}=P_{s(e_0)},\\
S_{c_0^u}^*S_{c_0^u} &= s_{\uuline{c}_0^u}^*s_{\uuline{c}_0^u} = p_{\uuline{u}}=P_{u}=P_{r(c_0^u)}, &
S_{c_0^u}S_{c_0^u}^* &= s_{\uuline{c}_0^u}s_{\uuline{c}_0^u}^* \leq p_{\uuline{u}_0}=P_{u_0}=P_{s(c_0^u)},\\
S_{c_1^u}^*S_{c_1^u} &= s_{\uuline{c}_1^u}^*s_{\uuline{c}_1^u} = p_{\uuline{u}}=P_{u}=P_{r(c_1^u)}, &
S_{c_1^u}S_{c_1^u}^* &= s_{\uuline{c}_1^u}s_{\uuline{c}_1^u}^* \leq p_{\uuline{u}}=P_{u}=P_{s(c_1^u)},\\
S_{c_2^u}^*S_{c_2^u} 
&=p_{r({c_2^u})} =P_{r({c_2^u})} &
S_{c_2^u}S_{c_2^u}^* 
&=y_{c_2^u} y_{c_2^u}^* = p_{\uuline{u}}-\left(s_{\uuline{f}_{11}^u}s_{\uuline{f}_{11}^u}^* +s_{\uuline{c}_1^u}s^*_{\uuline{c}_1^u}\right) \\
&&&\leq p_{\uuline{u}}=P_{u}=P_{s(c_2^u)}, \\
S_{f_{i}^u}^*S_{f_{i}^u}&=s_{\uuline{f}_{i}^u}^*s_{\uuline{f}_{i}^u}=p_{\uuline{r_{E_{u_0,\#\#}}(f_i^u)}} & S_{f_{i}^u}S_{f_{i}^u}^*&=s_{\uuline{f}_{i}^u}s_{\uuline{f}_{i}^u}^*\leq p_{\uuline{s_{E_{u_0,\#\#}}(f_i^u)}} \\
&=P_{r_{E_{u_0,\#\#}}(f_i^u)}, & &=P_{s_{E_{u_0,\#\#}}(f_i^u)}, 
\intertext{for all $u\in r(s^{-1}(u_0))\setminus\{u_0\}$ and $i=1,2,\ldots,16$,}
S_e^*S_e&=s_e^*s_e=p_{r(e)}=P_{r(e)} &
S_eS_e^*&=s_es_e^*\leq p_{s(e)}=P_{s(e)} , 
\intertext{for all edges $e$ between vertices in $E^0\setminus r(s^{-1}(u_0))$,}
S_e^*S_e &= y_e y_e^* =p_{\uuline{r(e)}}=P_{r(e)}, & 
S_eS_e^* &= s_es_e^* \leq p_{s(e)}=P_{s(e)},
\end{align*}
for edges $e$ with $s(e)\in E^1\setminus r(s^{-1}(u_0))$ and $r(e)\in r(s^{-1}(u_0))$. 

Using the above, we immediately see that the Cuntz-Krieger relations hold at the vertices $u_0$ and $w_1^u$, for $u\in r(s^{-1}(u_0))\setminus\{u_0\}$. 
If $v\in E^0\setminus r(s^{-1}(u_0))$ such that $v$ is a finite emitter in $E_{u_0,\#\#}$ and there is an edge from $v$ to a vertex in $r(s^{-1}(u_0))$, then $v$ is also a finite emitter in $E$, and from the above it is clear that the Cuntz-Krieger relations hold at this vertex. 
If $u\in r(s^{-1}(u_0))\setminus\{u_0\}$, then $u$ is regular and
\begin{align*}
\smash{\sum_{e\in s_{E_{u_0,\#\#}}^{-1}(u)}S_eS_e^*} 
&=S_{c_1^u}S_{c_1^u}^* + S_{c_2^u}S_{c_2^u}^* + S_{f_{11}^u}S_{f_{11}^u}^* \\
&=s_{\uuline{c}_1^u}s_{\uuline{c}_1^u}^* + \left(p_{\uuline{u}}-\left(s_{\uuline{f}_{11}^u}s_{\uuline{f}_{11}^u}^* +s_{\uuline{c}_1^u}s^*_{\uuline{c}_1^u}\right)\right) +s_{\uuline{f}_{11}^u}s_{\uuline{f}_{11}^u}^* \\
&=p_{\uuline{u}}=P_{u}.
\end{align*}
Hence $\mathcal{T}$ is a Cuntz-Krieger $E_{u_0,\#\#}$-family. 

The universal property of $C^*(E_{u_0,\#\#})$ provides a surjective \starhomo from $C^*(E_{u_0,\#\#})$ to $C^*(\mathcal{T}) \subseteq \mathcal{O}_2$. 
Let $\alpha_1 \alpha_2 \cdots \alpha_n$ be a vertex-simple cycle in $E_{u_0,\#\#}$ without any exit. 
We see that all the edges $\alpha_i$ must be between vertices in $E^0\setminus r(s^{-1}(u_0))$, and hence we have
\begin{align*}
	S_{\alpha_1} S_{\alpha_2} \cdots S_{\alpha_n} = s_{\alpha_1} s_{\alpha_2}  \cdots s_{\alpha_n}  
\end{align*} 
and so its spectrum contain the entire unit circle. 
It now follows from 
\cite[Theorem~1.2]{MR1914564} that the \starhomo from $C^*(E_{u,\#\#})$ to $C^*(\mathcal{T})$ is in fact a \stariso.

Since $\A \subseteq C^*(\mathcal{S})$ and since 
$p_{\uuline{u}_0},p_{\uuline{u}},p_{\uuline{w}_i^u}, s_{\uuline{c}_0^u}, s_{\uuline{c}_1^u}, s_{\uuline{e}_0}, s_{\uuline{f}_j^u}\in \mathcal{E} \subseteq C^*(\mathcal{S})$, for $i=1,2,3,4$, $j = 1,2,\ldots, 16$, $u\in r(s^{-1}(u_0))\setminus\{u_0\}$, we have that $\mathcal{T} \subseteq C^*(\mathcal{S})$. 
So $C^*(\mathcal{T}) \subseteq C^*(\mathcal{S})$. 
But since $\A$ is also contained in $C^*(\mathcal{T})$ and $\mathcal{E} \subseteq C^*(\mathcal{T})$, it follows from 
\begin{align*}
t_e&= x_1^e = s_es_e^* x_1^e p_{\uline{r(e)}}  =(y_e^*y_e+p)s_es_e^* x_1^e p_{\uline{r(e)}} (y_e^*y_e + p ) \\
&=z_e^*z_ex_1^e z_e^*z_e =z_e^* S_e z_e 
\intertext{for all $e\in r^{-1}(r(s^{-1}(u_0)))\setminus s^{-1}(r(s^{-1}(u_0)))$,}
t_{c_2^u}&=x_2^{c_2^u} 
=(p_{\uline{u}} - (s_{\uline{e}_5^{u}} s_{\uline{e}_5^{u}}^* + s_{\uline{c}_1^{u}} s_{\uline{c}_1^{u}}^*))x_2^{c_2^u}p_{r({c_2^u})} \\
&= (y_{c_2^u}^*y_{c_2^u}+p)(p_{\uline{u}} - (s_{\uline{e}_5^{u}} s_{\uline{e}_5^{u}}^* + s_{\uline{c}_1^{u}} s_{\uline{c}_1^{u}}^*))x_2^{c_2^u}p_{r({c_2^u})}(y_{c_2^u}^*y_{c_2^u}+p) \\
&= z_{c_2^u}^*z_{c_2^u}x_2^{c_2^u}z_{c_2^u}^*z_{c_2^u} \\
&= z_{c_2^u}^*S_{c_2^u}z_{c_2^u} ,
\end{align*}
for all $u\in r(s^{-1}(u_0))\setminus\{u_0\}$, that we have that $\mathcal{S} \subseteq C^*(\mathcal{T})$, and hence $C^*(\mathcal{S}) \subseteq C^*(\mathcal{T})$. 
Therefore 
\[
	C^*(E_{u_0,\#}) \cong C^*(\mathcal{S}) = C^*(\mathcal{T}) \cong C^*(E_{u_0,\#\#}).
\]

Now assume that all cyclic components are singletons and $\gcd(\Bsf_E\{i\})=1$ for all $i\in\calP$ corresponding to non-cyclic strongly connected components. 
Then we know from Theorem~\ref{thm:mainBH} that there exists a \GLPEe $(U,V)$ from $-\iota_{\mathbf{r}}(-\Bsf_{E_{u_0,\#}})$ to $\Bsf_{E_{u_0,\#\#}}$ such that $\FKR(\calP;\Phi)=\FKR(U,V)$. 

It is clear that the reduced filtered $K$-theory isomorphism from $C^{*}(\mathcal{S})$ to $C^{*}(\mathcal{T})$ has to be positive on $K_0$ of the cyclic components, so it has to be the identity with the canonical identification of $K_0$ as cokernels --- \ie, $V\{i\}=1$ for all $i\in\calP$ corresponding to cyclic components. 
Moreover, from the construction above, it is clear that it is giving the identity map on $K_1$ for all other cyclic components than $\{u_0\}$ (since $t_e$ is mapped to $S_e$ for all edges in other cyclic components) --- \ie, $U\{i\}=1$ for all cyclic components $i\neq j$. 

From Proposition~\ref{prop:csdouble-hash} it follows that $[s_{\uline{e}_0}]_1=-[s_{\uuline{e}_0}]_1$ in $K_1$ of the gauge simple subquotient of $C^*(\mathcal{S})=C^*(\mathcal{T})$ corresponding to the component $\{u_0\}$. Thus it follows that the isomorphism changes the sign of $K_1$ for this component --- \ie, $U\{j\}=-1$. 
\end{proof}

Now we have the following immediate consequence.

\begin{corollary}\label{cor:hash-1}
Let $E$ be a graph with finitely many vertices and let $u_0$ be a vertex of $E$ such that $E$ and $u_0$ satisfy Assumption~\ref{assumption:1-hash}. 
Then there exists a \calP-equivariant isomorphism from $C^{*}(E_{u_0,\#})$ to $C^{*}(E_{u_0,\#\#})$. 
\end{corollary}

Thus we have the following fundamental result.

\begin{proposition}\label{prop:hash}
Let $E$ be a graph with finitely many vertices, and let $u_0$ be a vertex in $E$ such that Assumption~\ref{assumption:1-hash} holds. 
Then there exists a $\Prime_\gamma(C^*(E))$-equivariant isomorphism $\Phi$ from $C^*(E)\otimes\K$ to $C^*(E_{u_0,\#})\otimes\K$.

If, moreover, $j\in\calP$ corresponds to the component $\{u_0\}$, $\Bsf_E\in\MPZccc$, $\gcd(\Bsf_E\{i\})=1$ for all $i\in\calP$ corresponding to non-cyclic strongly connected components, then $\Bsf_{E_{u_0,\#}}\in\MPZccc[(\mathbf{m}+2\mathbf{r})\times(\mathbf{m}+2\mathbf{r})]$ and we can choose $\Phi$ such that there exists a \GLPEe $(U,V)$ from $-\iota_{4\mathbf{r}}(-B_{E})$ to $-\iota_{2\mathbf{r}}(-B_{E_{u_0,\#}})$ such that $U\{i\}$ and $V\{i\}$ are the $1\times 1$ matrix $1$ for all cyclic components $i\neq j$, $U\{j\}=-1$ and $V\{j\}=1$ and  $\FKR(\calP;\Phi)=\FKRs(U,V)$ under the canonical identifications, where $\mathbf{r}=(r_i)_{i=1}^N$ is defined by $r_i=1$ for all immediate successors $i$ of $j$, and $r_i=0$ otherwise. 
\end{proposition}

\begin{proof}
The first part is a direct consequence of Corollary~\ref{cor:hashmovetwice} and Corollary~\ref{cor:hash-1}.

Now assume in addition that $j\in\calP$ corresponds to the component $\{u_0\}$, $\Bsf_E\in\MPZccc$, $\gcd(\Bsf_E\{i\})=1$ for all $i\in\calP$ corresponding to non-cyclic strongly connected components. 
By Theorem~\ref{thm:hash-1}, Proposition~\ref{prop:hashmovetwice} and Theorem~\ref{thm:SLP-equivalence-implies-stable-isomorphism}, the conclusions of the proposition follow. 
\end{proof}

\subsection{Transform graphs from canonical form to satisfy Assumption~\ref{assumption:1-hash}}
\label{subsec:transform-to-apply-spec-pulelehua}
Suppose that there are given graphs $E$ and $F$ with finitely many vertices such that $(\Bsf_E,\Bsf_F)$ is in standard form and a \GLPEe $(U,V)$ from $\Bsf_E^\bullet$ to $\Bsf_F^\bullet$ that satisfies that $V\{i\}=1$ for all $i\in\calP$ with $n_i=1$. Then the purpose of this subsection is to show that the problem of lifting $\FKR(U,V)$ to a stable isomorphism can be reduced to the case where we moreover assume that $U\{i\}=1$ for all $i\in\calP$ with $m_i=1$. 

\begin{lemma}\label{lem: canonical form to new form}
Let $E$ be a graph in canonical form and let $\Bsf_E \in \MPZccc$.  Suppose $\{u_0\}$ is a cyclic strongly connected component and the immediate successors of $\{u_0\}$ are all non-cyclic strongly connected components. Let $\calP_0$ denote the set of immediate successors of $\{u_0\}$, and assume that $\calP_0\neq\emptyset$.  
Let $\mathbf{r}=(r_i)_{i\in\calP}$, where $r_i = 1$ for $i \in \calP_0$, and $r_i = 0$ for $i \notin \calP_0$.

Then there exists a graph $F$ such that $\Bsf_F\in\MPZccc[(\mathbf{m}+3\mathbf{r})\times(\mathbf{n}+3\mathbf{r})]$, $E \MCeq F$, with $E^0 \subseteq F^0$ and $F$ satisfies the assumptions in Assumption~\ref{assumption:1-hash}.  
Moreover, $F$ can be chosen such that there exists a \calP-equivariant isomorphism $\Psi$ from $C^*(E)\otimes\K$ to $C^*(F)\otimes\K$ and a \GLPEe $(U,V)$ from $-\iota_{3\mathbf{r}}(-\Bsf_{E}^\bullet)$ to $\Bsf_F^\bullet$ that satisfy, that $U\{i\}=1$ for all $i\in\calP$ with $m_i=1$, $V\{i\}=1$ for all $i\in\calP$ with $n_i=1$ and $\FKR(\calP;\Psi)=\FKRs(U,V)$. 
\end{lemma}

\begin{proof}
For each $i\in\calP_0$ we let $S_{ i }$ denote the subset of $E^0$ consisting of the vertices of the component corresponding to $i$. Then $S_{ i } \cap S_{ j } = \emptyset$ when $i \neq j$ for $i,j\in\calP_0$. Moreover, $S_{i} = H_i \setminus H_i'$ for saturated hereditary subsets $H_i$ and $H_i'$ of $E^0$ with $H_i' \subseteq H_i$, for every $i\in\calP$.  By assumption, the subgraph $E_{S_{i} }$ has the property that it has at least one regular vertex and all the entries of $\Bsf_{ E_{S_{i} } }$ are positive or $\infty$, for $i\in\calP_0$. 

We now construct $F$.  For each $i\in\calP_0$, choose a regular vertex $w_i \in S_{i}$.  Since the entries of $\Bsf_{E_{S_{i}} }$ are positive or $\infty$, $w_i$ supports at least 2 loops.  Therefore, we may perform Move \CC at each $w_i$, for $i\in\calP$, to get a graph $E_1$. Clearly $E\MCeq E_1$.  Using Corollary~\ref{cor:cuntzspliceinvariant}, we see that $\Bsf_{E_1}\in\MPZccc[(\mathbf{m}+2\mathbf{r})\times(\mathbf{n}+2\mathbf{r})]$, and there exist $U\in\SLPZ[\mathbf{m}+2\mathbf{r}]$ and $V\in\GLPZ[\mathbf{n}+2\mathbf{r}]$ such that the following holds. All the diagonal blocks of $V$ have determinant $1$ except for the diagonal blocks $V\{i\}$, for $i\in\calP_0$, which have determinant $-1$ and $(U,V)$ is a \GLPEe from $-\iota_{2\mathbf{r}}(-\Bsf_{E}^\bullet)$ to  $\Bsf_{E_1}^\bullet$. 
Moreover, there exists a $\calP$-equivariant isomorphism $\Phi$ from $C^{*}(E) \otimes \K$ to $C^{*}(E_1) \otimes \K$ such that $\FKR(\calP;\Phi)=\FKRs(U,V)$. In particular, the \GLPEe $(U,V)$ induces the identity on $K$-theory corresponding to all components that are not non-cyclic strongly connected components (\ie, those satisfying $n_i=1$). 

Let $i\in\calP_0$ be given, and let $v_{1,i}$ and $v_{2,i}$ be the vertices from the Cuntz-splice.  We will now do column and row operations to $\Bsf_{E_1}$ such that each operation is legal in the sense that the resulting graph is move equivalent to $E_1$, and we obtain corresponding \calP-equivariant isomorphism and \SLPEe.  First add column $v_{1,i}$ to column $v$ for all $v$ with the property that $s_E^{-1}( u_0 ) \cap r_E^{-1}( v ) \neq \emptyset$ and $v_{1,i} \geq v$.  Now subtract column $v_{2,i}$ from column $w_i$.  Next, add row $v_{1,i}$ to row $u_0$.   We now add row $v_{2,i}$ to row $v_{1,i}$, $M_i$ times, where 
\[
M_i = \sum_{ \substack{ v \in E^0 ,  w_i \geq v \\ | s_E^{-1}( u_0 ) \cap r_E^{-1}( v ) | \geq 1 } } | s_E^{-1}( u_0 ) \cap r_E^{-1}( v ) |. 
\] 
Lastly, subtract column $v_{2,i}$ from column $v$, $| s_E^{-1}( u_0 ) \cap r_E^{-1}( v )|$ times, for all $v$ with the property that $s_E^{-1}( u_0 ) \cap r_E^{-1}( v ) \neq \emptyset$ and $v_{1,i} \geq v$. 

Using the resulting graph instead of $E$ we successively do the analogue operations for all the other $i \in\calP_0$.
We have now constructed a graph $E_2$ such that $E_2 \Meq E_1$ and $E_2$ has all the desired properties except for ``each vertex in $r(s^{-1}(u_0))\setminus\{u_0\}$ has exactly one loop and it has exactly one other edge going out (to a vertex that is part of the same component)'', \cf\ Assumption~\ref{assumption:1-hash}.  Moreover, since $E_2$ is obtained from $E_1$ by column and row operations, Theorem~\ref{thm:SLP-equivalence-implies-stable-isomorphism} gives us a \calP-equivariant isomorphism $\Psi_1$ from $C^*(E_1)\otimes\K$ to $C^*(E_2)\otimes\K$, such that $\FKR(\calP;\Psi_1)$ is induced by an \SLPEe from $\Bsf_{E_1}^\bullet$ and $\Bsf_{E_2}^\bullet$.
In particular, $E\MCeq E_2$ and we have a \calP-equivariant isomorphism $\Psi_1\circ\Phi$ from $C^*(E)\otimes\K$ to $C^*(E_2)\otimes\K$ such that $\FKR(\calP;\Psi_1\circ\Phi)$ is induced by a \GLPEe from $ -\iota_{2\mathbf{r}}(-\Bsf_{E}^\bullet)$ to $\Bsf_{E_2}^\bullet$ that is the identity on the $K$-theory corresponding to all components that are not non-cyclic strongly connected (\ie, those satisfying $n_i=1$).  

Now let $f_i$ be the edge going from $v_{2,i}$ to $v_{1,i}$. We edge expand this edge and call the new vertex $z_i$.
Then we add row $v_{2,i}$ to row $z_i$. Finally we subtract column $z_i$ from each column (other than $z_i$ and $v_{2,i}$) that receives an edge from $v_{2,i}$. 
We do this for each $i\in\calP_0$ and call the resulting graph $F$. 

Since we are only doing legal column and row operations, we have that $E_2 \Meq F$ and that there exists a \calP-equivariant isomorphism $\Psi_2$ from $C^*(E_2)\otimes\K$ to $C^*(F)\otimes\K$ such that $\FKR(\calP;\Psi_2)$ is induced by an \SLPEe from $-\iota_{\mathbf{r}}(-\Bsf_{E_2}^\bullet)$ to $\Bsf_F^\bullet$ (\cf\ Propositions~\ref{prop:edge-expansion} and~\ref{prop:toke}). 

By composing the above, we get the desired result. 
\end{proof}

\begin{lemma}\label{lem:nosucc}
Let $E=(E^{0},E^{1},r,s)$ be a graph with $\Bsf_E\in\MPZccc$. 
Assume that $i\in\calP$, that $i$ corresponds to a cyclic component, and that $i$ does not have any successors. 
Then there exists a \calP-equivariant automorphism $\Phi$ of $C^{*}(E)$ such that $\FKR(\calP;\Phi)=\FKR(U,I)$, where $(U,I)$ is the \GLPEe from $\Bsf_E^{\bullet}$ to $\Bsf_E^{\bullet}$ that satisfies that $U\{i\}=-1$ and it is the identity everywhere else. 
\end{lemma}
\begin{proof}
Let $v_0$ be the vertex corresponding to $i$ and let $e_0$ be the unique edge going out from $v_0$. Note that $e_0$ is a loop. 
Let $(p_v)_{v\in E^{0}}$ and $(s_{e})_{e\in E^{1}}$ be a generating Cuntz-Krieger $E$-family for $C^{*}(E)$.
Let $P_v=p_v$, for all $v\in E^{0}$, and $S_e=s_e$, for all $e\in E^{1}\setminus \{e_0\}$. 
Let, moreover, $S_{e_0}=s_{e_0}^{*}$. 
It is clear that $(p_v)_{v\in E^{0}}$ and $(s_{e})_{e\in E^{1}}$ is a Cuntz-Krieger $E$-family. Thus we get a canonical \calP-equivariant homomorphism from $C^{*}(E)$ to $C^{*}(E)$, that is clearly surjective. 
That it is injective follows from 
\cite[Theorem~1.2]{MR1914564}. 
It is clear that the induced map on reduced filtered $K$-theory is as described. 
\end{proof}

\begin{corollary}\label{cor:Pulelehua}
Let $E_1$ and $E_2$ be graphs with finitely many vertices such that $\Bsf_{E_1},\Bsf_{E_2}\in\MPZccc$ and $(\Bsf_{E_1},\Bsf_{E_2})$ is in standard form. 
Assume that we have an isomorphism $\varphi\colon\FKRplus(\calP;C^*(E_1))\rightarrow\FKRplus(\calP;C^*(E_2))$ and a \GLPEe $(U,V)$ from $\Bsf_{E_1}^\bullet$ to $\Bsf_{E_2}^\bullet$ such that $\varphi=\FKR(U,V)$. 

Then there exists a pair of graphs $(F_1,F_2)$ such that $(\Bsf_{F_1} , \Bsf_{F_2} )$ is in standard form with $E_i\MCPeq F_i$, for $i=1,2$, $\Bsf_{F_1}, \Bsf_{F_2}\in\MPZccc[\mathbf{m}'\times\mathbf{n}']$ for some $\mathbf{m}'\geq\mathbf{m}$ and $\mathbf{n}'\geq\mathbf{n}$ such that we have a \calP-equivariant isomorphism $\Psi_i$ from $C^*(E_i)\otimes\K$ to $C^*(F_i)\otimes\K$, for $i=1,2$, and a \GLPEe $(U',V')$ from $\Bsf_{F_1}^\bullet$ to $\Bsf_{F_2}^\bullet$ such that 
$\varphi'=\FKRs(U',V')$ with 
$U'\{i\}=1$ for all $i\in\calP$ with $m_i'=1$ and $V'\{i\}=1$ for all $i\in\calP$ with $n_i'=1$, where $\varphi'=\FKR(\calP;\Psi_2)\circ\FKR(\calP;\kappa_{C^*(E_2)})\circ\varphi\circ\FKR(\calP;\kappa_{C^*(E_1)})^{-1}\circ\FKR(\calP;\Psi_1)^{-1}$. 
\end{corollary}
\begin{proof}
Since $\varphi$ is positive on $K_0$ of the gauge simple subquotients, we get that $V\{i\}$ is the $1\times 1$ matrix $1$ for all $i\in\calP$ with $n_i=1$. 

Suppose that $i\in\calP$, that $i$ corresponds to a cyclic component (\ie, $m_i=1$), and $i$ has an immediate successor $j$ that does not correspond to a noncyclic strongly connected component.
Then the \GLPEe $(U,V)$ satisfies that $V\{i\}=1=V\{j\}$ and that $\Bsf_{E_1}^\bullet\{i,j\}$ and $\Bsf_{E_2}^\bullet\{i,j\}$ are positive integers. 
Thus it follows directly from this and the fact that $(U,V)$ is a \GLPEe from $\Bsf_{E_1}^\bullet$ to $\Bsf_{E_2}^\bullet$ that $U\{i\}=1$.

Now assume that we have a $j\in\calP$ such that $m_j=1$ and $U\{j\}=-1$, and let $u_0$ be the vertex corresponding to $j$. 
Then all immediate successors of $j$ correspond to noncyclic strongly connected components.  Let $\mathbf{r}=(r_i)_{i\in\calP}$ be the multiindex given by $r_i=1$ for all immediate successors $i$ of $j$, and $r_i=0$ otherwise.  First assume that $j$ has successors. 
Now it follows from Lemma~\ref{lem: canonical form to new form}, that there exists a graph $E_1'$ such that $E_1\MCeq E_1'$, $\Bsf_{E_1'}\in\MPZccc[(\mathbf{m}+3\mathbf{r})\times(\mathbf{n}+3\mathbf{r})]$ and $E_1'$ together with $u_0$ satisfies Assumption~\ref{assumption:1-hash}. Moreover, $E_1'$ can be chosen such that there exist a \calP-equivariant isomorphism $\Phi_1$ from $C^{*}(E_1)\otimes\K$ to $C^{*}(E_1')\otimes\K$ and a \GLPEe $(U',V')$ from $-\iota_{3\mathbf{r}}(-\Bsf_{E_1}^\bullet)$ to $\Bsf_{E_1'}^{\bullet}$ that satisfy, that $U'\{i\}=1$ for all $i\in\calP$ with $m_i=1$, $V'\{i\}=1$ for all $i\in\calP$ with $n_i=1$, and $\FKR(\calP;\Phi_1)=\FKR(U',V')$ (under the canonical identifications). 
Now it follows from Proposition~\ref{prop:hash} that there exist a graph $E_1''=(E_1')_{u_0,\#}$ such that $E_1'\MCPeq E_1''$, $\Bsf_{E_1''}\in\MPZccc[(\mathbf{m}+5\mathbf{r})\times(\mathbf{n}+5\mathbf{r})]$, a \calP-equivariant isomorphism $\Phi_1'$ from $C^*(E_1')\otimes\K$ to $C^*(E_1'')\otimes\K$ and a \GLPEe $(U'',V'')$ from $-\iota_{4\mathbf{r}}(-\Bsf_{E_1'}^\bullet)$ to $-\iota_{2\mathbf{r}}(-\Bsf_{E_1''}^\bullet)$ satisfying 
$U''\{i\}=1$ for all $i\in\calP\setminus\{j\}$ with $m_i=1$ and $V''\{i\}=1$ for all $i\in\calP$ with $n_i=1$, $U''\{j\}=-1$, and $\FKR(\calP;\Phi_1')=\FKRs(U'',V'')$. 
Let $E_2''=E_2$. 
Using Lemma~\ref{lem:we-can-put-it-in-canonical-form}, we get graphs $E_1'''$ and $E_2'''$ in canonical form such that 
$E_i''\Meq E_i'''$, $\Bsf_{E_i'''}\in\MPZccc[(\mathbf{m}+7\mathbf{r})\times(\mathbf{n}+7\mathbf{r})]$, a \calP-equivariant isomorphism $\Phi_i''$ from $C^*(E_i'')\otimes\K$ to $C^*(E_i''')\otimes\K$ and $\FKR(\calP;\Phi_i'')$ is induced by an \SLPEe from $-\iota_{2\mathbf{r}}(-\Bsf_{E_i''}^\bullet)$ to $\Bsf_{E_i'''}^\bullet$, for $i=1,2$. 
By composing all the above, we have graphs $E_1'''$ and $E_2'''$ in canonical form such that $E_i\MCPeq E_i'''$, $\Bsf_{E_i'''}\in\MPZccc[(\mathbf{m}+7\mathbf{r})\times(\mathbf{n}+7\mathbf{r})]$, a \calP-equivariant isomorphism $\Psi_i$ from $C^*(E_i)\otimes\K$ to $C^*(E_i''')\otimes\K$ for $i=1,2$, $\FKR(\calP;\Psi_1)=\FKRs(U_1''',V_1''')$ for a \GLPEe $(U_1''',V_1''')$ from $-\iota_{7\mathbf{r}}(-\Bsf_{E_1}^\bullet)$ to $\Bsf_{E_1'''}^\bullet$ satisfying $U_1'''\{i\}=1$, for all $i\neq j$ with $m_i=1$, and $V_1'''\{i\}=1$, for all $i\in\calP$ with $n_i=1$, $U_1'''\{j\}=-1$, $\FKR(\calP;\Psi_2)=\FKRs(U_2''',V_2''')$ for an \SLPEe $(U_2''',V_2''')$ from $-\iota_{7\mathbf{r}}(-\Bsf_{E_2}^\bullet)$ to $\Bsf_{E_2'''}^\bullet$.
We see that
$\varphi'=\FKRs(U_2'''U(U_1''')^{-1},(V_1''')^{-1}VV_2''')$ with 
$U_2'''U(U_1''')^{-1}\{i\}=U\{i\}$ for all $i\in\calP$ with $m_i=1$ and $i\neq j$ and $(V_1''')^{-1}VV_2'''\{i\}=1$ for all $i\in\calP$ with $n_i=1$, where $\varphi'=\FKR(\calP;\Psi_2)\circ\FKR(\calP;\kappa_{C^*(E_2)})\circ\varphi\circ\FKR(\calP;\kappa_{C^*(E_1)})^{-1}\circ\FKR(\calP;\Psi_1)^{-1}$. 
If $j$ does not have any successors, then it follows from Lemma~\ref{lem:nosucc} that there exists a \calP-equivariant automorphism $\Psi'$ of $C^*(E_1)$ such that $\FKR(\calP;\Psi')=\FKR(U',I)$, where $(U',I)$ is a \GLPEe from $\Bsf_{E_1}$ to $\Bsf_{E_1}$ satisfying that $U'\{j\}=-1$ and $U'$ is equal to the identity everywhere else. 
So we get the same result.

Now the corollary follows by induction. 
\end{proof}

\subsection{Invariance of the \Pul move}
\label{subsec:invariance-of-pulelehua}
In this subsection we show the invariance of Move~\PP from Definition~\ref{def:hashmove}, \ie, $C^*(E)\otimes\K\cong C^*(E_{u,P})\otimes\K$ whenever $E$ and $u\in E^0$ satisfy the conditions in Definition~\ref{def:hashmove}. We have already shown this in Proposition~\ref{prop:hash} for the special case when $E$ and $u$ satisfy the conditions of Assumption~\ref{assumption:1-hash}, so we show that we can reduce to that case.

\begin{proposition}\label{prop:first-step-of-invariance-of-Pulelehua-move}
Let $E$ be a graph and let $u$ be a regular vertex that supports a loop and no other return path, the loop based at $u$ has an exit, and if $w \in E^0 \setminus \{ u \}$ and $s_E^{-1} (u) \cap r_E^{-1}(w) \neq \emptyset$, then $w$ is a regular vertex and supports at least two distinct return paths.  Then there exists a graph $F$ with finitely many vertices and a regular vertex $v$ such that 
\begin{itemize}
\item $v$ supports a loop and no other return path;

\item the loop based at $v$ has an exit;

\item if $w \in F^0 \setminus \{ v \}$ and $s_F^{-1}(v) \cap r_F^{-1}(w) \neq \emptyset$, then $w$ is a regular vertex and supports at least two distinct return paths;

\item $E \MCeq F$;

\item $E_{u, P } \MCeq F_{v, P }$; and

\item $\Bsf_F \in \MPZccc$.
\end{itemize}
\end{proposition}

\begin{proof}
Note that $w$ is a regular source of $E_{u, P}$ if and only if $w \in E^0$ and $w$ is a regular source of $E$.  Thus, removing all regular sources of $E$, then removing all new regular sources, then continuing this process, we get a graph $G$ with $u \in G^{0} \subseteq E^{0}$ such that $G \Meq E$ and $E_{u,P} \Meq G_{u, P}$.  So, we may assume that $E$ has no regular sources.  Note that the regular vertices that do not support a cycle in $E^{0}$ are precisely the regular vertices that do not support a cycle in $E_{u,P}$.  So collapsing these regular vertices that do not support a cycle results in a graph $G$ with $u \in G^{0} \subseteq E^{0}$ such that $E \Meq G$ and $E_{u,P} \Meq G_{u,P}$.  Therefore, we may assume that every regular vertex of $E$ supports a cycle. 

Note that the regular vertices of $E$ that support exactly one return path are precisely the regular vertices of $E_{u,P}$ that support exactly one return path.  Thus, collapsing all vertices that support exactly one return path results in a graph $G$ with $u \in G^{0} \subseteq E^{0}$ such that $E \Meq G$ and $E_{u,P} \Meq G_{u,P}$.  Hence, we may assume that every vertex of $E$ that supports exactly one return path supports precisely one loop and no other return path.  

We will now outsplit all infinite emitters $w$ of $E$ with the property that there exists $w' \in E^{0}$ such that $0 < | s_{E}^{-1} ( w ) \cap r_{E}^{-1} (w') | < \infty$.  Let
\[ 
\mathcal{T} = \setof{ w \in E^{0} }{ | s_{E}^{-1} (w) | = \infty\text{ and }\exists w' \in E^{0}\text{ such that } 0 < | s_{E}^{-1} ( w ) \cap r_{E}^{-1} (w') | < \infty }.
\]
For each $w \in \mathcal{T}$, partition $s_{E}^{-1} (w)$ as follows:  
\[
\mathcal{E}_{1,w} = \setof{ e \in s_E^{-1}(w) }{ | s_{E}^{-1} ( w ) \cap r_{E}^{-1} (r_{E}(e) ) | < \infty }
\]
and $\mathcal{E}_{2,w} = s_{E}^{-1} (w) \setminus \mathcal{E}_{1,v}$.  Performing Move~\OO to $E$ at $w$ for every $w \in \mathcal{T}$, we get a graph $E_{1}$ such that $E \Meq E_{1}$ and $E_{1}$ has no breaking vertices.   

Note that $s_{E}^{-1} (u) \cap r^{-1}_{E}(w) = \emptyset$  for all $w \in \mathcal{T}$ since $w$ is an infinite emitter and $s_{E}^{-1} (u) \cap r^{-1}_{E}(w') \neq \emptyset$ implies $w'$ is a regular vertex.  Therefore, $s_{E}^{-1} (w) = s_{E_{u,P}}^{-1}(w)$.  Partition $s_{E_{u,P}}^{-1}(w)$ using the partition of $\mathcal{E}_{1,w} \sqcup \mathcal{E}_{2,w}$, \ie,
\[
s_{E_{u,P}}^{-1}(w) = s_{E}^{-1} (w) = \mathcal{E}_{1,w} \sqcup \mathcal{E}_{2,w}.
\]
Performing Move~\OO to $E_{u, P}$ at $w$ for each $w \in \mathcal{T}$, using the above partition, we get a graph $G_{1}$ with no breaking vertices and $u \in G_{1}^{0}$.  Note that $G_{1} = (E_{1})_{u,P}$. 

We may have introduced new regular vertices that are not base point of a cycle but using a similar argument as in the first paragraph, we obtain a graph $F$ with $u \in F^{0}$ such that $F$ and $u$ has the desired property.
\end{proof}

\begin{proposition}\label{prop:second-step-of-invariance-of-Pulelehua-move}
Let $E$ be a graph and let $u$ be a regular vertex that supports a loop and no other return path, the loop based at $u$ has an exit, and if $w \in E^0 \setminus \{ u \}$ and $s_E^{-1} (u) \cap r_E^{-1}(w) \neq \emptyset$, then $w$ is a regular vertex and supports at least two distinct return paths.  Suppose $\Bsf_{E} \in \MPZccc$ and that $j\in\calP$ corresponds to the component $\{u\}$.  Set 
\[
\calP_0 = \setof{ i \in \calP }{ \exists w\text{ in the }i\text{'th component, such that }s_E^{-1}(u) \cap r_E^{-1}(w) \neq \emptyset }.
\]
Then there exists a \GLPEe, $(U,V)$, from $\Bsf_{E_{u, P} }$ to $- \iota_{\mathbf{r}} ( - \Bsf_E )$ such that $U\{i\}=1$ for all $i\neq j$ with $m_i=1$, $V\{i\}=1$ for all $i$ with $n_i=1$, and $U \{ j \} = -1$, where $r_i = 0$ for all $i \notin \calP_0$ and 
\[
r_i = 2\cdot \left| \setof{ w }{ w\text{ is in the }i\text{'th component and }s^{-1}(u) \cap r^{-1}(w) \neq \emptyset } \right|
\]
for all $i \in \calP_0$.

Denote the set $\setof{ w \in E^0 \setminus \{u\} }{ s_E^{-1}(u) \cap r_E^{-1}(w) \neq \emptyset }$ by $S$.  Then there exists a \GLPEe, $(U_1,V_1)$, from $\Bsf_{E_{u, P} }$ to $\Bsf_{E_{S,-}}$ such that $U_1\{i\}$ and $V_1\{i\}$ are \SL-matrices for $i \neq j$ and $U_1 \{ j \} = -1$ and $V_1 \{ j \} = 1$.     
\end{proposition}

\begin{proof}
Throughout the proof, $E_{v,w}$ will denote the matrix that acts on the left by adding row $w$ to row $v$ and acts on the right by adding column $v$ to column $w$ and $S_{v,w}$ will denote the matrix that acts on the right by switching columns $v$ and $w$. 

Let $S = \{ w_1, \dots, w_t \}$ and let $n_k = | s_E^{-1} (u) \cap r_E^{-1}(w_k) |$.  Set 
\[
V = E_{v_{2}^{w_1},w_1}^{-1} \cdots E_{ v_{2}^{w_t} , w_t }^{-1} S_{v_1^{w_1} , v_2^{w_1} }\cdots S_{v_1^{w_t} , v_2^{w_t} }.
\]  
Let $U'$ be the identity except at the $j$'th component in which $U' \{ j \} = -1$.   Set
\[
U = U' E_{u, v_{1}^{w_t} }^{-2n_t} \cdots E_{u, v_{1}^{w_1} }^{-2n_1} E_{w_t, v_{2}^{w_t} }^{-1} \cdots E_{w_1, v_{2}^{w_1} }^{-1}.
\] 
A computation shows that $(U,V)$ is a \GLPEe from $\Bsf_{E_{u, P} }$ to $- \iota_{\mathbf{r}} ( - \Bsf_E )$.  By construction, $U\{i\}=1$ for all $i\neq j$ with $m_i=1$, $V\{i\}=1$ for all $i$ with $n_i=1$, and $U \{ j \} = -1$.  

Let $V_1$ be the identity matrix and let $U_1 = U' E_{u, v_{1}^{w_t} }^{-2n_t} \cdots E_{u, v_{1}^{w_1} }^{-2n_1}$.  A computation shows that $(U_1,V_1)$ is a \GLPEe from $\Bsf_{E_{u, P} }$ to $\Bsf_{E_{S,-}}$ such that $U_1\{i\}$ and $V_1\{i\}$ are \SL-matrices for $i \neq j$ and $U_1 \{ j \} = -1$ and $V_1 \{ j \} = 1$.
\end{proof}

\begin{theorem}\label{thm:InvarianceOfPulelehua}
Let $E=(E^0,E^1,r,s)$ be a graph with finitely many vertices and let $u$ be a regular vertex that supports a loop and no other return path, the loop based at $u$ has an exit, and if $w \in E^0 \setminus \{u\}$ and $s_E^{-1} (u) \cap r_E^{-1}(w) \neq \emptyset$, then $w$ is a regular vertex that supports at least two distinct return paths.  

Then $C^*(E)\otimes\K\cong C^*(E_{u, P})\otimes\K$. 
\end{theorem}

\begin{proof}
By Proposition~\ref{prop:first-step-of-invariance-of-Pulelehua-move} and \cite[Theorem~4.8]{MR3713535}, we may assume that $\Bsf_E\in\MPZccc$.  By Proposition~\ref{prop:second-step-of-invariance-of-Pulelehua-move}, Lemma~\ref{lem:we-can-put-it-in-canonical-form}, Corollary~\ref{cor:Pulelehua}, Theorem~\ref{thm:GLtoSL}, and Theorem~\ref{thm:SLP-equivalence-implies-stable-isomorphism}, we have that $C^*(E) \otimes \K \cong C^*( E_{u, P} ) \otimes \K$.
\end{proof}

\section{Proof of the main results}
\label{sec:proof}
According to the explanations following the main results, we only need to prove Theorem~\ref{thm:main-3} and that $\ref{thm:main-1-item-3}\Rightarrow\ref{thm:main-1-item-1}$ in Theorem~\ref{thm:main-1}. This is done in this section, using the results developed above.
 
\begin{proof}[Proof of Theorem~\ref{thm:main-3} and $\ref{thm:main-1-item-3}\Rightarrow\ref{thm:main-1-item-1}$ in Theorem~\ref{thm:main-1}]
Let $E_1$ and $E_2$ be graphs with finitely many vertices. 
Set $X=\Prime_\gamma(C^*(E))$, assume that we have a homeomorphism from $\Prime_\gamma(C^*(E_1))$ to $\Prime_\gamma(C^*(E_2))$ and view $C^*(E_1)$ and $C^*(E_2)$ as $X$-algebras under this homeomorphism. 
Assume, moreover, that we have an isomorphism $\varphi\colon\FKRplus(X;C^*(E_1))\rightarrow\FKRplus(X;C^*(E_2))$. 
Now we want to find an $X$-equivariant isomorphism $\Phi\colon C^*(E_1)\otimes\K\rightarrow C^*(E_2)\otimes\K$ such that $\FKR(X;\Phi)\circ\FKR(X;\kappa_{C^*(E_1)})=\FKR(X;\kappa_{C^*(E_2)})\circ \varphi$. We also want to show that $E_1\MCPeq E_2$.

According to Proposition~\ref{prop:standard-form}, we can find graphs $F_1$ and $F_2$ such that $E_i\Meq F_i$, for $i=1,2$, $(\Bsf_{F_1}, \Bsf_{F_2})$ is in standard form, $\Bsf_{F_1},\Bsf_{F_2}\in\MPZccc$, for some multiindices $\mathbf{m},\mathbf{n}$ and a poset \calP satisfying Assumption~\ref{ass:preorder}. 
The block matrices are given in such a way that when we consider $C^*(F_i)$ as \calP-algebras according to this block structure, this is exactly the same $X$-algebra structure as the original. 
And the move equivalences $E_i\Meq F_i$ induce \calP-equivariant isomorphisms between $C^*(E_i)\otimes\K$ and $C^*(F_i)\otimes\K$, for $i=1,2$. 

Therefore, we can without loss of generality assume that $E_1$ and $E_2$ are in this form.  
According to Section~\ref{sec:red-filtered-K-theory-K-web-GLP-and-SLP-equivalences} and Theorem~\ref{thm:mainBH}, there exists a \GLPEe $(U,V)$ from $\Bsf_{E_1}$ to $\Bsf_{E_2}$ such that $\FKR(U,V)=\varphi$.
Since $\varphi$ is positive on $K_0$ of the gauge simple subquotients, we get that $V\{i\}$ is the $1\times 1$ matrix $1$ for all $i\in\calP$ with $n_i=1$. 

From Corollary~\ref{cor:Pulelehua}, we see that we can get a new pair of graphs $(F_1',F_2')$ with $E_i\MCPeq F_i'$, for $i=1,2$, $\Bsf_{F_1'}, \Bsf_{F_2'}\in\MPZccc[\mathbf{m}'\times\mathbf{n}']$ such that $(\Bsf_{F_1'},\Bsf_{F_2'})$ is in standard form, and we have a \calP-equivariant isomorphism $\Psi_i$ from $C^*(E_i)\otimes\K$ to $C^*(F_i')\otimes\K$, for $i=1,2$, and a \GLPEe $(U',V')$ from $\Bsf_{F_1'}$ to $\Bsf_{F_2'}$ such that 
$\varphi'=\FKRs(U',V')$ with 
$U'\{i\}=1$ for all $i\in\calP$ with $m_i'=1$ and $V'\{i\}=1$ for all $i\in\calP$ with $n_i'=1$, where $\varphi'=\FKR(\calP;\Psi_2)\circ\FKR(\calP;\kappa_{C^*(E_2)})\circ\varphi\circ\FKR(\calP;\kappa_{C^*(E_1)})^{-1}\circ\FKR(\calP;\Psi_1)^{-1}$. 

Now we can use the Corollary~\ref{cor:GLtoSL} to see that $F_1'\MCeq F_2'$ and to find a \calP-equivariant isomorphism $\Phi$ from $C^*(F_1')\otimes\K$ to $C^*(F_2')\otimes\K$ such that $\FKR(\calP;\Phi)=\FKRs(U',V')$.

Combining all the above, we get a lifting of $\varphi$. Also we see that $E_1\MCPeq E_2$. 
Moreover, if $E_1$ and $E_2$ satisfy Condition~(K), then there are no cyclic components and thus we never used Move \PP. 
So if $E_1$ and $E_2$ satisfy Condition~(K), then in fact $E_1\MCeq E_2$. 
\end{proof}

\section{Applications and further results}\label{sec:addenda}

\subsection{Order on the reduced filtered \texorpdfstring{$K$}{K}-theory}\label{ordereasy}
Since the proof of the lifting result only uses the order on $K_0$ of the gauge simple subquotients, it follows that the order on the other groups is superfluous for unital graph \cas: An isomorphism between the reduced filtered $K$-theories of two unital graph \cas that is an order isomorphism on $K_0$ of all the gauge simple subquotients is automatically an isomorphism between the ordered, reduced filtered $K$-theories (since it is induced by an equivariant isomorphism). 

At first, this might seem surprising; but this really just tells us, that for unital graph \cas the order on all the subquotients that are not gauge simple is determined by the order on the gauge simple subquotients. This is somehow part of the assumption of being a unital graph \ca, and in the light of, \eg, \cite{MR3194162,MR3453358}, this should really not come as a big surprise.

This observation raises the question of whether there is an external classification result --- here one of course needs to state which universe one wants to work in, and the most natural question that comes to our minds in this respect is the following.

\begin{question}\label{question:phantom}
Let $X$ be a finite $T_0$-space. 
Assume that \A is a \ca over $X$ that is unital, separable, nuclear with $\A(\{x\})$ in the bootstrap class for every $x\in X$ such that for each $x\in X$, either $\A(\{x\})\otimes\K\cong\K$, $\A(\{x\})\otimes \K\cong C(S^{1})\otimes\K$, or $\A(\{x\})$ is a simple, purely infinite \ca. 
Assume, moreover, that there exists a graph $E$ with finitely many vertices such that $X$ is homeomorphic to $\Prime_\gamma(C^{*}(E))$ and that there exists a (reduced, filtered) $K$-theory order isomorphism from $\FKRplus(X;\A)$ to $\FKRplus(X;C^{*}(E))$, when we consider $C^{*}(E)$ as an $X$-algebra under the above homeomorphism. 

Is \A necessarily stably isomorphic to a unital graph \ca?
\end{question}
A \ca \A that answers this question in the negative, we could consider a \emph{phantom} unital graph \ca (up to stable isomorphism) as in \cite{MR3142030}. It is known from \cite{arXiv:1511.09463v1} that no such  phantoms exist when $E$ is a finite graph having Condition~(K) and with no sinks and sources, \emph{i.e.} when $C^*(E)$ is a Cuntz-Krieger algebra of real rank zero. More partial answers in the positive can be extracted from \cite{MR3056712}.

Arklint \cite{sara-paper}  has  completely described the range of the ordered, reduced filtered $K$-theory for unital graph \cas. These results explain explicitly how the order of the other components of $\FKRplus(X;-)$ can be determined from the order on $K_0$ of the gauge simple subquotients, and can be used to reformulate Question \ref{question:phantom} to replace the mention of $E$ by a completely algebraic description.

\subsection{Comparison and decidability}\label{sec:decidability}

In the following, we will need to rescale permutation matrices $P$ in $\GLZ[N]$ (where $N=|\calP|$) to block matrices $P_{\mathbf{n}}\in\MZ[\mathbf{n}\times\mathbf{n}]$ in the way that each $0$ or $1$ in the original matrix is replaced by a zero matrix or an identity matrix, respectively, of the appropriate size. We will only consider permutation matrices $P$ in $\operatorname{GL}(N,\Z)$ that implement an order automorphism of the poset $(\calP,\preceq)$. Note that $P_\mathbf{n}$ will always be a \GL-matrix, but it will usually not belong to $\GLPZ[\mathbf{n}]$. 

\begin{theorem}\label{thm:comparison}
Let $(E,F)$ be a pair of  graphs with finitely many vertices with $(\Bsf_E, \Bsf_F)$ in standard form, and assume that $\Bsf_E^{\bullet} , \Bsf_F^{\bullet}\in\MPplusZ$ are given such that there exists an integer $M\geq 3$ satisfying that $m_i=M$ whenever $m_i>1$. 
Then \ref{thm:main-1-item-1}--\ref{thm:main-1-item-3} of Theorem~\ref{thm:main-1} is equivalent to 
\begin{enumerate}[(1)]
\addtocounter{enumi}{3}
\item \label{thm:comparison-item-2}
There exists a permutation matrix $P\in\GLZ[|\calP|]$ that implements an order automorphism of the poset $(\calP,\preceq)$ and there exist $U\in \GLPZ[\mathbf{m}]$ and $V\in\GLPZ$ with $V\{i\} = 1$ whenever
$n_i=1$ so that 
$(P_{\mathbf{n}}\Bsf_E P_{\mathbf{n}}^{-1})^{\bullet}\in\MPplusZ$ and $U(P_{\mathbf{n}}\Bsf_EP_{\mathbf{n}}^{-1})^{\bullet} V = \Bsf_F^{\bullet}$.
\end{enumerate}
\end{theorem}
\begin{proof}
\ref{thm:comparison-item-2}$\Rightarrow$\ref{thm:main-1-item-3}: 
This implication follows since $U, V, P$ will induce an isomorphism between the ordered reduced filtered $K$-theories of $C^*(E)$ and $C^*(F)$ (see Section~\ref{sec:red-filtered-K-theory-K-web-GLP-and-SLP-equivalences}).

\ref{thm:main-1-item-3}$\Rightarrow$\ref{thm:comparison-item-2}: 
Let $X=\Prime_\gamma(C^{*}(E))$ and assume that there exists a homeomorphism between $\Prime_\gamma(C^*(E))$ and $\Prime_\gamma(C^*(F))$ such that 
$\FKRplus(X; C^{*} (E ) )$ is isomorphic to $\FKRplus(X; C^{*} ( F ) )$, where we view $C^{*} ( F )$ as an $X$-algebra via the given homeomorphism between $\Prime_\gamma(C^*(E))$ and $\Prime_\gamma(C^*(F))$. 
Under the standard identifications, this induces an automorphism of \calP that respects the type of the gauge simple subquotients --- thus we might up to a permutation assume that this is the identity automorphism. 
Moreover, under the standard identifications, this (reduced, filtered) order  isomorphism $\varphi$ corresponds to a $K$-web isomorphism as explained in Section~\ref{sec:red-filtered-K-theory-K-web-GLP-and-SLP-equivalences}. 
From Theorem~\ref{thm:mainBH}, it follows that there exists a \GLPEe $(U,V)$ from $\Bsf_E^\bullet$ to $\Bsf_F^\bullet$ inducing $\varphi$. Necessarily, we have that $V\{i\}=1$ whenever $n_i=1$ (because of positivity on $K_0$ of the gauge simple subquotients that are not purely infinite).
\end{proof}

As in \cite{MR3759003}, it is possible to obtain a version of this result where the \Pul\ equivalence must respect a fixed block structure,  and where $P$ (and hence $P_\mathbf{n}$) are identity matrices.

Although the  existence clause in \ref{thm:comparison-item-2} can be translated into the linear system $U\Bsf_{E}^\bullet =\Bsf_{F}^\bullet V'$, deciding whether or not there is a solution leading to invertible matrices using the naive approach of invoking determinants takes us into the territory of Hilbert's tenth problem. Nevertheless, as explained by Boyle and Steinberg (\cite{bs}), one may instead appeal to decidability results by Grunewald and Segal \cite{MR595206} to obtain:  

\begin{corollary}\label{cor:mor_dec}
Morita equivalence among unital graph $C^*$-algebras is decidable.
\end{corollary}
\begin{proof}
From the adjacency matrices of a pair of graphs $(E,F)$ with finitely many vertices it is possible to decide all order automorphisms of the component posets that preserve the temperatures and the difference in rank between $K_0$ and $K_1$ of the gauge simple subquotients --- and there are only finitely many such order automorphisms. 
So assume that we have fixed such an order automorhism. 
In this case, we have provided explicit algorithms to pass to a pair of graphs $(E',F')$ with $(\Bsf_{E'},\Bsf_{F'})$ 
in standard form without changing the Morita equivalence class of the \cas, so  we may assume without loss of generality that $(\Bsf_{E},\Bsf_{F})$ is already in standard form, and the given algorithms may easily be amended to arrange also that all blocks are of size $1$ or $M$ for some appropriately chosen $M$. We note that there are only finitely many possible $P$ to consider in Theorem \ref{thm:comparison}, and hence that we may assume without loss of generality that $P=I$. The argument  is completed by Corollary 3.3 of \cite{bs}.
\end{proof}

To prove that \stariso is also decidable is more complicated, and requires a variation of our concept of standard form to ensure that the $C^*$-algebras are not changed in the initial step of passing to a pair of identically shaped matrices. We will do so by appealing to a variation of the Move~\RR from a 
forthcoming paper by the first and third named authors; this is a refinement of a concept from \cite{MR3391894}. The component needed here is not deep and we present the full proof for completeness.

\begin{definition}[Move \RRplus:  Unital Reduction]\label{dfn:reduction-plus}
Let $E$ be a graph and let $w$ be a regular vertex such that $s_E^{-1}(w) \cap r_E^{-1} (w) = \emptyset$.  Let $E_{R+}$ be the graph defined by 
\begin{align*}
E_{R+}^0 &= ( E^0  \setminus \{ w \} ) \sqcup \{ \widetilde{w} \} \\
E_{R+}^1 &= \left( E^1 \setminus (  r_E^{-1} (w) \cup s_E^{-1} (w)  ) \right) \sqcup \setof{[ ef ]}{e \in r_E^{-1}(w) , f \in s_{E}^{-1}(w)} \sqcup \setof{\widetilde{f}}{f\in s_E^{-1}(w)}
\end{align*}
the source and range maps of $E_{R+}$ extend those of $E$ and satisfy $s_{E_{R+}} ( [ef] ) = s_E(e)$, $s_{E_{R+}}( \widetilde{f} ) = \widetilde{w}$, $r_{E_{R+}} ( [ef] ) = r_{E}(f)$ and $r_{E_{R+}} ( \widetilde{f} ) = r_E(f)$.
\end{definition}

\begin{theorem}\label{thm:reduction-plus}
Let $E$ be a graph and let $w$ be a regular vertex such that $s_E^{-1}(w) \cap r_E^{-1} (w) = \emptyset$ and let $E_{R+}$ be the graph in Definition~\ref{dfn:reduction-plus}.  Then there exists an isomorphism $\psi \colon C^*(E_{R+}) \to C^*(E)$.
\end{theorem}

\begin{proof}
Set $P_v = p_v$ for $v\in E^0\setminus\{w\}$, $P_{\widetilde{w}} = p_w$, $S_e = s_e$ for $e\in E^1\setminus (r_E^{-1}(w)\cup s_E^{-1}(w))$, $S_{[ef]} = s_e s_f$ for $e\in r_E^{-1}(w)$ and $f\in s_E^{-1}(w)$, and $S_{ \widetilde{f} } = s_f$ for $f\in  s_E^{-1}(w)$.  A computation shows that $\{ P_v , P_{\widetilde{w}}, S_e, S_{[ef]}, S_{\widetilde{f}} \}$ is a Cuntz-Krieger $E_{R+}$-family in $C^*(E)$.  Therefore, there exists a $*$-homomorphism $\psi \colon C^*(E_{R+}) \to C^*(E)$ such that $\psi ( p_v ) = P_v$, $\psi ( p_{\widetilde{w}} ) = P_{\widetilde{w}}$, $\psi ( s_e ) = S_e$, $\psi ( s_{ [ef] } ) = S_{[ef]}$, and $\psi ( s_{\widetilde{f}} ) = S_{ \widetilde{f}}$.  Note that $\psi$ sends each vertex project to a nonzero projection and if $e_1 e_2 \ldots e_n$ is a vertex-simple cycle in $E_{R+}$ with no exits, then $\psi ( s_{e_1 e_2\cdots e_n})$ is a unitary in a corner of $C^*(E)$ with full spectrum.  Hence, by \cite[Theorem~1.2]{MR1914564}, $\psi$ is injective.

We now show that $\psi$ is surjective.  Note that the only generators of $C^*(E)$ that are not obviously in the image of $\psi$ are $s_e$ with $r_E(e) = w$.  Let $e \in r_E^{-1}(w)$.  Then 
\[
s_e = s_e \sum_{ f \in s_E^{-1} (w) } s_f s_f^* = \sum_{ f \in s_{E}^{-1} (w) } S_{[ef]} S_{\widetilde{f}}^* = \psi \left( \sum_{ f \in s_{E}^{-1} (w) } s_{[ef]} s_{\widetilde{f}}^* \right).
\]
Therefore, the image of $\psi$ contains every generator of $C^*(E)$ which implies that $\psi$ is surjective.
\end{proof}

Let $E$ be a graph with finitely many vertices and a single regular source $v_0$. Then we let $\widetilde{E}$ denote the graph we obtain from $E$ by removing the vertex $v_0$ and all edges emanating from it and we let $\underline{x}_E$ denote the nonzero vector with $|E^0|-1$ entries counting the number of edges from $v_0$ to every other vertex of $E$. If $E$ is a graph with finitely many vertices and with no regular sources, then we let $\widetilde{E}$ denote $E$ itself and we let $\underline{x}_E$ denote the zero vector with $|E^0|$ entries. 
In the remainder of this section, we denote by  $\underline{1}$ a vector of appropriate size having all entries equal to $1$. 

\begin{theorem}\label{thm:comparison-unital}
Let $(E,F)$ be a pair of  graphs with finitely many vertices and at most one regular source each such that $\widetilde{E}$ and $\widetilde{F}$ have no regular sources and $\Bsf_{\widetilde{E}},\Bsf_{\widetilde{F}}\in\MPZccc$. 
Assume, moreover, that there exists an integer $M\geq 2$ such that whenever $i$ corresponds to a strongly connected component that is not cyclic (both for $\widetilde{E}$ and for $\widetilde{F}$) then $m_i=M$ and the Smith normal forms of $\Bsf_{\widetilde{E}}^\bullet\{i\}$ and $\Bsf_{\widetilde{F}}^\bullet\{i\}$ have at least one $1$ each. 
Then \ref{cor:main-5-item-2}--\ref{cor:main-5-item-3} of Corollary~\ref{cor:main-5} are equivalent to
\begin{enumerate}[(1)]\addtocounter{enumi}{2}
\item
There exists a permutation matrix $P\in\GLZ[|\calP|]$ that implements an order automorphism of the poset $(\calP,\preceq)$ and there exist $U\in
\GLPZ[\mathbf{m}]$ and $V\in\GLPZ$ with $V\{i\} = 1$ whenever
$n_i=1$ so that 
$(P_{\mathbf{n}}\Bsf_{\widetilde{E}} P_{\mathbf{n}}^{-1})\in\MPZccc$,
$$U(P_{\mathbf{n}}\Bsf_{\widetilde{E}}P_{\mathbf{n}}^{-1})^{\bullet} V = \Bsf_{\widetilde{F}}^{\bullet}$$
and, moreover,
$$V^\mathsf{T}P_{\mathbf n}(\underline{1}+\underline{x}_E)-(\underline{1}+\underline{x}_F)\in\operatorname{im}(\Bsf_{\widetilde{F}}^\bullet)^\mathsf{T}.$$
\end{enumerate}
\end{theorem}

\begin{proof}
The ordered, filtered $K$-theory of $C^*(E)$ and $C^*(\widetilde{E})$ agree, and the class of the unit of $C^*(E)$ is represented in  $\operatorname{coker}(\Bsf_{\widetilde{E}}^\bullet)^\mathsf{T}$ as $\underline{1}+\underline{x}_E$. Since the parallel statement holds for $F$ the proof of Theorem~\ref{thm:comparison} applies also here. 
\end{proof}

\begin{corollary}
$*$-isomorphism among unital graph $C^*$-algebras is decidable.
\end{corollary}
\begin{proof}
Let $E$ and $F$ be graphs with finitely many vertices. 
First, using Proposition~\ref{prop:structure-2}\ref{prop:structure-2-circ}, we can find graphs $E_1$ and $F_1$ with finitely many vertices such that every infinite emitter emits infinitely many edges to any vertex it emits any edge to, every transition state has exactly one edge going out and $C^*(E)\cong C^*(E_1)$ and $C^*(F)\cong C^*(F_1)$ (the proof of Lemma~3.17(iv) in \cite{MR3759003} gives a concrete procedure of how to do this). 
Now, we can use Move~\RRplus to successively remove every transition state that is not a (regular) source and we can make all cyclic components singletons (at the cost of possibly introducing more sources). 
Call the resulting graphs $E_2$ and $F_2$, respectively. 
Now, we can collect all the regular sources (if any) in a single regular source using Move~\OO in reverse; call the resulting graphs $E_3$ and $F_3$, respectively. 
This ensures us that $E_3$ and $F_3$ have at most one regular source, that $\widetilde{E_3}$ and $\widetilde{F_3}$ have no sources and that we can write the adjacency matrices of $\widetilde{E_3}$ and $\widetilde{F_3}$ as block matrices (if we reorder the vertices). 
Choose an $M\in\mathbb{N}$, such that $M$ is one greater than the size of the largest diagonal block of the adjacency matrices of $\widetilde{E_3}$ and $\widetilde{F_3}$. 
Now, we can use Move~\OO successively on vertices in strongly connected components that are not cyclic to get exactly $M$ vertices in all these components.
Call the resulting graphs $E_4$ and $F_4$, respectively. 
From the adjacency matrices, it is also clearly decidable, whether the component posets are isomorphic in such a way that it preserves the temperatures --- in fact, there are only finitely many permutations of the component poset to consider. 
Also we can easily decide from the adjacency matrices whether the difference in rank of $K_0$ and $K_1$-groups match up for each simple subquotient under such an isomorphism. 
So we may assume that this is the case. 
Therefore, we can assume that we have chosen an order on the vertices of $E_4$ and $F_4$ such that $\Bsf_{\widetilde{E_4}},\Bsf_{\widetilde{F_4}}\in \MPZccc$ for some poset $\calP$ satisfying the conditions in Assumption~\ref{ass:preorder}. 
Note that $C^*(E)\cong C^*(E_i)$ and $C^*(F)\cong C^*(F_i)$ for $i=1,2,3,4$. 
The previous theorem thus gives us a characterization of when $C^*(E)\cong C^*(F)$. 
Since there are only finitely many permutation matrices $P$ to consider, we may assume that $P=I$. 
Boyle and Steinberg prove in Theorem 3.5 of \cite{bs} that existence of $U$ and $V$ intertwining $\Bsf_{\widetilde{E_4}}^\bullet$ and $\Bsf_{\widetilde{F_4}}^\bullet$ as well as satisfying the added condition 
 \[
 V^\mathsf{T}x-y\in\operatorname{im}(\Bsf_{\widetilde{F_4}}^\bullet)^\mathsf{T}.
 \]
remains decidable. Thus choosing $x=\underline{1}+\underline{x}_E$ and $y=\underline{1}+\underline{x}_F$ establishes decidability.
\end{proof}

\subsection{The type I/postliminal case}\label{sec:postliminal}

In \cite{MR3759003} we analyzed the class of quantum lens spaces $C(L_q(r;(m_1,\dots, m_n)))$ defined and studied in  \cite{MR2015735}, noting among other things that when $r=3s$ any pair of such  $C^*$-algebras would be isomorphic precisely when they were Morita equivalent. We may now prove that this applies to all graphs with matrices in standard form defining postliminal/type I $C^*$-algebras:

\begin{proposition}
Assume that $(E,F)$ is a pair of graphs with $(\Bsf_E, \Bsf_F)$ in standard form defining $C^*$-algebras $C^*(E)$ and $C^*(F)$ that are postliminal/type I. Then 
\[
C^*(E)\cong C^*(F)\Longleftrightarrow C^*(E)\otimes\K \cong C^*(F)\otimes\K
\]
\end{proposition}
\begin{proof}
One direction is clear.  
Now, suppose $C^*(E)$ and $C^*(F)$ are stably isomorphic.  By Theorem~\ref{thm:main-1}, there exists a homeomorphism between $\Prime_\gamma(C^*(E))$ and $\Prime_\gamma(C^*(F))$ such that 
$\FKRplus(X; C^{*} (E ) ) \cong \FKRplus(X; C^{*} ( F ) )$, where we view $C^{*} ( F )$ as an $X$-algebra via the given homeomorphism between $X = \Prime_\gamma(C^*(E))$ and $\Prime_\gamma(C^*(F))$. 
Under the standard identifications, this induces an automorphism of \calP that respects the type of the gauge simple subquotients --- thus we might up to a permutation assume that this is the identity automorphism.  By assumption, $\Bsf_E,\Bsf_F\in\MPZccc$.  By Theorem~\ref{thm:mainBH}, there exist $U\in\GLPZ[\mathbf m]$ and $V\in\GLPZ[\mathbf n]$  such that $U\Bsf_E^\bullet V = \Bsf_F^\bullet$.  $V^{\mathsf T}$ may not send the class of the unit to the class of the unit.  

We now show that $V$ can be adjusted so that the class of the unit is sent to the class of the unit.
Let $v_1, v_2, \ldots, v_k$ be the vertices in $F$ that correspond to components that have no immediate
predecessors (\ie, $v\geq v_i\Rightarrow v=v_i$ in our case). 
We can assume that these correspond to the first $k$ vertices. %, and we let $\underline{1}_k$ denote $(\underbrace{1, 1, \ldots, 1}_{\text{$k$ times}})^\mathsf{T}$. %(and recall that $\underline{1}_k$ denotes $(\underbrace{1, 1, \ldots, 1}_{\text{$|\mathbf{n}|$ times}})^\mathsf{T}$. 

Let $\mathsf{C}$ be a given $k\times (|\mathbf{n}|-k)$-matrix and let $\mathbf{w}=( w_v )_{v \in F^0 \setminus \{v_1, v_2, \dots, v_k \} }^\mathsf{T}$ be a given vector. 
Note that 
\[
\mathsf{B}_F^\bullet 
\begin{pmatrix}
I & \mathsf{C} \\
0  &  I
\end{pmatrix} 
= \mathsf{B}_F^\bullet,
\]  
and 
\[
\begin{pmatrix}
I & \mathsf{C} \\
0  &  I
\end{pmatrix}^\mathsf{T}
\begin{pmatrix}
\underline{1} \\
\mathbf{w} 
\end{pmatrix} = 
\begin{pmatrix}
\underline{1} \\
\mathsf{C}^\mathsf{T}\underline{1} + \mathbf{w}
\end{pmatrix}.
\]
Define $\mathsf{C}_\mathbf{w}$ as follows:  
\begin{itemize}
\item The $(v_1, v)$ entry of $\mathsf{C}_\mathbf{w}$ is zero if $v_1 \not\geq v$, and is $1 - w_v$ if $v_1 \geq v$.

\item The $(v_2, v)$ entry of $\mathsf{C}_\mathbf{w}$ is zero if $v_2 \not\geq v$ or $v_1 \geq v$,  and is $1 - w_v$ otherwise.

\item The $(v_3, v )$ entry of $\mathsf{C}_\mathbf{w}$ is zero if $v_3 \not\geq v$ or $v_2 \geq v$ or $v_1 \geq v$, and is $1 - w_v$ otherwise.

\item In general, for $1 \leq l \leq k$, the $(v_l, v )$ entry of $\mathsf{C}_\mathbf{w}$ is zero if $v_l \not\geq v$ or $v_i \geq v$ for some $1 \leq i \leq l-1$, and $1 - w_v$ otherwise.
\end{itemize}
By construction, 
$
\left(\begin{smallmatrix}
I & \mathsf{C}_\mathbf{w} \\
0  &  I
\end{smallmatrix}\right)
$
is an element of $\mathrm{GL}_\mathcal{P}( \mathbf{n}, \Z )$ and  
\[
\begin{pmatrix}
I & \mathsf{C}_\mathbf{w} \\
0  &  I
\end{pmatrix}^\mathsf{T} 
\begin{pmatrix}
\underline{1} \\
\mathbf{w}
\end{pmatrix}
=\begin{pmatrix}
\underline{1} \\
\mathsf{C}_\mathbf{w}^\mathsf{T} \underline{1} + \mathbf{w}
\end{pmatrix} 
=\underline{1}
\]

Note that $V^\mathsf{T} \underline{1} = \left(\begin{smallmatrix}
\underline{1}\\
\mathbf{w} 
\end{smallmatrix}\right)$
for some $\mathbf{w}$.  Therefore, we choose $\mathsf{C}_{\mathbf{w}}$ as above, and set $V' =\left( \begin{smallmatrix}
I & \mathsf{C}_\mathbf{w} \\
0  &  I
\end{smallmatrix}\right)$.   
Therefore, by construction, $U \Bsf_E^\bullet VV' = \Bsf_F^\bullet V' = \Bsf_F^\bullet$ and $(VV')^\mathsf{T} \underline{1} = \underline{1}$.  The argument is completed by Theorem \ref{thm:comparison-unital}.
\end{proof}

The observation above applies directly to quantum lens spaces, allowing us to complete the investigation started in \cite{MR3759003} of how large the dimension $n$ of the quantum lens space must be compared to the leading parameter $r$ to allow for variation of the secondary parameters in $\mm=(m_1,\dots, m_n)$ to lead to more than one quantum lens space $C(L_q(r;\mm))$.
%Letting again $\oneone$ denote a vector $(1,\dots,1)$ of appropriate size, 
Setting 
\[
\varphi(r)=\min\{n\in \N\mid \exists \mm\in \N^n: (m_i,r)=1, C^*(L_{2n-1}^{(r;\mm)})\not \cong
C^*(L_{2n-1}^{(r;\oneone)})\},
\]
we proved in \cite{MR3759003} that $\varphi(3s)=4$ for all $s\in \N$.

\begin{theorem}
\[
{\varphi}(r)=\min\setof{2n}{ \exists a,b\in\N: 2n>a>2, ab=r},
\]
for any $r>2$.
\end{theorem}

In words, $\varphi(r)$ is the smallest even number strictly larger than the smallest divisor of $r$ that is not  $2$.

\begin{proof}
The proof in \cite{MR2015735} that these objects are unital graph algebras gives a  matrix description which may be organized in canonical form in a unique way. Hence given $r$ and $n$, the question of whether or not there exists $\mm=(m_1,\dots, m_n)$ so that $C^*(L_{2n-1}^{(r;\mm)})\not \cong
C^*(L_{2n-1}^{(r;\oneone)})$ may be translated completely into integer matrices, where it was solved in \cite[Theorem~5.1]{arXiv:1701.04003v1}. 
\end{proof}

Although elementary in nature, the argument in \cite{arXiv:1701.04003v1} is far from trivial, drawing on properties of the Bernoulli numbers.

\subsection{Leavitt path algebras}\label{sec:abrams-tomforde}

It was conjectured by Gene Abrams and Mark Tomforde in
\cite{MR2775826} that if the Leavitt path algebras $L_{\C} (E)$ and
$L_{\C} (F)$ are Morita equivalent, respectively isomorphic as rings, then $C^{*} (E)$ and $C^{*}(F)$
are strongly Morita equivalent, respectively $\sp *$-isomorphic (see \cite{MR2417402} for the
definition of $L_{\C} (E)$).  Since the formulation of the conjectures in 2011, it has been confirmed in many special cases.

Our classification result confirms the conjectures in the general unital case:

\begin{theorem}
The Abrams-Tomforde conjectures hold true for graphs with finitely many vertices.
\end{theorem}
\begin{proof}
As we show in \cite[Corollary 5.8]{arXiv:1610.02232v1}, 
 when $L_\C(E)$ and $L_\C (F)$ are Morita equivalent one may conclude that the ordered filtered $K$-theories of $C^*(E)$ and $C^*(F)$ are isomorphic, and when    $L_\C (E)$ and $L_\C (F)$ are isomorphic as rings one may conclude that the ordered filtered $K$-theories are isomorphic with an isomorphism preserving  the class of the units. That $C^*(E)$ and $C^*(F)$ are then, respectively, stably isomorphic or isomorphic follows from our main results Theorems \ref{thm:main-1} and \ref{thm:main-4}.
\end{proof}

Whereas move equivalent graphs also yield Morita equivalent Leavitt path algebras over any ring or field, it is unknown whether $L_{\C} (E)$ is Morita equivalent to $L_{\C} (E')$ if $E'$ is obtained from $E$ by a move of the form \CC or \PP. It was proved in \cite{MR3517565} that $L_{\Z} (E)$ may fail to be $\sp *$-isomorphic to $L_{\Z} (E')$ when $E'$ is obtained using Move \CC\ --- even in cases where it can be arranged that also $C^*(E)\cong C^*(E')$ --- and similar arguments based on  
\cite{arXiv:1601.00777v1} shows the same to hold for Move \PP.

\section*{Acknowledgements}

This work was partially supported by the Danish National Research Foundation through the Centre for Symmetry and Deformation (DNRF92), by VILLUM FONDEN through the network for Experimental Mathematics in Number Theory, Operator Algebras, and Topology, by a grant from the Simons Foundation (\# 279369 to Efren Ruiz), and by the Danish Council for Independent Research | Natural Sciences.

The third and fourth named authors would like to thank the School
of Mathematics and Applied Statistics at the University of Wollongong
for hospitality during their visit there, and the first and second named authors likewise thank the Department of
Mathematics, University of Hawaii, Hilo. The initial work was carried out at these two long-term visits, and
it was completed while
all four authors were attending the research program
\emph{Classification of operator algebras: complexity, rigidity, and                                                              
 dynamics} at the Mittag-Leffler Institute, January--April 2016. We
thank the institute and its staff for the excellent work conditions
provided.

The authors would also like to thank Mike Boyle and James Gabe for many fruitful discussions.

%\bibliographystyle{../../../bibtexGR/hamsalpha-nomr}
%
%\bibliography{../../../bibtexGR/mrreferences,%
%../../../bibtexGR/preprints,../../../bibtexGR/theses,%
%../../../bibtexGR/arxiv}%,additions}

\providecommand{\bysame}{\leavevmode\hbox to3em{\hrulefill}\thinspace}
\providecommand{\MR}{\relax\ifhmode\unskip\space\fi MR }
\providecommand{\MRhref}[2]{%
  \href{http://www.ams.org/mathscinet-getitem?mr=#1}{#2}
}
\providecommand{\href}[2]{#2}

\end{document}